\documentclass{article}
\usepackage{arxiv}

\usepackage{hyperref}       % hyperlinks
\usepackage{url}            % simple URL typesetting
\usepackage{booktabs}       % professional-quality tables
\usepackage{amsfonts}       % blackboard math symbols
\usepackage{microtype}      % microtypography
\usepackage{color}
\usepackage{amsthm}
\usepackage{indentfirst}
\usepackage{graphicx}
\usepackage{float}
\usepackage{epstopdf}
\usepackage{fourier}
\usepackage{bm}
\usepackage{bbm} % for Indicator function
\usepackage{mathtools}
\usepackage{latexsym,enumerate}
\usepackage{amsmath}
\usepackage{mlmath}
\usepackage{soul}
\usepackage{adjustbox}
\usepackage{multirow}
\usepackage{subfigure}
\usepackage{xcolor}

\def\bm{\boldsymbol}

\newcommand{\comment}[1]{}

\newcommand{\BEA}{\begin{eqnarray}}
\newcommand{\EEA}{\end{eqnarray}}

\newcommand{\width}{m}
\newcommand{\E}{\mathbb{E}}

\DeclareMathOperator*{\argmin}{arg\,min}

\newcommand{\bmtheta}{{\bmalpha}}
\newcommand{\bmalpha}{{\bm{\alpha}}}

\newtheorem{lem}{Lemma}[section]
\newtheorem{theo}{Theorem}[section]

\newtheorem{defi}{Definition}[section]

\newtheorem{assu}{Assumption}[section]

\newcommand{\rev}[1]{{\color{black}#1}}
\newcommand{\revv}[1]{{\color{black}#1}}

\title{Solving PDEs on Unknown Manifolds with Machine Learning}

\author{
  Senwei Liang \\
  Department of Mathematics, Purdue University, IN 47907, USA \\
  \texttt{liang339@purdue.edu} \\
  \And
  Shixiao W. Jiang \\ 
  Institute of Mathematical Sciences, ShanghaiTech University, Shanghai, 201210, China  \\
    \texttt{jiangshx@shanghaitech.edu.cn} \\
  \And 
  John Harlim \\
  Department of Mathematics, Department of Meteorology and Atmospheric Science, \\ Institute for Computational and Data Sciences \\
  The Pennsylvania State University, University Park, PA 16802, USA\\
  \texttt{jharlim@psu.edu} \\
  \And
  Haizhao Yang \\
  Department of Mathematics, University of Maryland, College Park, MD 20742, USA\\
  \texttt{hzyang@umd.edu} \\
}

\begin{document}

\maketitle

\begin{abstract}
This paper proposes a mesh-free computational framework and machine learning theory for solving elliptic PDEs on unknown manifolds, identified with point clouds, based on diffusion maps (DM) and deep learning. The PDE solver is formulated as a supervised learning task to solve a least-squares regression problem that imposes an algebraic equation approximating a PDE (and boundary conditions if applicable). This algebraic equation involves a graph-Laplacian type matrix obtained via DM asymptotic expansion, which is a consistent estimator of second-order elliptic differential operators. The resulting numerical method is to solve a highly non-convex empirical risk minimization problem subjected to a solution from a hypothesis space of neural networks~\rev{(NNs)}. In a well-posed elliptic PDE setting, when the hypothesis space consists of neural networks with either infinite width or depth, we show that the global minimizer of the empirical loss function is a consistent solution in the limit of large training data. When the hypothesis space is a two-layer neural network, we show that for a sufficiently large width, gradient descent can identify a global minimizer of the empirical loss function. {Supporting numerical examples demonstrate the convergence of the solutions, ranging from simple manifolds with low and high co-dimensions, to rough surfaces with and without boundaries. We also show that the proposed NN solver can robustly generalize the PDE solution on new data points with generalization errors that are almost identical to the training errors, superseding a Nystr\"om-based interpolation method.}   
\end{abstract}


\section{Introduction}

Solving high-dimensional PDEs on unknown manifolds is a challenging computational problem that commands a wide variety of applications. {In physics and biology, such a problem arises in modeling of granular flow \cite{rauter2018finite}, liquid crystal \cite{virga2018variational}, biomembranes \cite{elliott2010modeling}. In computer graphics \cite{bertalmio2001variational}, PDEs on surfaces have been used to restore damaged patterns on a surface \cite{macdonald2010implicit}, brain imaging \cite{memoli2004implicit}, among other applications.} By unknown manifolds, we refer to the situation where the parameterization of the domain is unknown. The main computational challenge arising from this constraint is on the approximation of the differential operator using the available sample data (point clouds) that are assumed to lie on (or close to) a smooth manifold. Among many available methods proposed for PDEs on surfaces embedded in $\mathbb{R}^3$, they typically parameterize the surface and subsequently use it to approximate the tangential derivatives along surfaces. For example, the finite element method represents surfaces \cite{dziuk2013finite,camacho,bonito2016high} using triangular meshes. Thus its accuracy relies on the quality of the generated meshes which may be poor if the given point cloud data are randomly distributed. Another class of approach is to estimate the embedding function of the surface using e.g., level set representation \cite{bertalmio2001variational} or closest point representation \cite{ruuth2008simple,chu2018volumetric,martin2020equivalent}, and subsequently solve the embedded PDE on the ambient space. The key issue with this class of approaches is that since the embedded PDE is at least one dimension higher than the dimension of the two-dimensional surface (i.e., co-dimension higher than one), the computational cost may not be feasible if the manifold is embedded in high-dimensional ambient space. Another class of approaches is the mesh-free radial basis function (RBF) method \cite{piret2012orthogonal,fuselier2013high} for solving PDEs on surfaces. This approach, however, may not be robust in high dimensional and rough surface problems  as pointed out in \cite{fuselier2013high,Shankar2014RBFFD}, and the convergence near the boundary can be problematic.

% unsupervised learning algorithm to solve PDE (diffusion maps)
Motivated by \cite{li2016convergent}, an unsupervised learning method called the \emph{Diffusion Map} (DM) algorithm \cite{cl:06} was proposed to directly solve the second-order elliptic PDE on point clouds data that lie on the manifolds \cite{gh:18}. The proposed DM-based solver has been extended to elliptic problems with various types of boundary conditions that typically arise in applications \cite{jiang2023ghost}, such as non-homogeneous Dirichlet, Neumann, and Robin types and to time-dependent advection-diffusion PDEs \cite{yan2021kernel}. The main advantage of this approach is to avoid the tedious parameterization of sub-manifolds in a high-dimensional space. However, the estimated solution is represented by a discrete vector whose components approximate the function values on the available point clouds, analogous to standard finite-difference methods. With such a representation, one will need an interpolation method to find the solutions on new data points, which is a nontrivial task when the domain is an unknown manifold. This issue is particularly relevant if the available data points come sequentially. Another problem with the DM-based solver is that the size of the matrix approximating the differential operator increases as a function of the data size, which creates a computational bottleneck when the PDE problem {involves solving a singular linear system with a pseudo-inversion or an eigenvalue decomposition. }

% Neural network-based solver
One way to overcome these computational issues is to solve the PDE in a supervised learning framework using neural networks (NNs). On an Euclidean domain, where extensive research has been conducted \cite{E2017,Han2018,Khoo2017SolvingPP,Sirignano2018,Berg2018,Zang2019,Li2019,Beck2019,RAISSI2019686,Fang2020}, NN-based PDE solvers reformulate a PDE problem as a regression problem. Subsequently, the PDE solution is approximated by a class of NN functions. NNs have good approximation properties \cite{Barron1993,Weinan2019,E2019_2,Montanelli2019,SIEGEL2020313,daubechies2019nonlinear,Montanelli2019_2,Hutzenthaler2019,hutzenthaler2020,Shen4,Shen6,devore2021neural} that enable application of these PDE solvers to high-dimensional problems. With the advanced computational tools (e.g., TensorFlow and Pytorch) and the built-in optimization algorithms therein, developing mathematical software with parallel computing using NNs is much simpler than conventional numerical techniques. \rev{In fact, NN has been proposed to solve linear problems \cite{Cichocki,luz2020learning,YiqiMichael}.}

% contribution of the paper
Building upon this encouraging result, we propose to solve PDEs on unknown manifolds by embedding the DM algorithm in NN-based PDE solvers. In particular, the DM algorithm is employed to approximate the second-order elliptic differential operator defined on the manifolds. Subsequently, a least-squares regression problem is formulated by imposing an algebraic equation that involves a graph-Laplacian type matrix, obtained by the discretization of the DM asymptotic expansion on the available point cloud training data. Numerically, we solve the resulting empirical loss function by finding the solution from a hypothesis space of neural-network type functions (e.g., with feedforward NNs, we consider the compositions of power of ReLU
or polynomial-sine
activation functions). \revv{The proposed NN approach provides a more feasible approach to solve large linear systems associated with manifold PDEs compared to performing (a stable) pseudo-inversion and obtain a continuous function solution instead of a discrete solution at training points. Such feasibility could be achieved through mini-batch training as we shall demonstrate in Section~\ref{sec:2dtorus}, which circumvents the computational and memory issues of using the full linear system at once. Similar mini-batch training strategies have also been used in~\cite{YiqiMichael}. } Theoretically, we study the approximation and optimization aspects of the error analysis, induced by the training procedure that minimizes the empirical risk defined on available point cloud data. Under appropriate regularity assumptions, we show that when the hypothesis space has either an infinite width or depth, the global minimizer of the empirical loss function is a consistent solution in the limit of large training data. The corresponding error bound gives the relations between the desired accuracy and the required \rev{size} of training data and width (or depth) of the network. Furthermore, when the hypothesis space is a two-layer NN, we show that for a sufficiently large width, the gradient descent (GD) method can identify a global minimizer of the empirical loss function. {Numerically, we verify the proposed methods on several test problems on simple 2D and 3D manifolds of various co-dimensions and rough unknown surfaces with and without boundaries. In a set of instructive examples, we compare the accuracy of the generalization of the NN solution on new data points to the interpolation solution attained by applying the classical Nystr\"om scheme on a linear combination of the estimated Laplace-Beltrami eigenfunctions that spans the PDE solution. In rough surface examples, we will verify the accuracy of the solutions by comparing them against the Finite Element Method solutions.}

%Our goal is to verify the convergence of the scheme in the limit of large training data.

% organization of the paper
%The paper will be organized as follows. First, we give a brief review of the DM-based PDE solver on closed manifolds and an overview of the ghost point diffusion maps (GPDM) for manifolds with boundaries. The proposed PDE solver is then introduced with a theoretical foundation and numerical performance. %For the reader's convenience, we present an algorithmic perspective for GPDM in  SI Appendix~\ref{appendix_a}. We also include the proof for the optimization aspect of the algorithm in SI Appendix~\ref{appendix_c} and report all parameters used to generate the numerical results in SI Appendix~\ref{appendix_d}.}
% organization of the paper

The paper will be organized as follows. In Section~\ref{DMsolver}, we give a brief review on the DM (and variable bandwidth DM) based PDE solver on closed manifolds and an overview of the ghost point diffusion maps (GPDM) for manifolds with boundaries. The proposed PDE solver is introduced in Section~\ref{NNsolver}. The theoretical foundation of the proposed method is presented in Section~\ref{theory}. The numerical performance of the proposed method is illustrated in Section \ref{sec:num}. We conclude the paper with a summary in Section \ref{sec:con}. For reader's convenience, we present an algorithmic perspective for GPDM in Appendix~\ref{appendix_a}. We also include the longer proof for the optimization aspect of the algorithm in Appendix~\ref{appendix_c} and report all parameters used to generate the numerical results in Appendix~\ref{appendix_d}.

\section{DM-based PDE Solver on Unknown Manifolds}
\label{DMsolver}

To illustrate the main idea, let us discuss an elliptic problem defined on a $d$-dimensional closed sub-manifold $M\subseteq [0,1]^n\subseteq  \mathbb{R}^n$. {We will provide a brief discussion for the problem defined on compact manifolds with boundaries to end this section.} Let $u:M\to\mathbb{R}$ be a solution of the elliptic PDE,
\BEA
(-a(\mathbf{x}) + \mathcal{L}) u(\mathbf{x}) := -a(\mathbf{x})u(\mathbf{x}) + \mbox{div}_{g}\big(\kappa (\mathbf{x})\nabla _{g}u(\mathbf{x})\big) = f(\mathbf{x}), \quad \mathbf{x}\in M \label{ellipticPDE}
%\partial_\nu u(x) &=& g(x), \quad x\in \partial M.
\EEA
Here, we have used the notations $\mbox{div}_{g}$ and $\nabla _{g}$ for the
divergence and gradient operators, respectively, defined with respect to the
Riemannian metric $g$ inherited by $M$ from the ambient space $\mathbb{R}^n$. %We also use the conventional notations $\partial M$ and $M^o$ to denote the boundary and interior sets of $M$, respectively, such that $M=M^o\cup \partial M$ and $M^o \cap \partial M=\emptyset$.
The real-valued functions $a$ and $\kappa$ are strictly positive such that $(-a+\mathcal{L})$ is strictly negative definite. The problem is assumed to be well-posed for $f\in C^{1,\alpha}(M)$, for $0<\alpha <1/2$. For $a\in C^{1,\alpha}(M)$ and $\kappa\in C^{3,\alpha}(M)$, a unique classical solution $u\in C^{3,\alpha}(M)$ is guaranteed. Here, we raise the regularity by one-order of derivative (compared to reported results in the literature \cite{han2011elliptic,gilbarg2015elliptic}) for the following reason.

The key idea of the DM-based PDE solver rests on the
following asymptotic expansion \cite{cl:06},
\begin{align}\begin{split}
G_{\epsilon }u(\mathbf{x}):=\epsilon ^{-d/2}\int_{M}h\Big(\frac{\|\mathbf{x}-\mathbf{y}\|^{2}}{\epsilon }
\Big)u(\mathbf{y})dV(\mathbf{y})= u(\mathbf{x})+ \epsilon (\omega(\mathbf{x}) u(\mathbf{x}) + \Delta_g u(\mathbf{x}))+ \mathcal{O}(\epsilon^2),
\label{integralop1}
\end{split}\end{align}
where the second equality is valid for any $u\in C^3(M)$ and any $\mathbf{x}\in M$. \rev{To employ this asymptotic expansion in our approximation, we raise the regularity of the classical solution to be $C^3$ instead of the usual regularity assumption where $u\in C^2$ for continuous $f$.} Here, the function $h:[0,\infty )\rightarrow ( 0,\infty )$ is defined as $h(s)=\frac{e^{-s/4}}{(4\pi)^{d/2}}$ such that, effectively for a fixed bandwidth parameter $\epsilon>0$, $G_\epsilon$ is a local integral operator. In \eqref{integralop1}, $V$ denotes the volume form inherited by the manifold from the ambient space, the term $\omega$ depends on the geometry, \rev{ $\Delta_g = \text{div}_g \circ \nabla_g$ denotes the negative-definite Laplace-Beltrami operator,} $\|\cdot \|$ denotes the standard Euclidean norm for vectors in $\mathbb{R}^{n}$ and we will use the same notation for arbitrary finite-dimensional vector space. {Based on the asymptotic expansion in \eqref{integralop1}, one can approximate the differential operator $\mathcal{L}$ as follows,
\begin{align}\begin{split}
\mathcal{L}u(\mathbf{x}) =\frac{\sqrt{\kappa(\mathbf{x})}}{\epsilon}\big( G_\epsilon(u(\mathbf{x})\sqrt{\kappa(\mathbf{x})})-u(\mathbf{x})G_\epsilon \sqrt{\kappa(\mathbf{x})}\big)+ \mathcal{O}(\epsilon)
:=L_{\epsilon}u(\mathbf{x}) + \mathcal{O}(\epsilon).\label{integralapprox}
\end{split}\end{align}
}

In our setup, we assume that we are given a set of point cloud data $X:=\{ \mathbf{x}_i\in M\}_{i=1,\ldots, N}$,
independent and identically distributed (i.i.d.) according to $\pi$, with an empirical measure defined as,
$
\pi_N(\mathbf{x}) = \frac{1}{N} \sum_{i=1}^N \delta_{\mathbf{x}_i}(\mathbf{x}).
$
We use the notation $L^2(\pi)$ to denote the space of square-integrable functions with respect to the measure $\pi$. Accordingly, we define $L^2(\pi_N)$ as the space of functions $u:X\to \mathbb{R}^n$, endowed with the inner-product and norm-squared defined as,
\[
\langle u,u \rangle_{L^2(\pi_N)} = \|u\|_{L^2(\pi_N)}^2= \int_M u^2(x) d\pi_N(x) = \frac{1}{N}  \sum_{i=1}^N u^2(\mathbf{x}_i).
\]
Given point cloud data, we first approximate the sampling density, $q=d\pi/dV$ evaluated at $\mathbf{x}_i$, with $Q_i := \epsilon^{-d/2}N^{-1} \sum_{j=1}^N h\left(\frac{\|\mathbf{x}_i-\mathbf{x}_j\|^2}{\epsilon}\right)$. Define $\mathbf{W}\in\mathbb{R}^{N\times N}$ with entries, 
\BEA
\mathbf{W}_{ij}:= \epsilon^{-d/2-1}N^{-1} h\left(\frac{\|\mathbf{x}_i-\mathbf{x}_j\|^2}{\epsilon}\right)\sqrt{\kappa(\mathbf{x}_i)\kappa(\mathbf{x}_j)}Q_j^{-1}.\label{weightmatrix}
\EEA 
Define also a diagonal matrix $\mathbf{D}\in\mathbb{R}^{N\times N}$ with diagonal entries, $\mathbf{D}_{ii}= \sum_{j=1}^N\mathbf{W}_{ij}$. Then, we approximate the integral operator in \eqref{integralapprox} with the matrix $\mathbf{L}_\epsilon:= \mathbf{W} - \mathbf{D}$, similarly to a discrete unnormalized graph Laplacian matrix. Here, the matrix $\mathbf{L}_\epsilon$ is self-adjoint and semi negative-definite with respect to the inner-product in $\langle u,v\rangle_Q:=\frac{1}{N}\sum_{i=1}^N u(\mathbf{x}_i)v(\mathbf{x}_i)Q_i^{-1}$ such that it admits a non-positive spectrum, $0=\lambda_1>\lambda_2 \geq \ldots\geq \lambda_N$ with eigenvectors orthonormal in $\langle\cdot,\cdot\rangle_Q$.
Since the kernel function $h$ decays exponentially, the k-nearest-neighbor algorithm is usually used to impose sparsity to the estimator $\mathbf{L}_\epsilon$. The DM-based PDE solver approximates the PDE solution $u(\mathbf{x}_i)$ using the $i$-th component of the vector $\mathbf{u}_\epsilon\in\mathbb{R}^N$ that satisfies the linear system
\begin{equation}\label{eqn:dDM}
(-\mathbf{a}+\mathbf{L}_\epsilon)\mathbf{u}_\epsilon = \mathbf{f}
\end{equation}
of size $N$. Here, the $i$-th diagonal component of the diagonal matrix $\mathbf{a}\in\mathbb{R}^{N\times N}$ and the $i$-th component of $\mathbf{f}\in\mathbb{R}^N$ are $a(\mathbf{x}_i)$ and $f(\mathbf{x}_i)$, respectively. In \cite{gh:18}, this approach has been theoretically justified and numerically extended to approximate the non-symmetric, uniformly elliptic second-order differential operators associated to the generator of It\^o diffusions with appropriate local kernel functions.

{When the sampling density is not bounded away from zero or too far away from uniform distribution, the above operator
estimation obtained from the fixed bandwidth kernels can be unbounded. This issue can be
overcome by applying variable bandwidth kernels of the form:
\begin{equation}
K_{\epsilon ,\rho }(\mathbf{x},\mathbf{y})=\exp \left( -\frac{\left\Vert
\mathbf{x}-\mathbf{y}\right\Vert^2 }{4\epsilon \rho (\mathbf{x})\rho (\mathbf{y%
})}\right) ,\label{VBDM}
\end{equation}%
where $\rho $\ is a bandwidth function, chosen to be inversely proportional to the sampling density $q$. Since the sampling density is usually unknown, we 
first need to approximate it. While there are many ways to estimate density, in our algorithm we employ kernel density estimation with the following kernel that is closely related to the cKNN \cite{berry2019consistent} and self-tuning kernel \cite{zelnik2004self},
\[
K_{\epsilon ,0 }(\mathbf{x},\mathbf{y})=\exp \left( -\frac{\left\Vert
\mathbf{x}-\mathbf{y}\right\Vert^2 }{2\epsilon \rho_0(\mathbf{x})  \rho_0(\mathbf{y}) }\right),
\]
where $\rho_0(\mathbf{x}) := \Big(\frac{1}{k_2-1}\sum_{j=2}^{k_2} \|\mathbf{x} - \mathbf{x}_j \|^2\Big)^{1/2}$
denotes the average distance of $\mathbf{x}$ to the first $k_2$-nearest neighbors $\{\mathbf{x}_j\}, j=2,\ldots k_2$ excluding itself. Using this kernel, the sampling
density $q(\mathbf{x})$\ is estimated by $Q (\mathbf{x})=\sum_{j=1}^{N}K_{\epsilon ,0
}(\mathbf{x},\mathbf{x}_{j})/\rho_0( \mathbf{x})^{d}$\ at given point
cloud data. 

\rev{Before we estimate $\mathcal{L}$ with the variable bandwidth kernel in \eqref{VBDM}, let us give a brief overview of the estimation of the Laplace-Beltrami operator that occurs in the asymptotic expansion of the integral operator defined with kernel $K_{\epsilon,\rho}$. We refer interested readers to Eq.(A.12) in \cite{bh:16vb}) for the detailed of the asymptotic expansion, which is a generalization of \eqref{integralop1}, as it involves variable bandwidth function $\rho$. To estimate the Laplace-Beltrami operator, one chooses the bandwidth function to be $\rho(\mathbf{x})= q(\mathbf{x})^{\beta} \approx Q(\mathbf{x})^{\beta}$, with $\beta =-1/2$. With this bandwidth function, we employ the DM algebraic steps to approximate the Laplace-Beltrami operator. That is, }define $Q_\rho (\mathbf{x})=\sum_{j=1}^{N}K_{\epsilon ,\rho
}(\mathbf{x},\mathbf{x}_{j})/\rho( \mathbf{x})^{d}$. Then, we remove the sampling bias by applying a right
normalization $K_{\epsilon ,\rho ,\alpha }(\mathbf{x}_{i},\mathbf{x}%
_{j})= \frac{K_{\epsilon ,\rho }(\mathbf{x}_{i},\mathbf{x}_{j})}{Q_{\rho }(\mathbf{x}%
_{i})^{\alpha }Q_{\rho }(\mathbf{x}_{j})^{\alpha }}$,
where $\alpha = -d/4+1/2$.
We refer to \cite{harlim:18,bh:16vb} for more
details about the \rev{variable bandwidth diffusion map (VBDM)} estimator of weighted Laplacian with other choices of $\alpha$ and $\beta$. Define diagonal
matrices $\mathbf{Q}$ and $\mathbf{P}$ with entries $\mathbf{Q}_{ii}=Q_{\rho
}(\mathbf{x}_{i})$ and $\mathbf{P}_{ii}=\rho (\mathbf{x}_{i})$,
respectively, and also define the symmetric matrix $\mathbf{K}$ with entries
$\mathbf{K}_{ij}=K_{\epsilon ,\rho ,\alpha }(\mathbf{x}_{i},\mathbf{x}_{j})$%
. Next, one can obtain the \rev{VBDM}
estimator, $\mathbf{L}_{\epsilon ,\rho }:=\mathbf{P}^{-2}(\mathbf{D}^{-1}%
\mathbf{K-I})/\epsilon $, where $\mathbf{I}$\ is an identity matrix, as a discrete estimator to the Laplace-Beltrami operator in high probability. 

\rev{To approximate the differential
operator $\mathcal{L}$ in \eqref{ellipticPDE}, we use the fact that,}
\[
\mbox{div}_{g}(\kappa (\mathbf{x})\nabla _{g}u(\mathbf{x}))=\sqrt{\kappa }%
[\Delta _{g}(u\sqrt{\kappa })-u\Delta _{g}\sqrt{\kappa }].
\]%
This relation suggests that we can estimate $\mathcal{L}$ in \eqref{ellipticPDE}
with the matrix $\mathbf{L}_{\epsilon ,\rho }^{\kappa }:=\mathbf{AL}%
_{\epsilon ,\rho }\mathbf{A}-\mathbf{E}$, where $\mathbf{A}$ and $\mathbf{E}$
are diagonal matrices\ with entries $\mathbf{A}_{ii}=\sqrt{\kappa (\mathbf{x}%
_{i})}$ and $\mathbf{E}_{ii}=(\mathbf{AL}_{\epsilon ,\rho }\mathbf{A1)}_{i}$
with $\mathbf{1}$ being a vector with all entries equal to 1. Here, the
property of $\mathbf{L}_{\epsilon ,\rho }^{\kappa }$ is similar to that of a
discrete estimator $\mathbf{L}_{\epsilon ,\rho }$.\ The matrix $\mathbf{L}%
_{\epsilon ,\rho }^{\kappa }$ is self-adjoint and semi negative definite
with respect to the inner-product in $\left\langle u,v\right\rangle _{%
\mathbf{S}}:=\frac{1}{N}\sum_{i=1}^{N}u(\mathbf{x}_{i})v(\mathbf{x}_{i})%
\mathbf{S}_{ii}^{2},$ where $\mathbf{S}_{ii}$ is the diagonal entries of $%
\mathbf{S}=\mathbf{PD}^{1/2}$. Then it is clear that $\mathbf{L}_{\epsilon
,\rho }^{\kappa }$ also admits a non-positive spectrum $\{\lambda_i\}_{i=1}^{N}$ 
with $\lambda_1 = 0$ and an associated basis of eigenvectors 
$\{\psi_i\}_{i=1}^N$ orthonormal in 
$\left\langle \cdot ,\cdot \right\rangle _{\mathbf{S}}$ with $\psi_1 \equiv 1$.

\rev{Thus, we have two estimators, $\mathbf{L}_{\epsilon }$ obtained via the fixed bandwidth kernel $h$ in \eqref{integralop1} and $\mathbf{L}_{\epsilon ,\rho }^{\kappa }$ obtained via the variable bandwidth kernel in \eqref{VBDM}, for the differential operator $\mathcal{L}$. As we noted before, the motivation for VBDM estimation is to overcome samples that are far away from uniform. While theoretically, they induce the same estimate, practically, we found the VBDM estimate is more robust to the choice of $\epsilon$. Particularly, it produces accurate estimation when we specify $\epsilon$ using the auto-tuning method proposed in \cite{coifman2008TuningEpsilon} that is practically more convenient than hand-tuning $\epsilon$. To summarize, we refer to the VBDM-based PDE solver as the linear system in \eqref{eqn:dDM} but with DM estimator $\mathbf{L}_{\epsilon }$\ replaced with the VBDM estimator $\mathbf{L}_{\epsilon,\rho }^{\kappa }$.} In the remainder of this paper, we will only use the notation $\mathbf{L}_\epsilon$ as a discrete estimator of $\mathcal{L}$ and understand it as $\mathbf{L}_{\epsilon,\rho }^{\kappa }$ when VBDM is used.

Beyond the no boundary case, the approximation in \eqref{integralapprox}, unfortunately, will not produce an accurate approximation when $\mathbf{x}$ is sufficiently closed to the boundary. To overcome this issue, a modified DM algorithm, following the classical ghost points method to obtain a higher-order finite-difference approximation of Neumann problems (e.g. \cite{leveque2007finite}) was proposed in \cite{jiang2023ghost}. The proposed method, which is called the \emph{Ghost Point Diffusion Maps} (GPDM),
appends the point clouds with a set of ghost points away from the boundary along the outward normal collar such that the resulting discrete estimator is consistent in the $L^2(\mu_{N})$-sense, averaged over $N$ interior points on the manifold $M$ \cite{yan2021kernel}. To simplify the discussion in the remainder of this paper, we will use the same notation $\mathbf{L}_\epsilon$ to denote the discrete estimate of $\mathcal{L}$, obtained either via the classical or the ghost points diffusion maps. In Appendix~\ref{appendix_a}, we provide a brief review on the GPDM for Dirichlet boundary value problems.

\section{Solving PDEs on Unknown Manifolds using Diffusion Maps and Neural Networks}\label{NNsolver}

\comment{As mentioned previously, the DM-based PDE solver has several computational limitations. Namely, the estimated solution is in a discrete form with function values at the given point clouds data; the size of the matrix that approximates the differential operator increases as a function of the data size, which creates a computational bottleneck when the PDE solution involves a (pseudo) inversion operation or an eigenvalue decomposition. Hence,} We now present a new hybrid algorithm based on DM and NNs.

\subsection{Deep neural networks (DNNs)}\label{sec:fnn} Mathematically, DNNs are highly nonlinear functions constructed by compositions of simple nonlinear functions. For simplicity, we consider the fully connected feed-forward neural network (FNN), which is the composition of $L$ simple nonlinear functions as follows:
$\phi(\mathbf{x};\bm{\theta}):=\mathbf{a}^\top \mathbf{h}_L \circ \mathbf{h}_{L-1} \circ \cdots \circ \mathbf{h}_{1}(\mathbf{x})$,
 where $\mathbf{h}_{\ell}(\mathbf{x})=\sigma\left(\mathbf{W}_\ell \mathbf{x} + \mathbf{b}_\ell \right)$ with $\mathbf{W}_\ell \in \mathbb{R}^{N_{\ell}\times N_{\ell-1}}$, $\mathbf{b}_\ell \in \mathbb{R}^{N_\ell}$ for $\ell=1,\dots,L$, $\mathbf{a}\in \mathbb{R}^{N_L}$, $\sigma$ is a nonlinear activation function, e.g., a rectified linear unit (ReLU) $\sigma(x)=\max\{x,0\}$. Each $\mathbf{h}_\ell$ is referred as a hidden layer,  $N_\ell$ is the width of the $\ell$-th layer, and $L$ is called the depth of the FNN. In the above formulation, $\bm{\theta}:=\{\mathbf{a},\,\mathbf{W}_\ell,\,\mathbf{b}_\ell:1\leq \ell\leq L\}$ denotes the set of all parameters in $\phi$. For simplicity, we focus on FNN with a uniform width $m$, i.e., $N_\ell = m$ for all $\ell\neq 0$, in this paper.
 %defined recursively as follows:
%$  \mathbf{h}_0=\mathbf{V}\mathbf{x}$,  $\mathbf{g}_\ell=\sigma(\mathbf{W}_\ell\mathbf{h}_{\ell-1}+\mathbf{b}_{\ell})$,
%$\mathbf{h}_\ell=\mathbf{\bar{U}}_\ell \mathbf{h}_{\ell-2}+\mathbf{U}_\ell\mathbf{g}_\ell$,  %\ell=1,2,\dots,L,
%$\phi(\mathbf{x};
% \bm{\theta})=\mathbf{a}^{\T}\mathbf{h}_L$,  where $\ell=1,2,\dots,L$, $\bf V\in \mathbb{R}^{N_0\times d}$, $\bf W_{\ell}\in \mathbb R^{N_{\ell}\times N_0}$,
%$\tilde{\bf U}_{\ell}\in \mathbb R^{N_0\times N_0}$, $U_{\ell}\in \mathbb R^{N_0\times N_{\ell}}$, $\bf b_{{\ell}}\in \mathbb R^{N_{\ell}}$ for ${\ell}=1, \cdots, L$, $\bf a\in \mathbb R^{N_0}$, $\bf h_{-1}=0$, and $\bm{\theta}$ is the set of all parameters. For simplicity,   ResNets and the following settings are employed: $N_0=N_{\ell}=N$, and $U_{\ell}$ is an identity matrix in the numerical implementation. %Furthermore, as used in \cite{ey2018}, we set $\tilde{U}_{\ell}$ as the identity matrix when $\ell$ is even and  set $\tilde{U}_{\ell}=0$ when $\ell$ is odd. In all the numerical examples in this paper, we will use ResNet because it is better than the FNN empirically in our experiments.

\subsection{Supervised Learning} Supervised learning approximates an unknown target function $f:\mathbf{x}\in\Omega \rightarrow y{\in \mathbb{R}}$ from training samples $\{(\mathbf{x}_i,y_i)\}_{i=1}^N$, where $\mathbf{x}_i$'s are usually assumed to be i.i.d samples from an underlying distribution $\pi$ {defined on a domain $\Omega\subseteq\mathbb{R}^n$,} and $y_i=f(\mathbf{x}_i)$. Consider the square loss $\ell(\mathbf{x},y;\bm{\theta})=\frac{1}{2}\left| \phi(\mathbf{x};\bm{\theta})-y\right|^2$ of a given DNN $\phi(\mathbf{x};\bm{\theta})$ that is used to approximate $f(\mathbf{x})$,   the population risk (error) and empirical risk (error) functions are, respectively,
\begin{equation}\label{eqn:pop}
\mathcal{J}(\bm{\theta})=\frac{1}{2}\E_{\mathbf{x}\sim \pi}\left[ \left| \phi(\mathbf{x};\bm{\theta})-f(\mathbf{x})\right|^2 \right],\quad \hat{\mathcal{J}}(\bm{\theta})=\frac{1}{2N}\sum_{i=1}^N \left| \phi(\mathbf{x}_i;\bm{\theta})-y_i\right|^2.
\end{equation}
The optimal set $\hat{\bm{\theta}}$ is identified via $\hat{\bm{\theta}}=\argmin_{\bm{\theta}} \hat{ \mathcal{J}}(\bm{\theta})$, and $\phi(\mathbf{x};\hat{\bm{\theta}})$ is the learned DNN that {approximates} the unknown function $f$.

\subsection{The NN-based PDE solver with DM} Solving a PDE can be transformed into a supervised learning problem. Physical laws, like PDEs and boundary conditions, are used to generate training data in a supervised learning problem to infer the solution of PDEs. In the case when the PDE is defined on a manifold, we propose to use DM as an approximation to the differential operator according to \eqref{integralapprox} to obtain a linear system in \eqref{eqn:dDM}, apply NN to parametrize the PDE solution, and adopt a least-square framework to identify the NN that approximates the PDE solution. For example, to solve the PDE problem in \eqref{ellipticPDE} \rev{in the strong sense} on a $d$-dimensional closed, smooth manifold $M$ identified with points $X=\{\mathbf{x}_1,\ldots,\mathbf{x}_N\}\subset M$, we minimize the following empirical loss
\begin{eqnarray}\label{eqn:DMNN1}
{\bm{\theta}}_S=\argmin_{\bm{\theta}} \mathcal{J}_{S,\epsilon}(\bm{\theta}):= \argmin_{\bm{\theta}} {\frac{1}{2}}\|(-\mathbf{a}+\mathbf{L}_\epsilon) \bm{\phi}_{\bm{\theta}} - \mathbf{f}\|_{L^2(\pi_N)}^2,
\end{eqnarray}
where $\mathbf{L}_\epsilon\in\mathbb{R}^{N\times N}$ denotes the DM estimator for the differential operator $\mathcal{L}$, $\mathbf{a}$ and $\mathbf{f}$ are defined in \eqref{eqn:dDM}, $\bm{\phi}_{\bm{\theta}}\in \mathbb{R}^N$ with the $i$-th entry as $\phi(\mathbf{x}_i;\bm{\theta})$.
When $\mathbf{a}=\mathbf{0}$, we add a regularization term $\frac{\gamma}{2}\|\bm{\phi}_{\bm{\theta}}\|_{L^2(\pi_N)}^2$ in the loss function~\eqref{eqn:DMNN1} to guarantee well-posedness, where $\gamma>0$ is a regularization parameter.
When stochastic gradient descent (SGD) is used to minimize \eqref{eqn:DMNN1}, a small subset of the given point clouds is randomly selected in each iteration. This amounts to randomly choosing batches, consisting of a few rows of $(-\mathbf{a}+\mathbf{L}_\epsilon)$ and a few entries of $\mathbf{f}$, to approximate the empirical loss function.

%\jh{Do we need this sentence? When gradient descent is applied to minimize \eqref{eqn:DMNN1}, all data samples from the given point clouds are used in \eqref{eqn:DMNN1}. When stochastic gradient descent is used to reduce the computational cost per iteration, a small subset of the data sampled from the given point clouds are randomly selected in each iteration to minimize \eqref{eqn:DMNN1}, i.e., randomly choosing a few rows of $(-\mathbf{a}+\mathbf{L}_\epsilon)$ and a few entries of $\mathbf{f}$ to form a random empirical loss.}

In the case of Dirichlet problems with non-homogeneous
boundary conditions, $u(\mathbf{x})=g(\mathbf{x})$, $\forall \mathbf{x}\in \partial M$, given
boundary points $\{\bar{\mathbf{x}}_1,\ldots,\bar{\mathbf{x}}_{N_b}\} \subset X\cap \partial M$ as the last $N_b$ points of $X$, a penalty term is added to \eqref{eqn:DMNN1} to enforce the boundary condition as follows:
\begin{align}\begin{split}
\label{eqn:DMNN2}
{\bm{\theta}}_S
\comment{\argmin_{\bm{\theta}} \mathcal{J}_{S,\epsilon}(\bm{\theta})\\:&} := \argmin_{\bm{\theta}} {\frac{1}{2}}\|(-\mathbf{a}+\mathbf{L}_\epsilon) \bm{\phi}_{\bm{\theta}} - \mathbf{f}\|_{L^2(\pi_{N-N_b})}^2 + {\frac{\lambda}{2}} \| \bm{\phi}^b_{\bm{\theta}}-\mathbf{g}\|_{L^2(\pi_{N_b})}^2,
\end{split}\end{align}
where {$\mathbf{L}_\epsilon\in\mathbb{R}^{(N-N_b)\times N}$ denotes the GPDM estimator for the differential operator $\mathcal{L}$ defined for problem with boundary and $\lambda>0$ is a hyper-parameter.
The construction of the matrix $\mathbf{L}_\epsilon$ is discussed in Appendix~\ref{appendix_a}.
Letting $\T$ denotes the transpose operator, we have also defined a column vector $\bm{\phi}^b_{\bm{\theta}} = \big(\phi(\mathbf{\bar{x}}_1;\bm{\theta}),\ldots, \phi(\mathbf{\bar{x}}_{N_b};\bm{\theta})\big)^{\T} \in \mathbb{R}^{N_b}$, whose components are also elements of the column vector $\bm{\phi}_{\bm{\theta}}= \big(\phi(\mathbf{{x}}_1;\bm{\theta}),\ldots, \phi(\mathbf{{x}}_{N};\bm{\theta})\big)^{\T}\in \mathbb{R}^N$; the column vector $\mathbf{f} = \big(f(\mathbf{x}_1),\ldots,f(\mathbf{x}_{N-N_b})\big)^{\T} \in\mathbb{R}^{N-N_b}$ with function values on the interior points; and the column vector $\mathbf{g} = \big(g(\bar{\mathbf{x}}_1),\ldots,g(\bar{\mathbf{x}}_{N_b})\big)^{\T} \in\mathbb{R}^{N_b}$ with function values on the boundary points. In the case of other kinds of boundary conditions, a corresponding boundary operator can be applied to $\bm{\phi}^b_{\bm{\theta}}$ to enforce the boundary condition.
}

\section{Theoretical Foundation of the Proposed Algorithm}\label{theory}

The classical machine learning theory concerns with characterizations of the approximation error, optimization error estimation, and generalization error analysis.
For the proposed PDE solver, the approximation theory involves characterizing: 1) the error of the DM-based discrete approximation of the differential operator on manifolds, and 2) the error of NNs for approximating the PDE solution. In the optimization algorithm, a numerical minimizer (denoted as $\vtheta_N$) provided by a certain algorithm might not be a global minimizer of the empirical risk minimization in \eqref{eqn:DMNN1} and \eqref{eqn:DMNN2}. Therefore, designing an efficient optimization algorithm such that the optimization error $|\mathcal{J}_{S,\epsilon}(\vtheta_N)- \mathcal{J}_{S,\epsilon}(\vtheta_S)|\approx 0$ is important. In the generalization analysis, the goal is to quantify the error defined as $\|u - \phi(\cdot,\bm{\theta_S})\|_{L^2(\pi)}$ over the distribution $\mathbf{x}\sim\pi$ that is unknown.

In approximation theory, the error analysis of DM is relatively well-developed, while the error of NNs is still under active development. Recently, there are two directions have been proposed to characterize the approximation capacity of NNs. The first one characterizes the approximation error in terms of the total number of parameters in an NN \cite{yarotsky2017,yarotsky18a,Montanelli2019_3,E2018,Montanelli2019,yarotsky2019,PETERSEN2018296,guhring2019error,SIEGEL2020313}. The second one quantifies the approximation error in terms of NN width and depth \cite{Shen2,Shen3,Shen4,Shen5,Shen6,Wang}. In real applications, the width and depth of NNs are the required hyper-parameters to decide for numerical implementation instead of the total number of parameters. Hence, we will develop the approximation theory for the proposed PDE solver adopting the second direction.

In optimization theory, for regression problems without regularization (e.g., the right equation of \eqref{eqn:pop}), it has been shown in \cite{Arthur18,Simon18,MeiE7665,du2018gradient,Yiping20} that GD and SGD algorithms can converge to a global minimizer of the empirical loss function under the assumption of over-parametrization (i.e., the number of parameters are much larger than the number of samples). However, existing results for regression problems cannot be applied to the minimization problem corresponding to PDE solvers, which is much more difficult due to differential operators and boundary operators. A preliminary attempt was conducted in \cite{Luo2020}, but the results in \cite{Luo2020} cannot be directly applied to our minimization problem in \eqref{eqn:DMNN1} and \eqref{eqn:DMNN2}. Below, we will develop a new analysis to show that GD can identify a global minimizer of \eqref{eqn:DMNN1}.

The generalization analysis aims at quantifying the convergence of the generalization error $\|u-\phi(\cdot;\bm{\theta}_S)\|_{L^2(\pi)}$. %i.e., showing that a global minimizer of the empirical risk minimization can also keep the population risk small. 
Let $u(X)\in\mathbb{R}^N$ be a column vector representing the evaluation of $u$ on the training data set $X=\{\mathbf{x}_1,\ldots,\mathbf{x}_N\}$ and this notation is used similarly for other functions. 
Beyond the identification of $\bm{\theta_S}$ (a global minimizer of the empirical loss), which is addressed in the optimization theory, a typical approach is to estimate the difference between $\|u-\phi(\cdot;\bm{\theta}_S)\|_{L^2(\pi)}$ and $\|u(X)-\phi(X;\bm{\theta}_S)\|_{L^2(\pi_N)}$ via statistical learning theories. %There have been several papers for the generalization error analysis of PDE solvers, e.g., \cite{DBLP:journals/corr/abs-1809-03062,Han2018ConvergenceOT,PINNgen,Luo2020,lu2021priori,duan2021convergence,hong2021priori,lu2021priori2}. In \cite{DBLP:journals/corr/abs-1809-03062,Han2018ConvergenceOT,PINNgen,lu2021priori,duan2021convergence,hong2021priori,lu2021priori2}, under the assumption that there exists a minimizer of the empirical risk minimization, satisfying a certain norm constraint, that generalizes well, e.g., the minimizer corresponds to an NN with a small Lipschitz constant or norm. 
There have been several papers for the generalization error analysis of PDE solvers. In \cite{DBLP:journals/corr/abs-1809-03062,Han2018ConvergenceOT,PINNgen,lu2021priori,duan2021convergence,hong2021priori,lu2021priori2}, they show that minimizers of the empirical risk minimization can facilitate a small generalization error if these minimizers satisfy a certain norm constraint (e.g., with a small parameter norm or corresponding to an NN with a small Lipschitz constant). However, it is still an open problem to design numerical optimization algorithms to identify such a minimizer. Another direction is to regularize the empirical risk minimization so that a global minimizer of the regularized loss can generalize well \cite{Luo2020}. Nevertheless, there is no global convergence analysis of the optimization algorithm for the regularized loss, which suggests that there is no guarantee that one can practically obtain the global minimizer of the regularized loss that can generalize well.

The theoretical analysis in this paper focuses on the approximation and optimization perspectives of the proposed PDE solver to develop an error analysis of $\|u(X)-\phi(X;\bm{\theta}_S)\|_{L^2(\pi_N)}$. In the discussion below, we restrict to the case of manifolds without boundaries to convey our main ideas for the theoretical analysis of the proposed algorithm. We further assume that the ReLU activation function, i.e., $\max\{x,0\}$, its power, e.g., ReLU$^r$, and FNNs are used in the analysis. An extension to other activation functions and network architectures might be possible. We will show that the numerical solution of the proposed solver is consistent with the ground truth in appropriate limits, assuming that a global minimizer of our optimization problem is obtainable. We will also show that GD can identify a global minimizer of our optimization problem for two-layer NNs when their width is sufficiently large. The optimization theory together with the parametrization error analysis forms the theoretical foundation of the convergence analysis of the proposed PDE solver on manifolds.

\comment{

We should point out that $\bm{\theta_S}$ depends on $m$, $L$, $\epsilon$, and $N$. Hence, let us write it as $\bm{\theta_S}(N,\epsilon,m,L)$. The generalization error is defined as
\BEA
\|u - \phi(\cdot,\bm{\theta_S}(N,\epsilon,m,L))\|_{L^2(\pi)}.
\EEA
A typical method to analyze the generalization error is to use the Rademacher complexity (see Theorem B.1) to bound
\BEA
\left|\|u(X) - \phi(X,\bm{\theta_S}(N,\epsilon,m,L)) \|_{L^2(\pi_N)} - \|u - \phi(\cdot,\bm{\theta_S}(N,\epsilon,m,L))\|_{L^2(\pi)}\right|\leq C(N,\epsilon,m,L).
\EEA
If $C(N,\epsilon,m,L)$ and $\|u(X) - \phi(X,\bm{\theta_S}(N,\epsilon,m,L)) \|_{L^2(\pi_N)}$ are small, then the generalization error is small. In the theoretical analysis of the proposed PDE solver on unknown manifolds, we will quantitatively characterize $\|u(X) - \phi(X,\bm{\theta_S}(N,\epsilon,m,L)) \|_{L^2(\pi_N)}$ and show that it will converges to zero in the limit of $\epsilon \to 0$, $N\to \infty$, and $m\to \infty$ for any fixed $L$.}

\subsection{Parametrization Error} The proposed PDE solver on manifolds applies DM to parametrize differential operators on manifolds and uses an NN to parametrize the PDE solution. We will quantify the parametrization error due to these two ideas, assuming that a global minimizer of the empirical loss minimization in \eqref{eqn:DMNN1} is achievable via a certain numerical optimization algorithm, i.e., estimating $\|\mathbf{u}-\bm{\phi}_{S}\|_{L^2(\pi_N)}$, where $\bm{\phi}_{S}\in \mathbb{R}^N$ with the $i$-th entry as $\phi(\mathbf{x}_i;\bm{\theta_S})$
is the NN-solution of the PDE in \eqref{eqn:DMNN1}. %Define  $M_\mu:=\{\mathbf{x}\in\mathbb{R}^n:\underset{\mathbf{y}\in M}{\inf}\|\mathbf{x}-\mathbf{y}\|_2\leq \mu\}$.
% Our main convergence theory is summarized below and the proof is in SI Appendix \ref{appendix_b}.

\begin{theo}[Parametrization Error]\label{thm:pconv}
Let $u$ be the solution of \eqref{ellipticPDE} with $0<a_{\min}\leq a(x)\leq a_{\max}$ on $X=\{\mathbf{x}_1,\ldots,\mathbf{x}_N\}$, randomly sampled from a distribution $\pi$ on $M\subset [0,1]^n$, where $M$ is a $C^4$-manifold with condition number $\tau_M^{-1}$, volume $V_M$, and geodesic covering regularity $G_M$. For
$u,\kappa \in C^4(M)\cap  L^2(M)$ and $q\in C(M)$, where $q = d\pi/dV$, with probability higher than $1-N^{-2}$,
\begin{align}\begin{split}
\|u(X)-\phi(X;\bm{\theta_S})\|_{L^2(\pi_N)} = \mathcal{O} \left(\epsilon,\Big(\frac{\log(N)}{N}\Big)^{\frac{1}{2}} \epsilon^{-2-\frac{d}{4}},\Big(\frac{\log(N)}{N}\Big)^{\frac{1}{2}}\epsilon^{-\frac{1}{2}-\frac{d}{4}},N^{1/2}\epsilon^{-1} m^{-\frac{8}{d\ln(n)}}L^{-\frac{8}{d\ln(n)}}\right),\label{perror}
\end{split}\end{align}
as $\epsilon\to 0$, after $N\to\infty$ and $m$ or $L\to\infty$. Hence,
$\lim_{\epsilon\to 0}\lim_{N\to \infty}\lim_{m\to \infty} \|u(X)-\phi(X;\bm{\theta_S})\|_{L^2(\pi_N)} = 0$. Here $\phi(\mathbf{x};\bm{\theta_S})$ has width $\mathcal{O}(n\ln(n)m\log(m))$ and depth $\mathcal{O}(L\log(L)+\ln(n))$ with $m\in\mathbb{N}^+$ and $L\in\mathbb{N}^+$ as two integer hyper-parameters.
\end{theo}

In \eqref{perror} and throughout this paper, the big-oh notation means  $\mathcal{O}(f,g) := \mathcal{O}(f)+ \mathcal{O}(g)$,
as $f,g\to 0$. The first three error terms of \eqref{perror} come from the DM discretization, whereas the last error term is due to the approximation property of NNs. { The sequence of convergence in the big-O notation is due to the fact that $\epsilon>0$ is fixed when we analyze the discretization error. For example, the first error rate in \eqref{perror} corresponds to the error induced by the integral operator approximation in \eqref{integralapprox}, whereas the third error rate in \eqref{perror} corresponds to the discretization of $L_\epsilon$ using $\mathbf{L}_\epsilon$ for which one deduces error rate as $N \to \infty$ with fixed $\epsilon$. The same reasoning applies to the NN approximation. }

%In \eqref{perror}, for a given sufficiently small $\epsilon$, there exists a sufficiently large $N_0$ such that $ \mathcal{O} \big(\epsilon,{N}^{-\frac{1}{2}} \epsilon^{-2-\frac{d}{4}},N^{-\frac{1}{2}}\epsilon^{-\frac{1}{2}-\frac{d}{4}})=\mathcal{O}(\epsilon)$ for $N\geq N_0$. For this $\epsilon$ and $N$, there exists a sufficiently large $m_0$ such that $\mathcal{O} \big(\epsilon,{N}^{-\frac{1}{2}} \epsilon^{-2-\frac{d}{4}},N^{-\frac{1}{2}}\epsilon^{-\frac{1}{2}-\frac{d}{4}},N^{1/2}\epsilon^{-1} m^{-{8}/(d\ln(n))}L^{-{8}/(d\ln(n))}\big)=\mathcal{O}(\epsilon)$ when $m\geq m_0$ for any $L$. For example, 
In \eqref{perror}, in the case of uniformly sampled data, the second error term, induced by the estimation of the sampling density $q$, becomes irrelevant and the factor $N^{1/2} $ in the last error term disappears due to symmetry. In such a case, the number of data points needed to achieve order-$\epsilon$ of the DM discretization is $N\geq N_0 = \mathcal{O}(\epsilon^{-\frac{6+d}{2}})$, obtained by balancing the first and third error terms in \eqref{perror}. The NN width to achieve the same accuracy is $m\geq m_0 = \mathcal{O}(\epsilon^{-\frac{d\ln(n)}{4}})$, obtained by balancing the first and last error terms in \eqref{perror}.

Before we prove Theorem \ref{thm:pconv}, let us review relevant results that will be used for the proof.

\vspace{0.25cm}
\noindent {\bf The Parametrization Error of DM.} For reader's convenience, we briefly summarize the pointwise error bound of the discrete estimator, which has been
reported extensively (see e.g., \cite{cl:06,SingerEstimate,bh:16vb,dunson2021spectral}). Our particular interest is to quantify the error induced by the matrix $\mathbf{L}_{a,\epsilon}:=-\mathbf{a}+\mathbf{L}_\epsilon$, where $\mathbf{L}_\epsilon$ is defined right after \eqref{integralapprox} %in \eqref{unnormalizeddgl} 
and $\mathbf{a}$ is a diagonal matrix with diagonal entries $\mathbf{a}_{ii}=a(\mathbf{x}_i)$.

First, let us quantify the Monte-Carlo error of the discretization of the integral operator, i.e., the error for introducing $L_\epsilon$ in \eqref{integralapprox}. In particular,
using the Chernoff bound (see Appendix B.2 in \cite{bh:16vb} or Appendix A in \cite{gh:18}), for any $\mathbf{x}_i\in X$ and any fixed $\epsilon, \eta>0$, and $u\in L^2(M)$, we have
\BEA
\mathbb{P}(|(\mathbf{L}_{\epsilon}\mathbf{u})_i-L_\epsilon u(\mathbf{x}_i)|>\eta) < 2\exp\Big(-C \frac{\eta^2\epsilon^{d/2+1}N}{\|\nabla_g u(\mathbf{x}_i)\|^2 q(\mathbf{x}_i)^{-1}}\Big), \nonumber
\EEA
for some constant $C$ that is independent of $\epsilon$ and $N$.
{ This version of Chernoff bound accounts for the variance error and  yields similar rate as in \cite{hein2007graph}, which uses the Bernstein inequality.}
Choosing $N^2 = \frac{1}{2}\exp\Big(C \frac{\eta^2\epsilon^{d/2+1}N}{\|\nabla_g u(\mathbf{x}_i)\|^2 q(\mathbf{x}_i)^{-1}}\Big) $, one can deduce that  \[\eta = C^{-1/2}\Big(\frac{\log (2N^2)}{N}\Big)^{1/2}\epsilon^{-1/2-d/4}\|\nabla_g u(\mathbf{x}_i)\| q(\mathbf{x}_i)^{-1/2},\] which means that with probability greater than $1-N^{-2}$, 
\BEA
(\mathbf{L}_{\epsilon}\mathbf{u})_i = L_\epsilon u(\mathbf{x}_i) + \mathcal{O}\Big(\frac{\|\nabla_gu(\mathbf{x}_i)\| q(\mathbf{x}_i)^{-1/2}(\log(N))^{1/2}}{N^{1/2}\epsilon^{1/2+d/4}}\Big),\nonumber
\EEA
as $N\to\infty$. When the density $q=d\pi/dV$ is non-uniform, one can use the same argument (e.g., see Appendix B.1 in \cite{bh:16vb}) to deduce the error induced by the estimation of the density, which is of order $\mathcal{O}(q(\mathbf{x}_j)^{1/2}(\log(N)/N)^{-1/2}\epsilon^{-2-d/4})$ with probability $1-N^{-2}$, to ensure a density estimation of order-$\epsilon^2$. Together with \eqref{integralapprox}, we have the following pointwise error estimate.
\begin{lem} Let $u,\kappa\in C^3(M)\cap L^2(M)$ and $q\in C(M)$, where $M\subseteq\mathbb{R}^n$ is a $d$-dimensional closed $C^3-$submanifold, then for any $\mathbf{x}_i\in X$, with probability higher than $1-N^{-2}$,
\BEA
(\mathbf{L}_{a,\epsilon}\mathbf{u})_i - \mathcal{L}_au(\mathbf{x}_i) =(\mathbf{L}_{\epsilon}\mathbf{u})_i - \mathcal{L}u(\mathbf{x}_i) = \mathcal{O}\Big(\epsilon,\frac{q(\mathbf{x}_i)^{1/2}(\log(N))^{1/2}}{N^{1/2}\epsilon^{2+d/4}},\frac{\|\nabla_gu(\mathbf{x}_i)\| q(\mathbf{x}_i)^{-1/2}(\log(N))^{1/2}}{N^{1/2}\epsilon^{1/2+d/4}}\Big)\label{discreteforwarderror},
\EEA
as $\epsilon\to 0$ after $N\to\infty$.
\end{lem}

If we allow for $u,\kappa \in C^4(M)$, then by Lemma 3.3 in \cite{hsay:20}, we can replace the first-order error term in \eqref{discreteforwarderror} by $R_u(x)\epsilon^{4\beta-1}$, where $0<\beta<1/2$ such that
\BEA
\|R_u\|_{L^2(M)} \leq C\|u\|_{H^4(M)} \|\sqrt{\kappa}\|^2_{C^{4}(M)}<\infty\nonumber
\EEA
for some constant $C>0$ for $u,\kappa\in C^4(M)$. The term $R_u$ will also appear in the second-error term as well since this error bound is to ensure the density estimation to achieve $\mathcal{O}(\epsilon^2)$.  Also, with the given assumptions, it is immediate to show that $\|R_u q^{1/2}\|^2_{L^2(\pi)}, \|R_u \|^2_{L^2(\pi)}<\infty$.

For a fixed $\epsilon>0$ and using the fact that $\lim_{N\to\infty}\frac{1}{N}\sum_{i=1}^Nf(\mathbf{x}_i)^2=\int_M f(x)^2q(x)dV(x)= \|f\|^2_{L^2(\pi)}$, we have,
\BEA
\lim_{N\to\infty}\| \mathbf{L}_{a,\epsilon}\mathbf{u} - \mathcal{L}_au(X) \|^2_{L^2(\pi_N)} &=& \lim_{N\to\infty} \frac{1}{N}\sum_{i=1}^N |(\mathbf{L}_{a,\epsilon}\mathbf{u})_i - \mathcal{L}_au(x_i) |^2 \notag 
\\&\leq & C_1 \epsilon^{8\beta-2} \lim_{N\to\infty} \frac{1}{N} \sum_{i=1}^NR_u(x_i)^2 + C_2 \epsilon^{-4-\frac{d}{2}}\lim_{N\to\infty}\frac{1}{N^{2}} \sum_{i=1}^N(R_u(x_i)^2q(x_i)) \notag \\ 
&& + C_3 \epsilon^{-1-\frac{d}{2}}\lim_{N\to\infty}\frac{1}{N^{2}}\sum_{i=1}^N \|\nabla_gu(x_i)\|^2(q(x_i)^{-1}) \notag \\
&=& C_1 \epsilon^{8\beta-2}\| R_u \|^2_{L^2(\pi)}  + \big(\lim_{N\to\infty} \frac{1}{N}\big) \big(C_2 \epsilon^{-4-\frac{d}{2}}\| R_u q^{1/2}\|^2_{L^2(\pi)} + C_3 \epsilon^{-1-\frac{d}{2}} \|u\|_{H^1(M)}\big)  \notag\\
&=& C_1 \epsilon^{8\beta-2}\| R_u \|^2_{L^2(\pi)}.\notag
\EEA
To conclude, we have the following lemma.
\begin{lem}
Let $u, \kappa \in C^4(M)$ and $q\in C(M)$ and let $M\in\mathbb{R}^n$ be a $d-$dimensional, closed, $C^3$-submanifold. Then, 
\BEA
\lim_{\epsilon\to 0}\lim_{N\to\infty}\| \mathbf{L}_{a,\epsilon}\mathbf{u} - \mathcal{L}_au(X) \|_{L^2(\pi_N)} = 0,\label{DML2consistency}
\EEA
{almost surely.}
\end{lem}
This consistency estimate only holds when the limits are taken in the sequence as above.

\vspace{0.25cm}
\noindent {\bf The Parametrization Error of NNs.} Let us denote the best possible empirical loss as
%\BEA{\color{magenta}\mathcal{J}_{S,\epsilon}(\bm{\theta}_S) :=\| \mathbf{L}_{a,\epsilon} \bm{\phi}_{S}-\mathbf{f}  \|^2_{L^2(\pi_N)} = \min_{\bm{\theta}}  \| \mathbf{L}_{a,\epsilon} \bm{\phi}_{\bm{\theta}}-\mathbf{f}  \|^2_{L^2(\pi_N)},}\nonumber
%\EEA
%where 
$\mathcal{J}_{S,\epsilon}(\bm{\theta}_S)$, which depends only on the NN model and is independent of the optimization algorithm to solve the empirical loss minimization in \eqref{eqn:DMNN1}, since $\bm{\theta}_S$ is a global minimizer of the empirical loss. The estimation of $\mathcal{J}_{S,\epsilon}(\bm{\theta}_S)$ can be derived from deep network approximation theory in Theorem 1.1 of \cite{Shen3} and its corollary in \cite{du2022discovery}. The (nearly optimal) error bound in Theorem 1.1 of \cite{Shen3} focuses on approximating functions in $C^s([0,1]^n)$ and hence suffers from the curse of dimensionality; namely, the total number of parameters in the ReLU FNN scales exponentially in $n$ to achieve the same approximation accuracy. Fortunately, our PDE solution is only defined on a $d$-dimensional manifold embedded in $[0,1]^n$ with $d\ll n$. By taking advantage of the low-dimensional manifold, a corollary of Theorem 1.1 in \cite{Shen3} was proposed in \cite{du2022discovery} following the idea in \cite{Shen2} to conquer the curse of dimensionality. For completeness, we quote this corollary below.

%\begin{theorem}[Theorem 1.1 of \cite{Shen3}]
%    \label{thm:app}
%    Given a function $u\in  C^s([0,1]^n)$ with $s\in \mathbb{N}^+$, for any $m,L\in \mathbb{N}^+$, there exists a ReLU FNN $\phi$ with width $ C_1n(m+2)\log_2(4m)$ and depth $C_2(L+2)\log_2(2L)+2n$ such that
%    \begin{equation*}
%    \|u-\phi\|_{L^\infty([0,1]^n)}\le C_3 \|u\|_{C^s([0, 1]^n)} m^{-2s/n}L^{-2s/n},
%    \end{equation*}
%    where $C_1=22s^{n+1}3^n$, $C_2=18s^2$, and $C_3=85(s+1)^n8^s$.
%\end{theorem}

\begin{lem}[Proposition 4.2 of \cite{du2022discovery}] \label{col:app}
Given $m,L\in\mathbb{N^+}$, $\mu\in(0,1)$, $\nu\in(0,1)$. Let $M\subset\mathbb{R}^n$ be a compact $d$-dimensional Riemannian submanifold having condition number $\tau_M^{-1}$, volume $V_M$, and geodesic covering regularity $G_M$, and define the $\mu$-neighborhood as
$M_\mu:=\{\mathbf{x}\in\mathbb{R}^n:\underset{\mathbf{y}\in M}{\inf}\|\mathbf{x}-\mathbf{y}\|_2\leq \mu\}$. If $u\in C^s(M_\mu)$ with $s\in\mathbb{N^+}$, then there exists a ReLU FNN $\phi$ with width $17s^{d_\nu+1}3^{d_\nu} d_\nu(m+2)\log_2(8 m)$ and depth $18s^2(L+2)\log_2(4L)+2d_\nu$ such that for any $\mathbf{x}\in M_\mu$,
 \begin{multline}\label{eqn:appd}
|\phi(\mathbf{x})-u(\mathbf{x})|\leq 8\|u\|_{C^s(M_\mu)}  \mu \Big((1-\nu)^{-1}\sqrt{n/d_\nu}+1\Big)+170 (s+1)^{d_\nu}8^s(1-\nu)^{-1}\|u\|_{C^s(M_\mu)}m^{-2s/d_\nu}L^{-2s/d_\nu},
\end{multline}
where $d_\nu:=\mathcal{O}\left(d\ln\left(n V_MG_M\tau_M^{-1}/\nu\right)/\nu^2\right)=\mathcal{O}\left(d\ln(n/\nu)/\nu^2\right)$ is an integer with $d< d_\nu< n$.
\end{lem}

In Lemma \ref{col:app}, ReLU FNNs are used while in practice the learnable linear combination of a few activation functions might boost the numerical performance of neural networks \cite{liang2021reproducing}. The extension of Lemma \ref{col:app} to the case of multiple kinds of activation functions is interesting future work. Lemma \ref{col:app} can provide a baseline characterization to the approximation capacity of FNNs when ReLU is one of the choices of activation functions for a learnable linear combination. Hence, the following error estimation is still true for NNs constructed with the learnable linear combination of a few activation functions.

Using the same argument as the extension lemma for smooth functions (see e.g., Lemma 2.26 in \cite{lee2013smooth}), one can extend any $u\in C^4(M)$ to $u\in C^4(M_{\mu_0})$, for any $C^4$ manifold $M \subseteq\mathbb{R}^n$ and a positive $\mu_0$. Therefore, by Lemma \ref{col:app}, there exists a ReLU FNN $\phi$ with width $\mathcal{O}(n\ln(n)m\log(m))$ and depth $\mathcal{O}(L\log(L)+\ln(n))$ such that
\[
|u(\mathbf{x})-\phi(\mathbf{x})|\le C_{M,d,\nu}  \|u\|_{C^3(M_{\mu_0})} \left(\mu \sqrt{\frac{n}{\ln n} } + nm^{-{8}/(d\ln(n))}L^{-{8}/(d\ln(n))}\right) \quad \text{for all }\mathbf{x}\in M,
\]
for any $\mu\in (0,\mu_0)$, where $C_{M,d,\nu}$ is a constant depending only on $M$, $d$, and $\nu$. Note that the curse of dimensionality has been lessened in the above approximation rate. Therefore, by taking $\mu\to 0$, we have the following error estimation
\begin{equation}\label{eqn:estuphi}
\|\mathbf{u}-\bm{\phi}\|_{L^2(\pi_N)} =  \mathcal{O}(m^{-8/(d\ln(n))}L^{-8/(d\ln(n))}),
\end{equation}
where the prefactor depends on $M$, $d$, $\nu$, $\|u\|_{C^3(M_{\mu_0})} $, and $n$.
%Recall $X =\{  \mathbf{x}_i\}_{i=1}^N$. Let $\bm{u} = \big(u(\mathbf{x}_1),\ldots,u(\mathbf{x}_{N})\big)^{\T} \in\mathbb{R}^{N}$, $L_{a,\epsilon} u(X)=\big(L_{a,\epsilon} u(\mathbf{x}_1),\ldots,L_{a,\epsilon} u(\mathbf{x}_{N})\big)^{\T} \in\mathbb{R}^{N}$, and $\mathcal{L}_a u(X)=\big(\mathcal{L}_a u(\mathbf{x}_1),\ldots,\mathcal{L}_a u(\mathbf{x}_{N})\big)^{\T} \in\mathbb{R}^{N}$.
By \eqref{discreteforwarderror} and \eqref{eqn:estuphi},
\begin{eqnarray}\label{eqn:delta}
\| \mathbf{L}_{a,\epsilon} \bm{\phi}_{S}-\mathbf{f}  \|_{L^2(\pi_N)}&\leq &\| \mathbf{L}_{a,\epsilon} \bm{\phi}-\mathbf{f}  \|_{L^2(\pi_N)}
\leq \| \mathbf{L}_{a,\epsilon} \bm{\phi}-\mathbf{L}_{a,\epsilon} \mathbf{u}  \|_{L^2(\pi_N)} + \| \mathbf{L}_{a,\epsilon} \mathbf{u} - \mathcal{L}_a u(X)\|_{L^2(\pi_N)}  \nonumber\\
&\leq & \| \mathbf{L}_{a,\epsilon}\|_2 \|\bm{\phi}-\mathbf{u}\|_{L^2(\pi_N)} +  \mathcal{O} \left(\epsilon,\Big(\frac{\log(N)}{N}\Big)^{\frac{1}{2}} \epsilon^{-2-\frac{d}{4}},\Big(\frac{\log(N)}{N}\Big)^{\frac{1}{2}}\epsilon^{-\frac{1}{2}-\frac{d}{4}}\right)\nonumber\\
&= &  \| \mathbf{L}_{a,\epsilon}\|_2 \mathcal{O}(m^{-8/(d\ln(n))}L^{-8/(d\ln(n))}) +  \mathcal{O} \left(\epsilon,\Big(\frac{\log(N)}{N}\Big)^{\frac{1}{2}} \epsilon^{-2-\frac{d}{4}},\Big(\frac{\log(N)}{N}\Big)^{\frac{1}{2}}\epsilon^{-\frac{1}{2}-\frac{d}{4}}\right),
\label{eqn:deltaest0}
\end{eqnarray}
where, for simplicity, we suppressed the functional  dependence on $\mathbf{x}$ in the second error term above.

Recall that $\mathbf{L}_{a,\epsilon} =-\mathbf{a}+ \mathbf{L}_{\epsilon}$, where $\mathbf{a}$ is a diagonal matrix with diagonal entries $0<a_{\min}\leq a(x_i)\leq a_{\max}$ and $\mathbf{L}_{\epsilon}\mathbf{1}=\mathbf{0}$. By definition, $\mathbf{L}_{\epsilon}$ is diagonally dominant with diagonal negative entries and non-diagonal positive entries. This means
\[
\|\mathbf{L}_{a,\epsilon}\|_\infty = \max_{1\leq j\leq N} \Big\{ a(\mathbf{x}_j)-\mathbf{L}_{\epsilon,jj}+ \sum_{i\neq j} \mathbf{L}_{\epsilon,ij}  \Big\} = \max_{1\leq j\leq N} \Big\{ a(\mathbf{x}_j)-2\mathbf{L}_{\epsilon,jj} \Big\} = a_{\max}+ C\epsilon^{-1},
\]
for some constant $C$ that depends on $\|\kappa\|_{\infty}$.
Therefore, $\|\mathbf{L}_{a,\epsilon}\|_{2}\leq N^{1/2}\|\mathbf{L}_{a,\epsilon}\|_{\infty} \leq C N^{1/2}\epsilon^{-1}$. Plugging this to \eqref{eqn:deltaest0}, we obtain
\begin{equation}\label{eqn:deltaest}
\| \mathbf{L}_{a,\epsilon} \bm{\phi}_{S}-\mathbf{f}  \|_{L^2(\pi_N)} = \mathcal{O} \left(\epsilon,\Big(\frac{\log(N)}{N}\Big)^{\frac{1}{2}} \epsilon^{-2-\frac{d}{4}},\Big(\frac{\log(N)}{N}\Big)^{\frac{1}{2}}\epsilon^{-\frac{1}{2}-\frac{d}{4}},N^{1/2}\epsilon^{-1} m^{-8/(d\ln(n))}L^{-8/(d\ln(n))}\right),
\end{equation}
as $\epsilon\to 0$ after $N\to\infty$, and $m$ or $L\to\infty$. This concludes the upper bound for the best possible empirical loss.

\begin{proof}[Proof of Theorem~\ref{thm:pconv}.] Now we derive the overall parametrization error considering DM and NN parametrization together to prove Theorem \ref{thm:pconv}. Since $\mathbf{L}_\epsilon$ is an unnormalized discrete Graph-Laplacian matrix that is semi-negative definite, for $a\geq a_{\min}>0$, one can see that,
\BEA
\langle \bm{\xi},\mathbf{L}_{a,\epsilon}\bm{\xi}\rangle_{L^2(\pi_N)} = \langle \bm{\xi},\mathbf{L}_\epsilon\bm{\xi}\rangle_{L^2(\pi_N)} - \langle\bm{\xi},\mathbf{a}\bm{\xi} \rangle_{L^2(\pi_N)}
\leq -a_{\min} \|\bm{\xi}\|^2_{L^2(\pi_N)},\notag
\EEA
for any $\bm{\xi} \in L^2(\pi_N)$. Letting $\bm{\xi}=\mathbf{u}-\bm{\phi}_S$, we have,
\begin{eqnarray}
a_{\min} \|\mathbf{u}-\bm{\phi}_{S}\|^2_{L^2(\pi_N)} &\leq& -\langle \mathbf{u}-\bm{\phi}_{S},\mathbf{L}_{a,\epsilon}(\mathbf{u}-\bm{\phi}_{S})\rangle_{L^2(\pi_N)} \notag \leq \| \mathbf{L}_{a,\epsilon}(\mathbf{u}-\bm{\phi}_{S})\|_{L^2(\pi_N)}\|\mathbf{u}- \bm{\phi}_{S}\|_{L^2(\pi_N)} \notag
\end{eqnarray}
and with probability higher than $1-N^{-2}$,
\BEA
\|\mathbf{u}-\bm{\phi}_{S}\|_{L^2(\pi_N)} &\leq& \frac{1}{a_{\min}} \|\mathbf{L}_{a,\epsilon} (\mathbf{u}-\bm{\phi}_{S})\|_{L^2(\pi_N)} \notag \\
 &\leq& \frac{1}{a_{\min}} \Big(\|\mathbf{L}_{a,\epsilon} \mathbf{u}-\mathbf{f}\|_{L^2(\pi_N)} + \| \mathbf{L}_{a,\epsilon} \bm{\phi}_{S}-\mathbf{f}  \|_{L^2(\pi_N)}\Big),\nonumber\\
&\leq &  \frac{1}{a_{\min}} \Big( \|\mathbf{L}_{a,\epsilon}\mathbf{u} - \mathcal{L}_au(X) \|^2_{L^2(\pi_N)} + \| \mathbf{L}_{a,\epsilon} \bm{\phi}_{S}-\mathbf{f}  \|_{L^2(\pi_N)}\Big),\nonumber\\
&=& \mathcal{O} \left(\epsilon,\Big(\frac{\log(N)}{N}\Big)^{\frac{1}{2}} \epsilon^{-2-\frac{d}{4}},\Big(\frac{\log(N)}{N}\Big)^{\frac{1}{2}}\epsilon^{-\frac{1}{2}-\frac{d}{4}},N^{1/2}\epsilon^{-1} m^{-8/(d\ln(n))}L^{-8/(d\ln(n))}\right),\label{consistency}
\EEA
as $\epsilon\to 0$ after $N\to\infty$, and $m$ or $L\to\infty$. Here, we have used \eqref{discreteforwarderror} and \eqref{eqn:deltaest} in the last equality above. Since the inequality in \eqref{eqn:appd} is valid uniformly, together with \eqref{DML2consistency}, we achieve the convergence of the limit 
%in \eqref{limitingperror} 
and the proof is completed.
\end{proof}

\rev{We should remark that if $\mathbf{L}_{a,\epsilon} \bm{\phi}_{S} = \mathbf{f}$ is satisfied pointwise (e.g., solution induced by the classical finite-difference approximation), then the residual error in \eqref{eqn:deltaest} is zero and the proof is nothing but the classical Lax-equivalence argument of consistency in \eqref{discreteforwarderror} and stability implies convergence of the solution.}

\subsection{Optimization Error} In the parametrization error analysis, we have assumed that a global minimizer of the empirical loss minimization in \eqref{eqn:DMNN1} is achievable. In practice, a numerical optimization algorithm is used to solve this optimization problem and the numerical minimizer might not be equal to a global minimizer. Therefore, it is important to investigate the numerical convergence of an optimization algorithm to a global minimizer. The neural tangent kernel analysis originated in \cite{Arthur18} and further developed in \cite{finegrained,pmlr-v97-allen-zhu19a} has been proposed to analyze the global convergence of GD and SGD for the least-squares empirical loss function in \eqref{eqn:DMNN1}. In the situation of solving PDEs, the global convergence analysis of optimization algorithms is vastly open. In \cite{Luo2020}, it was shown that GD can identify a global minimizer of the least-squares optimization for solving second-order linear PDEs with two-layer NNs under the assumption of over-parametrization.
{ In this paper, we will extend this result in the context where the differential operator in the loss function is replaced by a discrete and approximate operator, $ \mathbf{L}_{a,\epsilon}$, applied to the NN solution (e.g., see \eqref{eqn:DMNN1} and \eqref{eqn:DMNN2}).
}
Furthermore, the activation function considered here is ReLU$^r$, which is more general than the activation function in \cite{Luo2020}.
The extension to deeper NNs follows the analysis in \cite{pmlr-v97-allen-zhu19a} and is left as future work.

To establish a general theorem for various applications, we consider the following emprical risk $\RS(\vtheta)$ with an over-parametrized two-layer NN optimized by GD:
 \begin{equation}\label{eqn:emloss4}
 \RS(\vtheta):=
    \frac{1}{2N} (\mA\phi(X;\vtheta)-{f}(X))^{\T}(\mA\phi(X;\vtheta)-{f}(X)),
\end{equation}
where $X:=\{\vx_i\}_{i=1}^N$ is a given set of samples in  $[0,1]^n$ {of an arbitrary distribution ($\pi$ in our specific application)};  $\mA$ is a given matrix of size $N\times N$; and the two-layer NN used here is constructed as $\phi(\vx;\vtheta)=\sum_{k=1}^\width a_k\sigma(\vw_k^\T\vx)$, where for $k\in[\width]:=\{1,\ldots,m\}$, $a_k\in\sR$, $\vw_k\in\sR^n$, $\vtheta=\mathrm{vec}\{a_k,\vw_k\}_{k=1}^\width$, and $\sigma(x)=\text{ReLU}^r(x)$, i.e., the $r$-th power of the ReLU activation function. Note that the matrix $\mA$ in our specific application is $\mathbf{L}_{a,\epsilon}$ for \eqref{eqn:DMNN1}; $\mA$ is a block-diagonal matrix for \eqref{eqn:DMNN2} with one block for the differential equation and another block for the boundary condition; $\mA$ is also a block-diagonal matrix when a regularization term $\|\bm{\phi}_{\bm{\theta}}\|_{L^2(\pi_N)}^2$ is applied to either \eqref{eqn:DMNN1} or \eqref{eqn:DMNN2}. In the remainder of this section, we use the notation $\mA$ since the result holds in general under some assumption discussed below. Without loss of generality, throughout our analysis, we also assume $|f|\leq 1$, since the PDE defined on a compact domain can be normalized.

Adopting the neuron tangent kernel point of view \cite{Arthur18}, in the case of a two-layer NN with an infinite width, the corresponding kernel $k^{(a)}$ for parameters in the last linear transform is a function from $M \times M$ to $\sR$ defined by $k^{(a)}(\vx,\vx') 
    := \Exp_{\vw\sim \mathcal{N}(\vzero,\mI_n)}g^{(a)}(\vw;\vx,\vx')$, where $g^{(a)}(\vw;\vx,\vx')
    := \left[\sigma(\vw^\T\vx)\right] \cdot\left[\sigma(\vw^\T\vx')\right]$. This kernel evaluated at $N\times N$ pairs of samples leads to an $N\times N$ Gram matrix $\mK^{(a)}$ with $K^{(a)}_{ij} = k^{(a)}(\vx_i,\vx_j)$. Our analysis requires the matrix $\mK^{(a)}$ to be positive definite, which has been verified for regression problems under mild conditions on random training data $X=\{\vx_i\}_{i=1}^N$ and can be generalized to our case. Hence, we assume this together with the non-singularity of $\mA$ as follows for simplicity.
\begin{assu}\label{assump..LambdaMin}
    The smallest eigenvalue of $\mK^{(a)}$, denoted as $\lambda_S$, and the smallest eigenvalue of $\mA \mA^{\T}$, denoted as $\lambda_{\mA}$, are positive.
\end{assu}

Let $\kappa_A$ denote the condition number of $A^{\T}A$. Our main result of the global convergence of GD for \eqref{eqn:emloss4} is summarized in Theorem \ref{thm:lcr} below with a proof in Appendix \ref{appendix_c}.

\begin{theo}[Global Convergence of GD: Two-Layer NNs] \label{thm:lcr}
Let $\vtheta(0) := \mathrm{vec}{\{a_k(0), \vw_k(0)\}}_{k=1}^\width$ at the GD initialization for solving \eqref{eqn:emloss4}, where $a_k(0) \sim \mathcal{N}(0, \gamma^2)$ and $\vw_k(0) \sim \mathcal{N}(\vzero, \mI_n)$ with any $\gamma\in(0,1)$. Let  $\lambda_S$ and $\lambda_A$ be positive constants in Assumption \ref{assump..LambdaMin}. For any $\delta\in(0,1)$, if
$m=\Omega(\kappa_A\mathrm{poly}(N,r,n,\frac{1}{\delta},\frac{1}{\lambda_S}))$, then with probability at least $1-\delta$ over the random initialization $\vtheta(0)$, we have, for all $t\geq 0$, $\RS(\vtheta(t))\leq\exp\left(-\frac{\width\lambda_S \lambda_{\mA}t}{N}\right)\RS(\vtheta(0))$.
\end{theo}

For the estimate of $R_S(\vtheta(0))$, see Lemma \ref{lem2}. In particular, if $\gamma=\mathcal{O}(\frac{1}{\sqrt{\width}(\log \width)^2})$, then $R_S(\vtheta(0))=\mathcal{O}(1)$. For two-layer NNs, Theorem \ref{thm:lcr} shows that, as long as the NN width $m= \Omega(\kappa_A\text{poly}(N,r,n,\frac{1}{\delta},\frac{1}{\lambda_S}))$, GD can identify a global minimizer of the empirical risk minimization in \eqref{eqn:emloss4}. For a quantitative description of $\mathcal{O}(\kappa_A\text{poly}(N,r,n,\frac{1}{\delta},\frac{1}{\lambda_S}))$, see \eqref{eqn:mmm} in Appendix \ref{appendix_c}. In the case of FNNs with $L$ layers, following the proof in \cite{pmlr-v97-allen-zhu19a}, one can show the global convergence of GD when $m=\Omega(\kappa_A\text{poly}(L,N,r,n,\frac{1}{\delta},\frac{1}{\lambda_S}))$, which is left as future work. 

\section{Numerical Examples}
\label{sec:num}
We numerically demonstrate the effectiveness and practicability of our proposed NN-based PDE solver on six examples, ranging from simple manifolds to unknown rough surfaces. In the numerical experiments on simple manifolds, we will compare the accuracy of NN and DM (or VBDM) solutions not only on training data but also on new data points. \rev{Since DM (or VBDM) only approximates the PDE solution evaluated on the training data, we will consider the classical Nystr\"om extension \cite{nystrom:30} as a means to interpolate the approximate PDE solutions on new data points.} We will show that the NN solution provides robustly more accurate generalization on new data points with lower complexity once the training is completed. To demonstrate the advantage of the manifold assumption, we demonstrate the ability of the NN-based solver to deal with PDE on manifolds of high co-dimension by validating it on a three-dimensional manifold embedded in $\mathbb{R}^{12}$. We will see that the NN-based solver can obtain a more accurate solution than DM. Finally, we will verify the performance on unknown (and rough) surfaces. In these numerical experiments, unlike the existing works reported in \cite{Shankar2014RBFFD,liang2012geometric,piret2012orthogonal}, we do not smooth the surfaces. To avoid singularities induced by the original data set in these rough surfaces, we resample the data points using the Marching Cubes algorithm \cite{lorensen1987marching} that is available through Meshlab \cite{meshlab}. Since the analytical solutions are unknown in this configuration, we will compare the NN solutions with the finite element method (FEM) solutions obtained from the FELICITY FEM MATLAB toolbox~\cite{walker2018felicity}. 

The remainder of this section is organized as follows. In Section~\ref{sec:num1}, we give an overview of the implementation detail of the neural-network regression. We will supplement this section with a list of detailed parameters used in each numerical example in Appendix~\ref{appendix_d}. In Section~\ref{sec:num2}, we compare the NN generalization with the Nystr\"om based interpolation method on simple manifolds. In Section~\ref{sec:num3}, we examine the accuracy of the approach on a manifold with high codimension. In Section~\ref{sec:num4}, we benchmark the numerical performances against the FEM solutions on unknown (and rough) surfaces. 

\subsection{Experimental design} \label{sec:num1} 

In each of the numerical examples below, we find NN regression solution to the elliptic PDE in \eqref{ellipticPDE} embedded in $\mathbb{R}^n$. Using the notation from previous section, given point cloud data
$X=\{\mathbf{x}_1,\ldots,\mathbf{x}_N \in \mathbb{R}^n\}$, we solve the PDE by optimizing the least-square problem given by,
\BEA
 \argmin_{\bm{\theta}} {\frac{1}{2}}\|(-\mathbf{a}+\mathbf{L}_\epsilon) \bm{\phi}_{\bm{\theta}} - \mathbf{f}\|_{L^2(\pi_N)}^2+\frac{\gamma}{2}\|\bm{\phi}_{\bm{\theta}}\|_{L^2(\pi_N)}^2 + {\frac{\lambda}{2}} \| \bm{\phi}^b_{\bm{\theta}}-\mathbf{g}\|_{L^2(\pi_{N_b})}^2.
 \label{eqn:loss4example1}
\EEA
Here, we add a regularization term with $\gamma>0$ that is needed to overcome the ill-posedness induced by $a=0$ in the no boundary case (as discussed right after \eqref{eqn:DMNN1}). For manifolds with boundary, we also set $\lambda>0$. The detailed choices of these parameters for each numerical experiments are reported in Appendix~\ref{appendix_d}.

\textbf{Devices and environments.} The experiments of DM are conducted on the workstation with 32$\times$ Intel(R) Xeon(R) CPU E5-2667 v4 @ 3.20GHz and 1 TB RAM and Matlab R2019a. The experiments of the NN-solver are conducted on the workstation with 16$\times$ Intel(R) Xeon(R) Gold 5122 CPU @ 3.60GHz and 93G RAM using Pytorch 1.0 and 1$\times$ Tesla V100.

\textbf{Implementation detail for neural-network regression.} We summarize notations and list the hyperparameter setting for each numerical example in Appendix~\ref{appendix_d}. In our implementation, we use a 3-hidden-layer FNN with the same width $m$ per hidden layer and the smooth Polynomial-Sine activation function in \cite{liang2021reproducing}. Polynomial-Sine is defined as 
$
    \alpha_1 \sin(\beta_1 x)+ \alpha_2 x + \alpha_3 x^2,
$
where the parameters are trainable and initialized by normal distributions $\beta_1,\alpha_1\sim\mathcal{N}(1,0.01), \alpha_2, \alpha_3\sim\mathcal{N}(0,0.01)$ respectively. As we shall demonstrate, our proposed NN solver is not sensitive to the numerical choices (e.g., activation functions, network types, and optimization algorithms) but a more advanced algorithm design may improve the accuracy and convergence. Particularly in the high codimensional example (Section~\ref{sec:num3}), we will compare the performance of ReLU FNNs and ReLU$^3$ FNNs trained by the gradient descent method and the performance of the Polynomial-Sine FNNs trained with the Adam optimizer~\cite{Kingma2014AdamAM}. Though the ReLU or ReLU$^3$ FNNs with the gradient descent method enjoy theoretical guarantees in our analysis, the Polynomial-Sine FNNs with Adam can generalize better, especially if the PDE solutions are in the class of sine or cosine functions. Numerically, we will employ the Polynomial-Sine activation and Adam optimizer in almost all of our examples and report it separately otherwise. In Adam, we use an initial learning rate of 0.01 for $T$ iterations. The learning rate follows cosine decay with the increasing training iterations, i.e., the learning rate decays by multiplying a factor $0.5( \cos(\frac{\pi t}{T}) + 1)$, where $t$ is the current iteration. We set the batch size to be the total number of training points unless otherwise stated. The NN results reported in this section are averaged over 5 independent experiments. 

\textbf{Time complexity of evaluating NN solution.} Since some numerical experiments will compare the accuracy of the solutions (generalization) on new data points, we document the time complexity of the NN-based function evaluation on a new datum. First, we note that the total of addition and multiplication of a linear transformation from $\mathbb{R}^A$ to $\mathbb{R}^B$ in a hidden layer of NN is $(A+(A-1)+1)\times B=2AB$, where $(A+(A-1)+1)$ represents the amount of computation required to compute a neuron, and the first $A$ is the amount of multiplication, $(A-1)$ is the amount of addition, and $1$ is from adding a bias. As introduced in Section~\ref{sec:fnn}, we use a $L-$layer FNN with a uniform width $m$. The input and output sizes are $n$ and $1$, respectively. Then the complexity of NN is $2nm+2(L-1)m^2+2m$, excluding the cost of evaluating the activation function on each layer. In our numerical examples, we use $m=\mathcal{O}(\sqrt{N})$ to facilitate training and fix $L=4$, so the time complexity of evaluating 1 datum becomes $\mathcal{O}(N)$ when $n<<N$.

\textbf{Implementation detail for DM and VBDM.} 
For efficient implementation, we use the $k-$nearest neighbor algorithm to avoid computing the distances of pair of points that are sufficiently large. The kernel bandwidth parameter, $\epsilon$, is tuned empirically for the fixed-bandwidth DM. For well-sampled data (data distributed equal angle intrinsically) in our first four examples, for large enough $k$, we found accurate estimate with $\epsilon\sim N^{-2/d}$, \rev{which is much faster than the theoretical error bound $\epsilon \sim N^{-\frac{2}{6+d}}$ obtained by balancing the first and third error bounds in Theorem~\ref{thm:pconv}}. For randomly sampled data, such as in our last two examples, we will use VBDM where $\epsilon$ is estimated by an auto-tuning method that was originally proposed in \cite{coifman2008TuningEpsilon}. Among the fixed-bandwidth DM applications, we found that we can directly use the auto-tuned $\epsilon$ for the semi-torus example. While the auto-tuned method may not produce the optimal parameter accurate estimation for fixed-bandwidth DM, this scheme is quite robust for the variable bandwidth diffusion maps (VBDM) as documented in \cite{bh:16vb}. We should also point out that in VBDM, we also need to prescribe an additional parameter $k_2< k$ to specify $\rho_0$ that is used to determine the variable bandwidth function $\rho$ in \eqref{VBDM}. We document the choice of nearest neighbor parameter $k$, the parameter $k_2$, and bandwidth parameter $\epsilon$ resulting either from hand-tuning or auto-tuning in the Tables in Appendix~\ref{appendix_d}. 

For a DM/VBDM solution, when the resulting problem is ill-posed (due to $\mathbf{a}=0$ and no boundary), the resulting matrix 
$\mathbf{L}_\epsilon$ is singular. In such a case, we report the numerical result based on applying the pseudo inverse of $\mathbf{L}_\epsilon$ to a right-hand-side vector. If the rank of the matrix $\mathbf{L}_\epsilon$ is $r$, then $\mathcal{O}(N^2r)$ complexity is required to stably compute a truncated SVD of $\mathbf{L}_\epsilon$ \cite{doi:10.1137/090771806}, which can be used to form a truncated SVD of the pseudo inverse of $\mathbf{L}_\epsilon$. Then it takes $\mathcal{O}(Nr)$ complexity to apply the pseudo inverse of $\mathbf{L}_\epsilon$ through its truncated SVD. The $\mathcal{O}(N^2r)$ complexity can be further reduced to $\mathcal{O}(Nr^2)$ at the price of not stable to entrywise outliers in $\mathbf{L}_\epsilon$ \cite{Engquist2009AFD,YANG2012148}.

\textbf{Nystr\"om-based interpolation.}
We consider a Nystr\"om-based interpolation \cite{nystrom:30} for comparing the accuracy of the estimates on new data points. For completeness of the discussion, we provide an algorithmic overview of the basic idea of Nystr\"om interpolation to extend symmetric eigenvectors to new data point and discuss its implemetation for DM. For a more rigorous discussion that covers VBDM see e.g. Section 5.2 of \cite{alexander2020operator}.  

Let $\mathbf{L}\in \mathbb{R}^{N\times N}$ be a symmetric positive definite matrix with eigenvalues $\lambda_1 \geq \lambda_2 \geq \ldots \geq \lambda_N>0$ and we denote the associated eigenvectors by $\{\mathbf{v}_j\}_{j=1,\ldots,N}$. For convenience of discussion, we define a function  $L:\mathbb{R}^n\times\mathbb{R}^n\to \mathbb{R}$ such that $L(\mathbf{x}_i,\mathbf{x}_j) = \mathbf{L}_{ij}$. Effectively, the eigenvalue problem
\[
\mathbf{L} \mathbf{v}_j = \lambda_j \mathbf{v}_j.
\]
is a discrete approximation on $\mathbf{x}\in X$ to a continuous eigenvalue problem of the following integral operator,
\[
\int_{M} L(\mathbf{x},\mathbf{y}) v_j(\mathbf{y})\,dV(\mathbf{y}) = \lambda_j v_j(\mathbf{x}),
\]
assuming that the data set $X$ is uniformly distributed for simplicity of discussion. Effectively, we approximate the function value $v_j(\mathbf{x}_i)$ with the $i$th component of the eigenvector $[\mathbf{v}_j]_i$. The main idea of Nystr\"om extension is to approximate the function value $v_j(\mathbf{x})$ on new data point $\mathbf{x}$ (not necessarily belong to the training data set $X$) using the equation above as follows,
\[
v_j(\mathbf{x}) = \frac{1}{\lambda_j}\int_{M} L(\mathbf{x},\mathbf{y}) v_j(\mathbf{y})\,dV(\mathbf{y}). 
\]
Numerically, this can be achieved by the following approximation,
\BEA
v_j(\mathbf{x}) \approx \frac{1}{\lambda_j}\sum_{i=1}^N L(\mathbf{x},\mathbf{x}_i) [\mathbf{v}_j]_i.\label{nystrom}
\EEA

For our problem, we are interested in applying this extension formula on an integral operator corresponding to the DM/VBDM asymptotic expansion that converges to the Laplace-Beltrami operator. In such a case, eigenfunctions of the Laplace-Beltrami operator can be approximated by eigenvectors of $\mathbf{L}^{ns} = \mathbf{D}^{-1}\mathbf{W}$, where $\mathbf{W}$ is defined as in \eqref{weightmatrix} with $\kappa=1$ and $Q=1$ for uniform data, and $\mathbf{D}$ is a diagonal matrix with diagonal entries $\mathbf{D}_{ii}= \sum_{j=1}^N\mathbf{W}_{ij}$. For convenience of the discussion below, we define the kernel function $W$ such that $\mathbf{W}_{ij} = W(\mathbf{x}_i,\mathbf{x}_j)$ as in \eqref{weightmatrix}.

We can now employ the Nystr\"om extension to interpolate eigenfunctions of $L$ associated with the symmetric discretization $\mathbf{L}=\mathbf{D}^{1/2}\mathbf{L}^{ns}\mathbf{D}^{-1/2} = \mathbf{D}^{-1/2} \mathbf{W}\mathbf{D}^{-1/2}$. Since $\bm{\phi}_j := \mathbf{D}^{-1/2} \mathbf{v}_j$ is an eigenvector of $\mathbf{L}^{ns}$ corresponding to the same eigenvalue $\lambda_j$ of $\mathbf{L}$, it is clear that given the function value $v_j(\mathbf{x})$ on a new datum $\bf{x}$, an extension attained by the formula in \eqref{nystrom} with $L(\mathbf{x},\mathbf{x}_i):= D(\mathbf{x})^{-1/2}W(\mathbf{x},\mathbf{x}_i)D(\mathbf{x}_i)^{-1/2}$ where $D(\mathbf{x}) = \sum_{j=1}^N W(\mathbf{x},\mathbf{x}_j)$, then the corresponding eigenfunction of the Laplace-Beltrami operator is approximated by $\phi_j(\mathbf{x}):=D(\mathbf{x})^{-1/2}v_j(\mathbf{x})$. Similar idea applies to the Nystr\"om extension corresponding to the VBDM asymptotic expansion with a slightly more complicated formula. For manifolds with homogeneous Dirichlet boundary conditions, we employ the extension using the truncated diffusion maps \cite{peoples2021spectral}.

Given the ability to approximate the function values  $\{\phi_j(\mathbf{x})\}_{j=1}^{J}$ for any new datum $\mathbf{x}$, we evaluate the DM solution for the PDE problem by interpolating the projected solution to the leading $J$ eigenfunctions of the Laplace-Beltrami operator as follows. Let $\mathbf{u}\in \mathbb{R}^N$ denotes the solution of the PDE attained by solving $(-\mathbf{a}+\mathbf{L}_\epsilon) \mathbf{u} = \mathbf{f}$ (either through direct or pseudo-inversion). Then we define the extension as,
\BEA
u(\mathbf{x}) \approx \sum_{j=1}^J \mathbf{u}^\top \bm{\phi}_j \phi_j(\mathbf{x}),\label{interpolate_u}
\EEA
where $J$ is chosen sufficiently large such that the residual error is comparable to the interpolation error. 

{\bf Time complexity of the Nystr\"om extension.} To employ the Nystr\"om extension, we need to first attain eigenvectors $\{\bm{\phi}_j\}_{j=1}^J$. Using the standard algorithm, one can attain the first $J$ eigenvectors with time complexity of $\mathcal{O}(Nk)+ \mathcal{O}(JkN)$, where $k$ is the number of nearest neighbors; the first complexity term corresponds to the arithmetic in the construction of $\mathbf{L}$; and the second complexity term corresponds to the arithmetic in solving the first $J$ eigenvalue problem. Subsequently, the arithmetic cost in attaining the coefficients in the expansion in \eqref{interpolate_u} is $\mathcal{O}(NJ)$ accounting for $J$ inner products of vectors in $\mathbb{R}^N$, which is negligible compared to the computations in solving the eigenvectors. Once these computations are performed, we can now evaluate any data point with additional time complexity due to the arithmetic $\mathcal{O}(NJ)$ per evaluation. This arithmetic cost includes the time complexity of $\mathcal{O}(N)$ in the extension formula in \eqref{nystrom}, not accounting for the cost of evaluating the kernel function $W$, and the use of $J$ eigenfunctions in \eqref{interpolate_u} is of $\mathcal{O}(JN)$.

\rev{Comparing the time complexities for evaluating the NN solution and Nystr\" om interpolation, notice that the latter has an additional order-$J$, which has to be empirically determined in practice. As we pointed out above, the additional order-$J$ appears in the estimation of leading $J$ eigenvectors and the \revv{Nystr\" om} extension. In the remainder of this section, we shall see that in some examples, setting $J$ to be of order 10 or 100 is sufficient to achieve the accuracy obtained by evaluating the NN solution. However, in some examples, we found that the accuracy of the NN solutions cannot be achieved by Nystr\" om extension of DM solutions even with $J=600$.}

\subsection{Comparison between Nystr\"om extension and NN solution}\label{sec:num2}

In addition to the complexity comparisons reported in the previous section, we now compare the accuracy of the NN and the Nystr\"om solutions. We will consider three instructive examples. In the first example, our goal is to validate the Nystr\"om approach on a problem when there are no residuals of the projection in \eqref{interpolate_u}. In the second example, we want to elucidate the dependence of the Nystr\"om approach on $J$ when the PDE solution does not belong to the space span by the leading eigenfunctions of the Laplace-Beltrami operator. In this case, we will encounter competing errors from the eigenvector estimation using small training data set, the choice of $J$, and interpolating error, in the overall generalization error. In the final example, we will elucidate the effect of boundary conditions on the  Nystr\"om-based interpolation method. In all of these experiments, we found that NN produces the most robust and accurate generalization.

\comment{
In this subsection, we compare the performance of the NN method and the traditional method (i.e., Nystrom) on extending solution for new data points in terms of the computational complexity and the testing accuracy. \sw{We need some description of the Nystrom method here.} We use ``$\text{Nvar}$'' to denote the number of eigenfunctions to extend the solution using Nystrom. 

We compare the computational complexity of the NN method and the Nystrom method on evaluating $1$ new data point. \textbf{NN method}. The sum of addition and multiplication of a linear transformation from $\mathbb{R}^A$ to $\mathbb{R}^B$ in the hidden layer of NN is $(A+(A-1)+1)\times B=2AB$, where $(A+(A-1)+1)$ represents the amount of computation required to compute a neuron, and the first $A$ is the amount of multiplication, $(A-1)$ is the amount of addition, and $1$ is from adding a bias. As introduced in Section~\ref{sec:fnn}, we use a $L-$layer FNN with a uniform width $m$. The input and output size are $n$ and $1$, respectively. Then the complexity of NN is $2nm+2(L-1)m^2+2m$. In our numerical examples, we use $m=\mathcal{O}(\sqrt{N})$ to facilitate training and fix $L=4$, so the complexity becomes $\mathcal{O}(N)$ when $n<<N$. \textbf{Nystrom method}. \sw{For the Nystrom method, the complexity is XXX}
}

\subsubsection{2D Sphere}
In this first example, we solve the elliptic PDE in \eqref{ellipticPDE} for $a = 0, \kappa=1$ in a two-dimensional sphere embedded in $\mathbb{R}^3$ with an embedding function defined as,
\BEA
\iota(t_1,t_2) = \begin{pmatrix}\cos t_1\sin t_2 \\ \sin t_1\sin t_2 \\ \cos t_2\end{pmatrix} \in \mathbb{R}^3 \ \text{for} \ \ \begin{matrix} 0\leq t_1\leq 2\pi\\ 0\leq t_2\leq \pi\end{matrix}.
\label{eqn:2dsphere}
\EEA
For this simple geometry, one can verify that the Laplace-Beltrami operator has leading nontrivial eigenvalue \rev{of} -2 with spherical harmonics eigenfunctions, $\phi_1(\mathbf{x}) = x_1, \phi_2(\mathbf{x}) = x_2, \phi_3(\mathbf{x}) = x_3$ for $\mathbf{x}=(x_1,x_2,x_3)$. In this first test example, we choose the 
true solution of the PDE as $u(\mathbf{x}) = x_1+x_2$ for $\mathbf{x}=(x_1,x_2,x_3)$, such that the expansion in \eqref{interpolate_u} commits no residual error. Numerically, we will set $J=20$ to ensure that even if the estimated eigenvectors corresponding to eigenvalue -2 are not exactly $x_1, x_2, x_3$ (up to an rotation), the projection still commits no residual error. 

For this example, the matrix $\mathbf{L}_\epsilon$ is constructed by VBDM. We solve the PDE with the NN method by optimizing the least-square problem~\eqref{eqn:loss4example1} with $\lambda=0$ since this manifold has no boundary. The hyperparameter setting for the different $N$ is reported in Table~\ref{tab:2dsphereconfigs} in Appendix~\ref{appendix_d}. 
Table~\ref{tab:sphere} displays the errors (in $\ell_\infty$ norm) as functions of $N$ for the VBDM solution (which we refer to as the {\it training error}), and the testing errors of the Nystr\"om-based solution projected on $J=20$ eigenfunctions and the NN method. 
Here, the {\it testing errors} are computed against the true solution on $300^2$ Gaussian Legendre quadrature nodes (which do not belong to the training data set). 
We can see that both Nystrom and NN methods achieve similar testing errors, that are also comparable to the VBDM training error. 

%Notice that we chose a relatively small number of the eigenfunctions ($\text{Nvar}=20$) to reconstruct the solution in the Nystrom method as the true solution is in the span of several leading eigenfuctions. However, as we shall see later, when the true solution is complicated, more eigenfunctions are needed to achieve a good generalization. \sw{I set FLAGDMTYPE = 1 which means the VBDM is used in Nystrom}
\begin{table}[H]
  \centering
    \begin{tabular}{ccccccc}
    \toprule
          & N   & 625	& 1225	& 2500	& 5041	& 10000 \\
    \midrule
    VBDM  & training error &  0.0352 &	0.0254 &	0.0178 &	0.0126 &	0.0090 \\
    \midrule
    Nystrom ($J=20$)  & testing error &  0.0411 &	0.0268 &	0.0199 &	0.0128 &	0.0091  \\
    \midrule
    NN  & testing error &  0.0387 &	0.0271 &	0.0193 &	0.0128 &	0.0094 \\
    \bottomrule
    \end{tabular}%
    \caption{(2D sphere) Error comparison of the PDE solutions among VBDM, Nystrom extension of the VBDM solution, and NN solution. The testing error is the $\ell_\infty$ error between the estimated solution and the true solution on $300^2$ Gauss-Legendre quadrature points. }
  \label{tab:sphere}%
\end{table}%

\subsubsection{2D Torus} 
\label{sec:2dtorus}
In this section, we consider solving the elliptic PDE in \eqref{ellipticPDE} for $a=0$ on a two-dimensional torus embedded in $\mathbb{R}^3$ with an embedding function defined as,
\BEA
\iota(t_1,t_2) = \begin{pmatrix}(2+\cos t_1)\cos t_2 \\ (2+\cos t_1)\sin t_2 \\ \sin t_1\end{pmatrix} \in \mathbb{R}^3 \ \text{for} \ \ \begin{matrix} 0\leq t_1\leq 2\pi\\ 0\leq t_2\leq 2\pi\end{matrix}.
\label{eqn:torus}
\EEA
Here $t_1,t_2$ denote the intrinsic coordinates. For this numerical experiment, we set the diffusion coefficient 
$\kappa(t_1,t_2) = 1.1 + \sin^2t_1\cos^2t_2$,
the true solution  
$u(t_1,t_2) = (\sin2 t_2-2\cos2 t_2/(2+\cos t_1))\cos t_1$,
and analytically compute the right hand side function $f$. The point cloud data $X=\{\mathbf{x}_1,\ldots,\mathbf{x}_N\}$  are well sampled (equal angles in  intrinsic coordinates $(t_1, t_2)$). We solve the PDE by optimizing the least-square problem in \eqref{eqn:loss4example1} with $\lambda=0$. The hyperparameter setting for the different $N$ is reported in Table~\ref{tab:2dconfigs} in Appendix~\ref{appendix_d}.

\begin{figure}
\centering
\includegraphics[width=.5\textwidth]{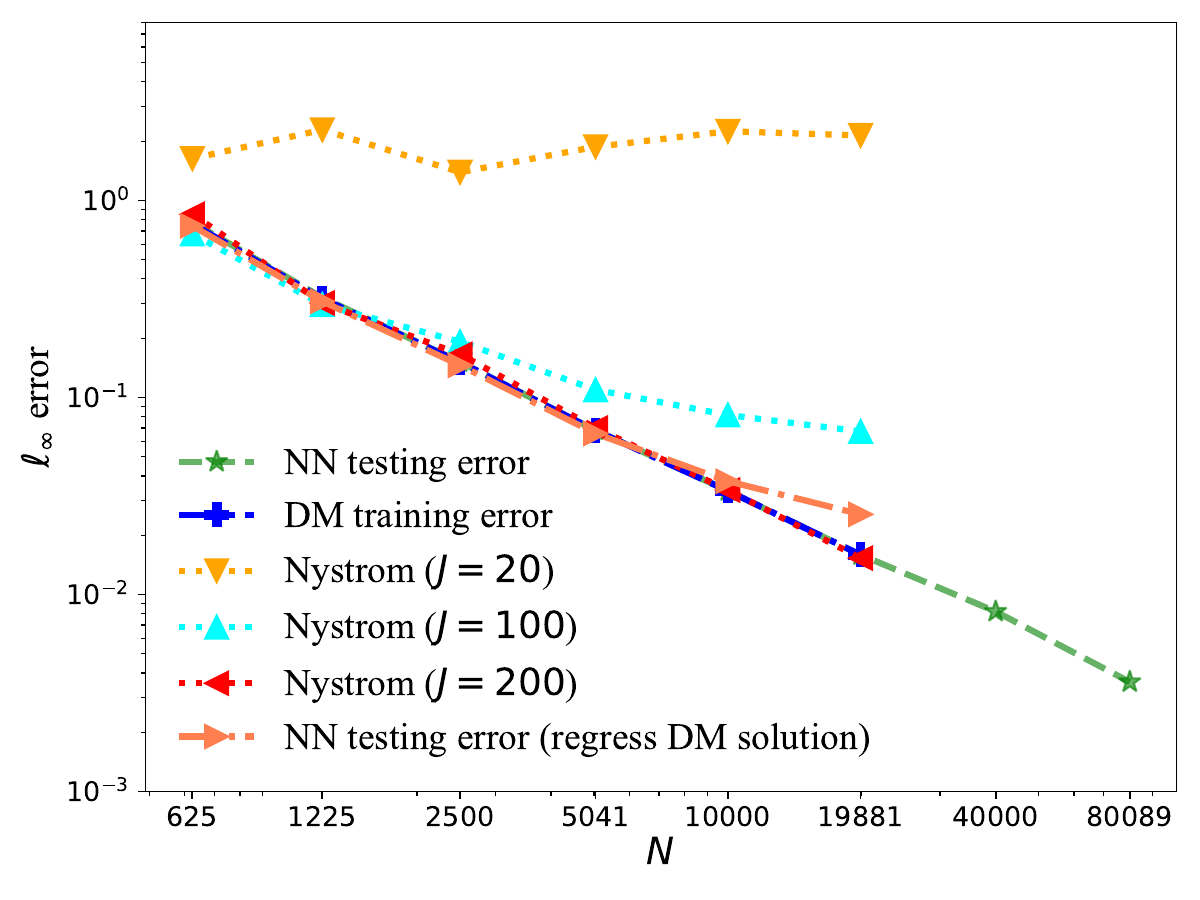}
\caption{{\bf 2D-torus:} \revv{The comparison of the $\ell_\infty$ errors as functions of the number of training points $N$ for DM and NN on a 2D torus embedded in $\mathbb{R}^3$. The results labeled as ``(regress DM solution)" represent conducting NN regression directly on the DM solution.}} %{\color{red}To Senwei: Please change ``Inverse'' error to ``training error''; change $\textup{Nvar}$ to $J$. I wonder what is the rate of the green curve? It looks about $N^{-1}-N^{-.8}$.} \sw{The figure is updated, and the rate is about $N^{-1.09}$ which is computed by linear regression $\log_{10}(\text{NN test error})\approx -1.0878\log_{10}(N)+2.888$}} 
\label{fig:2derror}
\end{figure}

 Figure~\ref{fig:2derror} shows the errors with the growth of training data $N$ on the 2D full-torus. As in the previous example, the testing error is defined with $\ell_\infty$ error on $300^2$ Gauss-Legendre quadrature points that are not in the training data set. One can see that the NN testing error is consistent with the DM training error, indicating a good generalization of new data points. We note that the accuracy of Nystr\"om method heavily depends on the number of eigenfunctions used. When the number of eigenfunctions is small ($J=20$), the errors on the testing points are large for all $N$. When the number of eigenfunctions is increased to $J=100$, Nystrom has a good generalization when $N$ is small ($N\leq 1225$). For $N\geq 2500$, the testing errors are smaller compared with $J=20$ but they are still large compared with the DM training error. If we continue increasing $J$ to 200, the testing errors of the large $N$ ($N\geq 1225$) become consistent with the DM error. These results suggest that more eigenfunctions are able to suppress the residual error induced by the finite projection in \eqref{interpolate_u}. At the same time, when $J$ is increased from 100 to 200, the error corresponds to $N=625$ increases as well. This is not surprising since the estimation of eigenvectors corresponding to large eigenvalues are not so accurate when the number of data points is small. From these results, we conclude that while the Nystr\"om method can be competitive with the NN solution, it depends crucially on the number of modes in \eqref{interpolate_u}, which ultimately also depends on the number of training data.
 
 In Figure~\ref{fig:2derror}, we also show the testing error with NN solution for larger $N$, where the error continues to decrease with rate $N^{-1}$, which confirms with the choice of $\epsilon\sim N^{-1}$ for this well-sampled data (see Table~\ref{tab:2dconfigs}). To obtain this solution with $N=80089$, mini-batch training is used to minimize \eqref{eqn:loss4example1} since $\mathbf{L}_\epsilon \in \mathbb{R}^{80089\times 80089}$ is too large to compute in GPU directly. In our implementation, at each time, $8000$ rows from $\mathbf{L}_\epsilon$ are randomly sampled and the submatrix $\mathbf{L}_\text{sub}$ of size $8000\times 80089$ substitutes $\mathbf{L}_\epsilon$ in the loss~\eqref{eqn:loss4example1}. Then the network parameters $\bm{\theta}$ are updated for 100 iterations at each time. We repeat this procedure for 80 times. As for the DM, we do not pursue the results since the computational cost in solving the linear problem
 $\mathbf{L}_\epsilon \mathbf{u} = \mathbf{f}$
with singular $\mathbf{L}_\epsilon$ is very expensive as it involves pseudo-inversion algorithm. 
 
\rev{To be more specific, we report the clock-time and memory cost for training DM and NN solutions in Figure~\ref{fig:timemem}. The clock-time for the pseudo-inverse of DM grows rapidly with the increasing $N$ while that of NN remains much smaller when $N< 80089$. The rapid increase of NN clock-time for the case $N=80089$ is attributed to the time consuming of the retrievals of the submatrix $\mathbf{L}_\text{sub}$ from $\mathbf{L}_\epsilon$. Memory in Figure~\ref{fig:timemem} refers to the sum of GPU and RAM. When conducting pseudo-inverse in the DM, the Matlab occupies RAM to load the full matrix and do the computation. The NN-based solver utilizes RAM to load the sparse matrix and uses GPU memory for model training. From Figure~\ref{fig:timemem}, the NN solver has larger memory consumption than DM when $N$ is small but the NN solver uses less memory than DM when $N$ is large. Since we set the batch size to be the total number of training points for all cases except $N=80089$, the memory consumption will significantly decrease if a mini-batch is used in Adam.}

\revv{To obtain the NN solution to the manifold PDE, one could consider an \textit{alternative method} consisting of two steps: First, use an inversion to solve the linear system, and then train a NN to regress the DM solution. While the alternative method may produce testing error similar to the DM method and the proposed NN method, we do not pursue it because the alternative method will cause a much higher cost even for small $N$ due to its two-step approach. Figure~\ref{fig:timemem} supplement the running time and memory usage of the NN training portion for the alternative method (labeled as ``(regress DM solution)"). One can see that NN regression on the DM solutions does not significantly reduce running time or memory cost. Furthermore, Figure~\ref{fig:2derror} provides testing error results for the alternative method. The testing error of the alternative method is slightly higher than the others, likely because we use the same training configuration for the alternative method. A better training configuration may maintain a testing error close to the DM training error. The NN training portion in the alternative method has comparable training time to the NN method but note that the alternative method includes two sequential steps. This means the total training time (DM solution time + NN training time to regress the DM solution) is much higher. Also, the memory requirement for the alternative method is the maximum of the NN regression memory and DM memory, which is more costly in terms of memory.}

\begin{figure}[htbp]
\centering     %%% not \center
\subfigure[clock-time]{\label{fig:clocktime}\includegraphics[width=.41\textwidth]{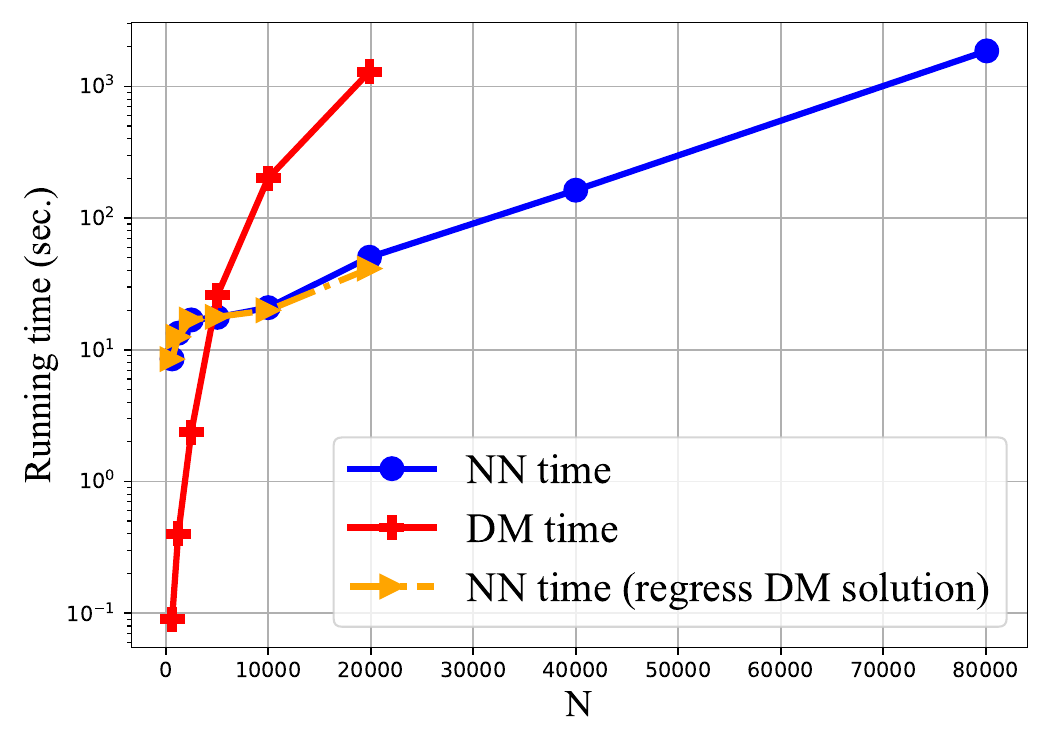}}
\subfigure[memory]{\label{fig:memory}\includegraphics[width=.41\textwidth]{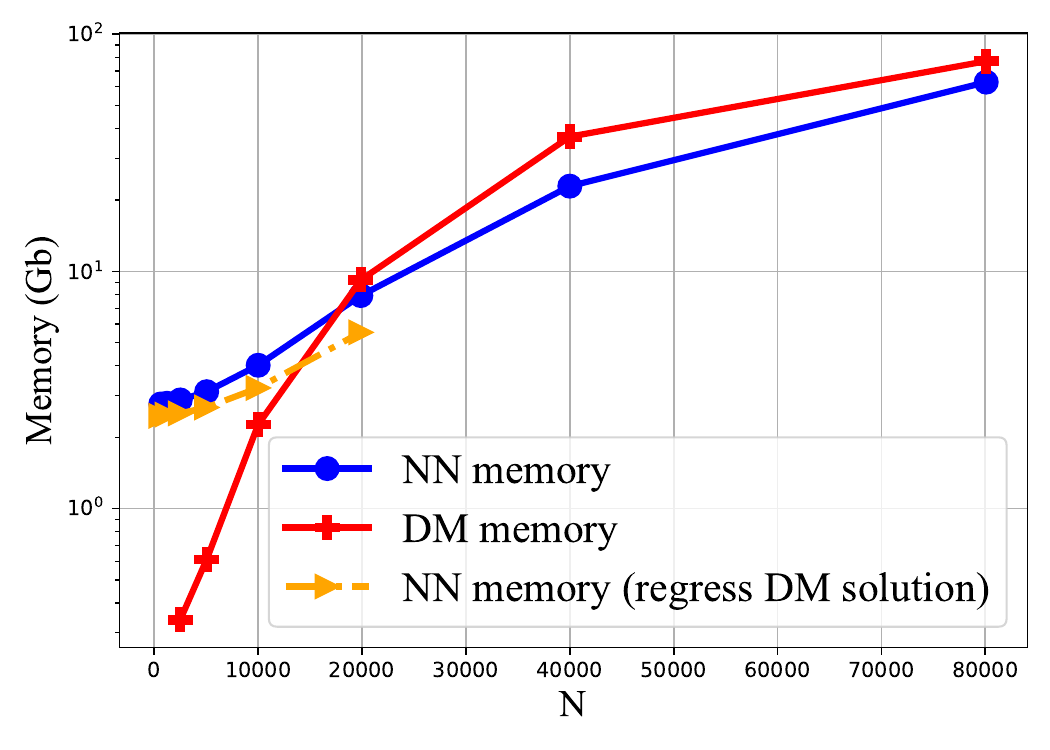}}
\vspace{-0.2cm}
\caption{\revv{The comparison of clock-time, memory for DM and NN solvers on 2D torus embedded in $\mathbb{R}^3$. In the DM case, we also report the require memory, estimated by the system process-manager, to solve the problem for large $N$, which we did not pursue due to the excessive wall-clock time. The results labeled as ``(regress DM solution)" represent conducting NN regression directly on the DM solution.}}
\label{fig:timemem}
\end{figure}

\comment{
\begin{table}[htbp]
  \centering
  \begin{adjustbox}{width=\textwidth,center}
    \begin{tabular}{cp{9.165em}cccccccc}
    \toprule
          & N  & 625   & 1225  & 2500  & 5041  & 10000 & 19881 & 40000 & 80089 \\
    \midrule
    DM & inverse error & 0.7606  & 0.3175  & 0.1511  & 0.0680  & 0.0336  & 0.0159  & N/A   & N/A \\
    \midrule
    Nystrom ($\text{Nvar}=20$) & testing error & 1.6312 &	2.2833	& 1.3910	& 1.8716	& 2.2438
  &  2.1381 & N/A   & N/A \\
    Nystrom ($\text{Nvar}=100$) & testing error & 0.6765 &	0.2960 & 0.1917 &	0.1094	& 0.0815 & 0.0673  &  N/A  &  N/A \\
    Nystrom ($\text{Nvar}=200$) & testing error & 0.8510 &	0.3017	&0.1656& 	0.0700&	0.0339
   &  0.0153  &  N/A  &  N/A \\
    \midrule
    NN & testing error & 0.7764  & 0.3213   & 0.1516   & 0.0682   & 0.0336   & 0.0159   & 0.0082   & 0.0036  \\
    \bottomrule
    \end{tabular}%
	\end{adjustbox}
  \caption{(2D Torus) Error comparison of the PDE solutions among DM, Nystrom extension of the DM solution, and NN. The testing error is the $\ell_\infty$ error between the estimated solution and the true solution on $300^2$ Gauss-Legendre quadrature points. N/A for DM indicates that the result is not computable. Since the DM solution is not available due to computational issue of pseudo-inverse, we can not extend the solution on new data points with Nystrom. }
  \label{tab:2dtorus_vs_nystrom}%
\end{table}%
}

\subsubsection{Semi-torus} 
We consider solving the PDE in \eqref{ellipticPDE} with $a=0$ on a two-dimensional semi-torus $M$ with a Dirichlet boundary condition. Here, the embedding function is the same as in \eqref{eqn:torus}
except that the range of $t_2$ is changed to $0\leq t_2\leq\pi$. While $\kappa$ is chosen to be the same as before, we set $u(t_1,t_2) =\cos(t_1)\sin(2t_2)$ to be the solution with homogeneous Dirichlet boundary condition on $t_2=0,\pi$. We obtain the NN-based solution by solving the optimization problem in \eqref{eqn:loss4example1} with $\lambda=5$ and no regularization $\gamma=0$. The other hyperparameter setting can be found in Table~\ref{tab:2dconfigsbc} in Appendix~\ref{appendix_d}. 

\begin{figure}
\centering
\includegraphics[width=.5\textwidth]{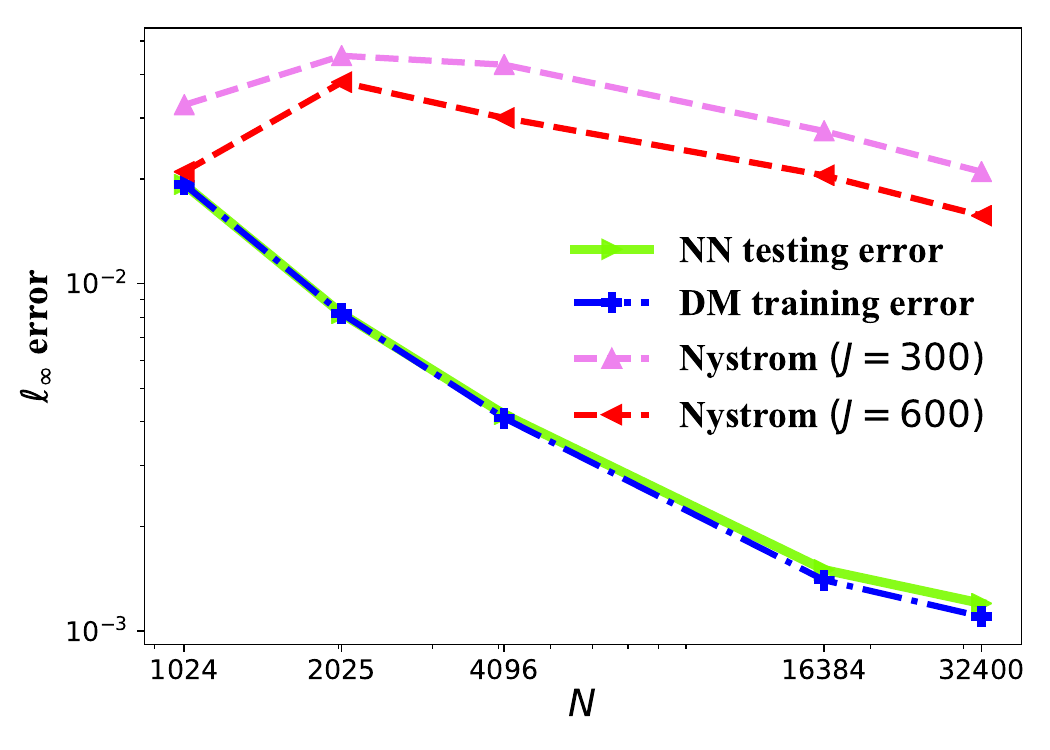}
\caption{{\bf Semi-torus:} The comparison of the $\ell_\infty$ errors as functions of the number of training points $N$ for DM and NN on a semitorus embedded in $\mathbb{R}^3$.} %{\color{red}To Senwei: Please change ``Inverse'' error to ``training error''; change $\textup{Nvar}$ to $J$.}\sw{Updated}}
\label{fig:semitoruserror}
\end{figure}

 Figure~\ref{fig:semitoruserror} shows the errors as functions of training data $N$ on the semi-torus. In this example, the testing error is computed with $\ell_\infty$ norm over $10,000$ uniformly sampled data. Notice that the error of the Nystr\"om is quite large even for $J=600$. In this case, this large error is attributed to the slow convergence of the estimation of eigenvectors near the boundary (see \cite{peoples2021spectral} for the detailed rate). To highlight this issue, we show the error comparison of the Nystrom (J=600) and NN solutions for $N=32400$ in Figure~\ref{fig:semitorus_nystrom}. We can see that the error of Nystrom concentrates near the boundary while the boundary error of the NN solution is small. 
\comment{
\begin{table}
  \centering
  \begin{adjustbox}{width=0.8\textwidth,center}
    \begin{tabular}{clccccc}
    \toprule
          & $N$ & 1024  & 2025  & 4096  & 16384 & 32400 \\
    \midrule
    GPDM & inverse error &0.0193 &	0.0082 &	0.0041 &	0.0014 &	0.0011 
 \\
    \midrule
    Nystrom ($\text{Nvar}=300$) & testing error & 0.0327&	0.0453&	0.0427	&0.0275	&0.0210
 \\
    Nystrom ($\text{Nvar}=600$) & testing error & 0.0210 &	0.0380	& 0.0300 &	0.0205 &	0.0157 
 \\
    \midrule
    
    NN & testing error & 0.0193	& 0.0082& 	0.0042& 	0.0015& 	0.0012
  \\
    \bottomrule
    \end{tabular}%
	\end{adjustbox}
  \caption{(Semi-torus) Error comparison of the PDE solutions among GPDM, the Nystrom extension of the GPDM solution, and NN. The testing error is the $\ell_\infty$ error between the estimated solution and the true solution on $100^2$ uniformly sampled points. }
  \label{tab:semitorus_vs_nystrom}%
\end{table}%
}
\begin{figure}[htbp]
\centering     %%% not \center
\subfigure[True solution]{\label{fig:semitrue}\includegraphics[width=.31\textwidth]{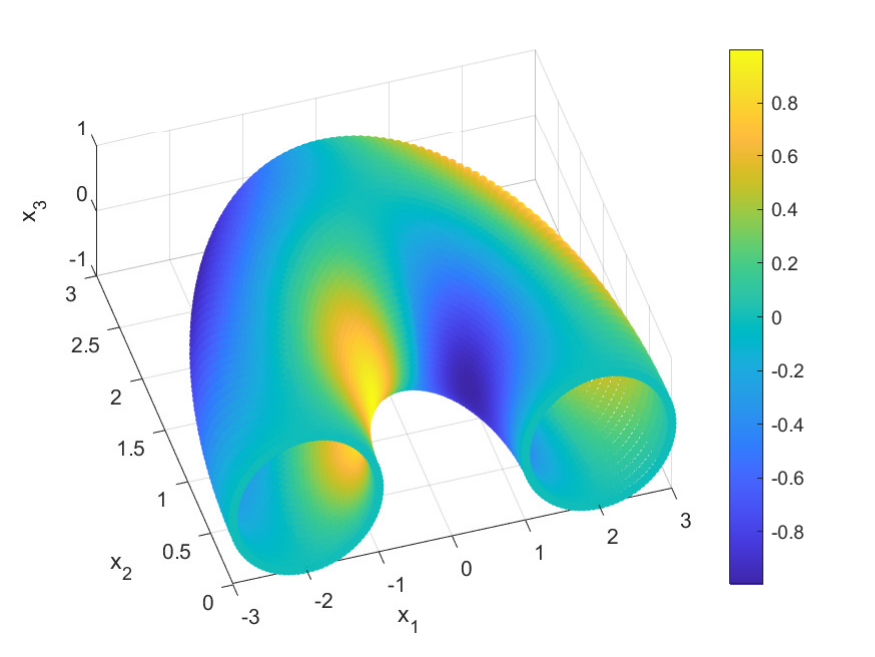}}
\subfigure[Difference between Nystrom \& truth]{\label{fig:seminystrom}\includegraphics[width=.31\textwidth]{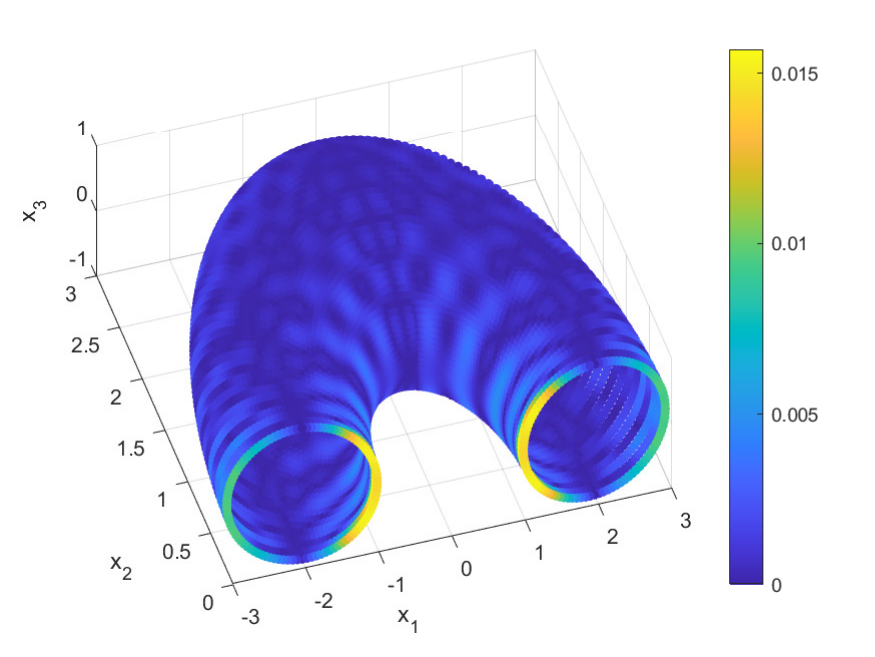}}
\subfigure[Difference between NN \& truth]{\label{fig:seminn}\includegraphics[width=.31\textwidth]{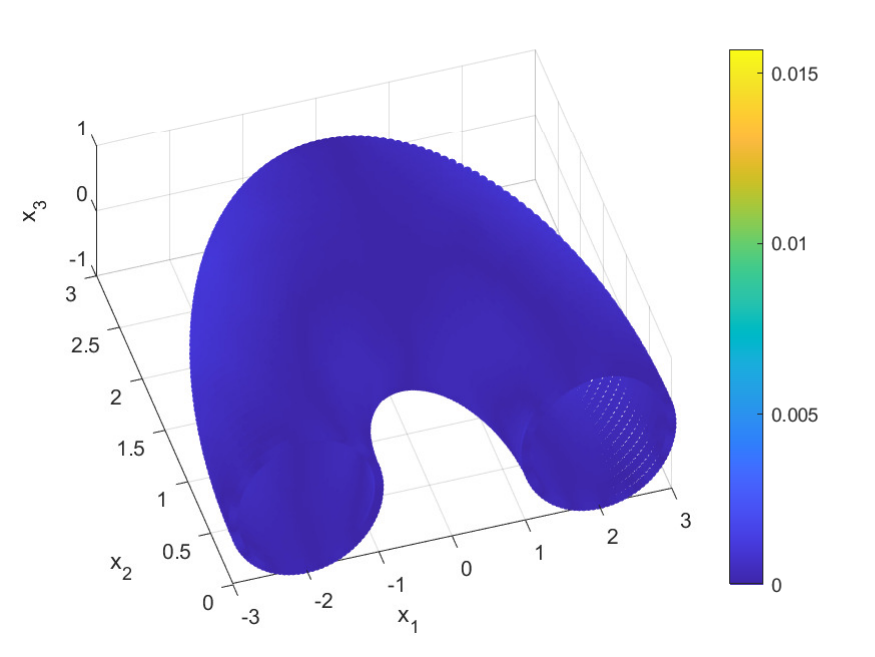}}
\caption{Comparison of the PDE solutions among the Nystrom extension ($J=600$) and NN on the semi-torus example ($N=32400$). (a) True solution. (b) Absolute difference between Nystrom and true solutions. (c) Absolute difference NN and true solutions. }
\label{fig:semitorus_nystrom}
\end{figure}

\textbf{Section summary:} Based on the numerical comparisons in the above three examples, we found that the NN solution produces more robust and accurate generalizations compared to the Nystr\"om-based interpolation. The  Nystr\"om-based method requires a large number of eigenfunctions to suppress the residual for general PDE solution. This requirement translates to the need for a large number of training data for accurate estimation of eigenvectors corresponding to large eigenvalues of the Laplace-Beltrami operator, which also means more computational power is needed in solving the corresponding eigenvalue problems. Based on this finding, we will not present the Nystr\"om-based interpolation in the next three examples. We will use the DM (or VBDM) and NN training errors as benchmarks for the NN testing error.

\subsection{A 3D Manifold of High Co-Dimension}\label{sec:num3}
In this section, we demonstrate the performance of the NN regression solution on a simple manifold with high co-dimension. For this example, we will also compare results from the Polynomial-Sine to other simpler activation functions. Specifically, we consider the elliptic PDE in \eqref{ellipticPDE} with $a = 0, \kappa=1$ on 
a closed manifold $M$, embedded by $\iota:M\hookrightarrow \mathbb{R}^{12}$, defined through the following embedding function, \[\iota(t_1,t_2,t_3):=\big(\sin(t_1),\cos(t_1),\sin(2t_1),\cos(2t_1), \sin(t_2),\cos(t_2),\sin(2t_2),\cos(2t_2),\sin(t_3),\cos(t_3),\sin(2t_3),\cos(2t_3)\big),\] for $t_1,t_2,t_3\in[0,2\pi)$. We manufacture the right hand data $f$ by setting the true solution to be
$u(t_1,t_2,t_3)=\sin t_{1}\cos t_{2}\sin 2t_{3}.$ The training points $\{\mathbf{x}_1,\ldots,\mathbf{x}_N\}$ are well sampled (equal angles in the intrinsic coordinates $(t_1,t_2,t_3)$). We solve the PDE problem by minimizing the loss function~\eqref{eqn:loss4example1} with $\lambda=0$ since the problem has no boundary. The hyperparameter setting of DM and NN is presented in Table~\ref{tab:3dconfigs} in Appendix~\ref{appendix_d}. 

Figure~\ref{fig:3derror} displays the error for DM and NN, where the latter is the solution we obtained using the Polynomial-Sine activation function. Here, the testing error refers to the $\ell_\infty$ error on $80^3$ Gauss-Legendre quadrature points obtained from the intrinsic coordinates $(t_1,t_2,t_3)$. From Figure~\ref{fig:3derror}, one can see that the NN solver produces convergent solutions. Besides, when $N$ is large (e.g., $N=4096$ or $12167$), NN obtains a slightly more accurate solution than DM. We report the training and testing errors for different activation functions and optimizers in Table \ref{tab:activations}. The training errors of Polynomial-Sine with Adam are comparable to that of ReLU and ReLU$^3$ with gradient descent, but the Polynomial-Sine FNN optimized by Adam obtains the lowest testing error for most $N$. This finding is not surprising since the true solution is a product of sine and cosine functions. In the next section, we will consider problems where the true solutions are unknown and we will see that other type of activation functions produce more accurate results.

%Based on this result, we present only the Polynomial-Sine and Adam in most of the examples we present in this paper.

\begin{table}[htbp]
  \centering
  \begin{adjustbox}{width=0.8\textwidth,center}
    \begin{tabular}{cccccccc}
    \toprule
        Activation  & Optimizer & N     & 512   & 1331  & 4096  & 12167 & 24389 \\
    \midrule
    \multirow{2}[2]{*}{Polynomial-Sine} & \multirow{2}[2]{*}{Adam} & training error & 0.2665 & 0.1148 & 0.0297 & 0.0066 & 0.0023 \\
          &       & testing error & 0.2715 & 0.1346 & 0.0302 & 0.0069 & 0.0024 \\
    \midrule
    \multirow{2}[2]{*}{ReLU$^3$} & \multirow{2}[2]{*}{gradient descent} & training error & 0.2624 & 0.1137 & 0.0309 & 0.0058 & 0.0025 \\
          &       & testing error & 0.2994 & 0.1977 & 0.0425 & 0.0147 & 0.0103 \\
    \midrule
    \multirow{2}[2]{*}{ReLU} & \multirow{2}[2]{*}{gradient descent} & training error & 0.2637 & 0.1145 & 0.032 & 0.0066 & 0.0016 \\
          &       & testing error & 0.4673 & 0.198 & 0.0326 & 0.0069 & 0.0016 \\
    \bottomrule
    \end{tabular}%
    \end{adjustbox}
  \caption{The comparison of the $\ell_\infty$ errors for different activation functions and optimizers on the 3D manifold embedded in $\mathbb{R}^{12}$.}
  \label{tab:activations}%
\end{table}%

\begin{figure}
\centering
\includegraphics[width=.5\textwidth]{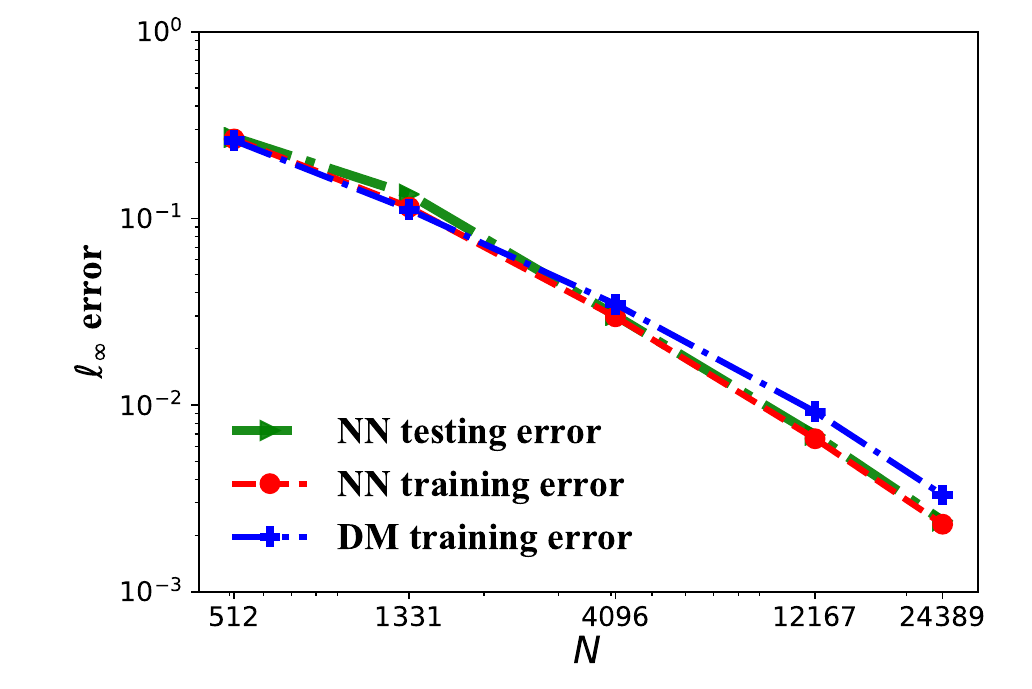}
\caption{The comparison of the $\ell_\infty$ errors as functions of the number of training points $N$ for DM and NN on the 3D manifold embedded in $\mathbb{R}^{12}$. }
\label{fig:3derror}
\end{figure}

\comment{
\subsection{Example 3: 2D Semi-Torus with Dirichlet Conditions}
\label{sec:semitorus}
We consider solving the PDE in \eqref{ellipticPDE} with $a=0$ on a two-dimensional semi-torus $M$ with a Dirichlet boundary condition. Here, the embedding function is the same as in \eqref{eqn:torus}
except that the range of $t_2$ is changed to $0\leq t_2\leq\pi$. Also, $\kappa$ and the true solution $u$ are defined as in Example~1 except for $0\leq t_2\leq \pi$, and the Dirichlet boundary condition is imposed by setting $g$ to correspond to the solution $u$ at $t_2=0,\pi$ and $0\leq t_1\leq 2\pi$.

We obtain the NN-based solution by solving the optimization problem in \eqref{eqn:DMNN2} with $\lambda=5$ and $\mathbf{a}=0$. The other hyperparameter setting can be found in Table~\ref{tab:2dconfigsbc} in Appendix~\ref{appendix_d}. The results (see Figure~\ref{fig:2derrorbc}) show that our NN method works well on the equation with boundary condition. We refer the readers to Table~\ref{tab:2derrorbc} for the detailed numerical values corresponding to this figure.

\begin{figure}
\centering
\includegraphics[width=.5\textwidth]{2DerrorBC.pdf}
\caption{The comparison of the errors as functions of the number of training points for DM and NN on 2D semi-torus with Dirichlet condition.}
\label{fig:2derrorbc}
\end{figure}
}
\subsection{Unknown surfaces in $\mathbb{R}^{3}$}\label{sec:num4}
In this subsection, we consider solving the elliptic PDEs on the unknown surfaces in $\mathbb{R}^{3}$ with or without boundary, including the Bunny (no boundary) and Face (with boundary). The data set for the Bunny surface is from the Stanford 3D Scanning Repository~\cite{bunny}. The original data set of the Stanford Bunny comprises a triangle mesh with 34,817 vertices. Since this data set has singular regions at the bottom, we generate a new mesh of the surface using the Marching Cubes algorithm~\cite{lorensen1987marching} that is available through the  Meshlab~\cite{meshlab}. We should point out that the Marching Cubes algorithm does not smooth the surface. Instead, we will use the vertices of the new mesh as sample points to avoid singularity induced by the original data set. To generate various sizes of sample points, we subsequently apply the Quadric edge algorithm \cite{garland1997surface} to simplify the mesh obtained by the Marching Cubes algorithm (which is a surface of 34,594 vertices) into 4,326, 8,650, and 17,299 vertices. The surface Face is obtained from Keenan Crane’s 3D repository~\cite{keenan}. The original Face consists of 17,157 vertices and we generate the datasets of 2,185, 4,340, and 8,261 vertices by employing the Quadric edge algorithm \cite{garland1997surface} in Meshlab~\cite{meshlab}. Besides these data resampling, we also multiply the coordinate of Face by 10 to avoid a solution close to zero function. Since the analytic solution is not accessible, we benchmark our solution against the finite element method (FEM) solution, obtained from the FELICITY FEM Matlab toolbox~\cite{walker2018felicity}. The estimator $\mathbf{L}_\epsilon$ of the Laplace–Beltrami operator is constructed by the variable bandwidth diffusion mapping (VBDM). When training with less than the largest available data set (34,594 vertices for the Bunny and 17,157 vertices for the Face), we report the testing error, which is based on $\ell_\infty$ norm over the largest data set. So, the testing errors reported in the following two examples also include the interpolation error since they are the maximum errors of the solutions on both the training and testing data sets.

\subsubsection{The Stanford Bunny} We denote the point cloud (vertices) of this surface as,  $\vx=(x_1,x_2,x_3)\in M \subset \mathbb{R}^3$. Here, we consider solving the PDE in \eqref{ellipticPDE} with $a = -0.2, \kappa=1$ and $f=x_1+x_2+x_3$. We solve the PDE problem by minimizing the loss function~\eqref{eqn:loss4example1} with $\gamma=0$ and $\lambda=0$ since the problem has nontrivial $a$ and no boundary. The hyperparameter setting of DM and NN is presented in Table~\ref{tab:bunnyparams} in Appendix~\ref{appendix_d}.

Table~\ref{tab:bunny} shows the error comparison of the solutions between VBDM and NN methods with the growth of the data point size. We can see that the NN solver gives convergent solutions as VBDM. \rev{We should also point out that the similar training error rates of NN and VBDM reported in this table suggest that the error is dominated by the VBDM approximation. The rate, however, is faster than the theoretical error bound $\epsilon\sim N^{-\frac{2}{6+d}}$ attained by balancing the first and third bounds in Theorem~\ref{thm:pconv} since the numerical experiment is conducted with empirically chosen parameters $\epsilon$ and $k$ in k-nearest neighbors algorithm to optimize the accuracy of the solutions (see Appendix C for the chosen parameters).} Besides, the NN solutions suggest a good generalization as the testing error is consistent with its training error. Figure~\ref{fig:bunny} displays the visualization for the case of $N=17299$ and $N=4326$ in Table~\ref{tab:bunny}. 

\begin{table}[H]
  \centering
    \begin{tabular}{cccccc}
    \toprule
          & N     & 4326  & 8650  & 17299 & 34594 \\
    \midrule
    VBDM  & training error & 0.2047 & 0.1017 & 0.0514 & 0.0283 \\
    \midrule
    \multirow{2}[2]{*}{NN} & training error & 0.2048 & 0.1019 & 0.0518 & 0.0296 \\
          & testing error & 0.2050 & 0.1023 & 0.0517 & -- \\
    \bottomrule
    \end{tabular}%
    \caption{Error comparison of the PDE solutions between VBDM and NN. The testing error is the $\ell_\infty$ error between the FEM and the NN solutions on the 34,594 vertices.}
  \label{tab:bunny}%
\end{table}%

\begin{figure}[H]
\centering     %%% not \center
\subfigure[FEM solution ($N=17\text{k}$)]{\includegraphics[width=.242\textwidth]{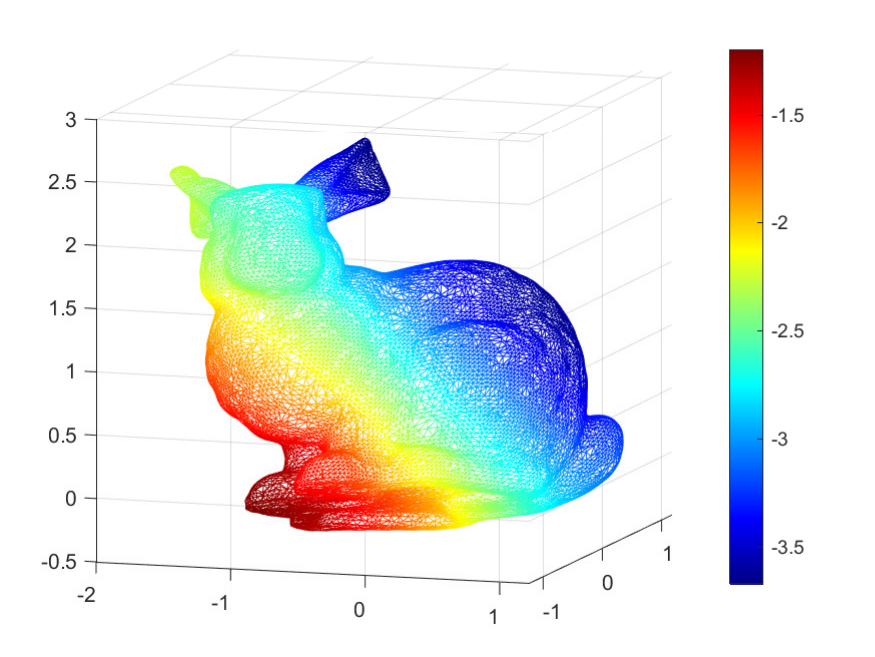}}
\subfigure[|FEM - VBDM| ($N=17\text{k}$) ]{\includegraphics[width=.242\textwidth]{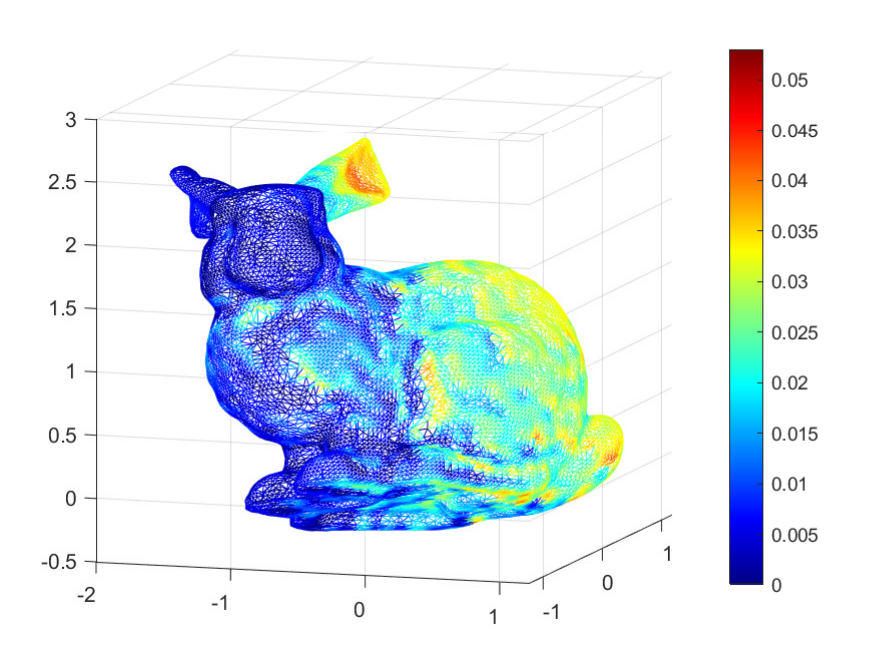}}
\subfigure[|FEM - NN| ($N=17\text{k}$)]{\includegraphics[width=.242\textwidth]{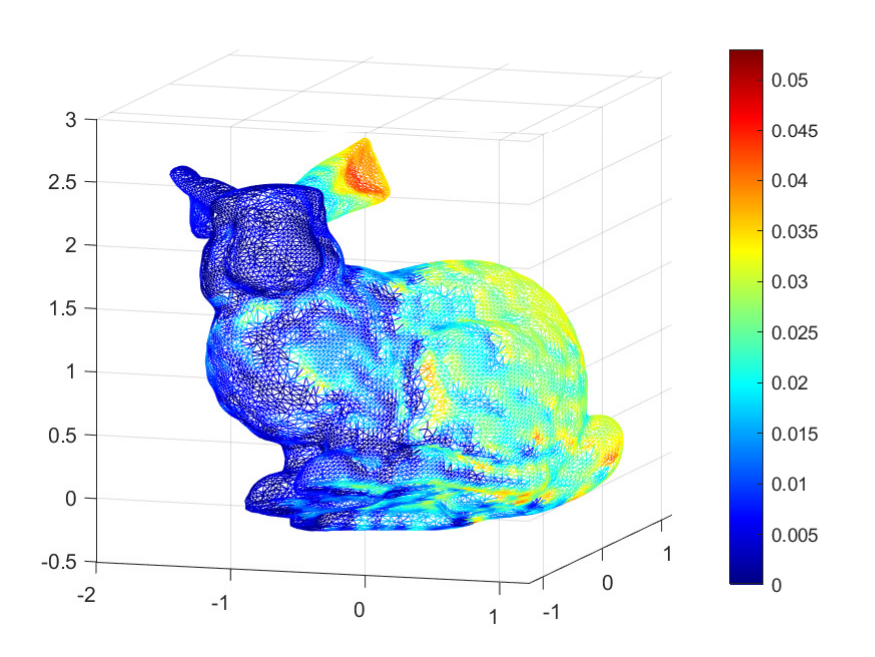}}
\subfigure[|FEM - NN| (test $N=35\text{k}$ )]{\includegraphics[width=.242\textwidth]{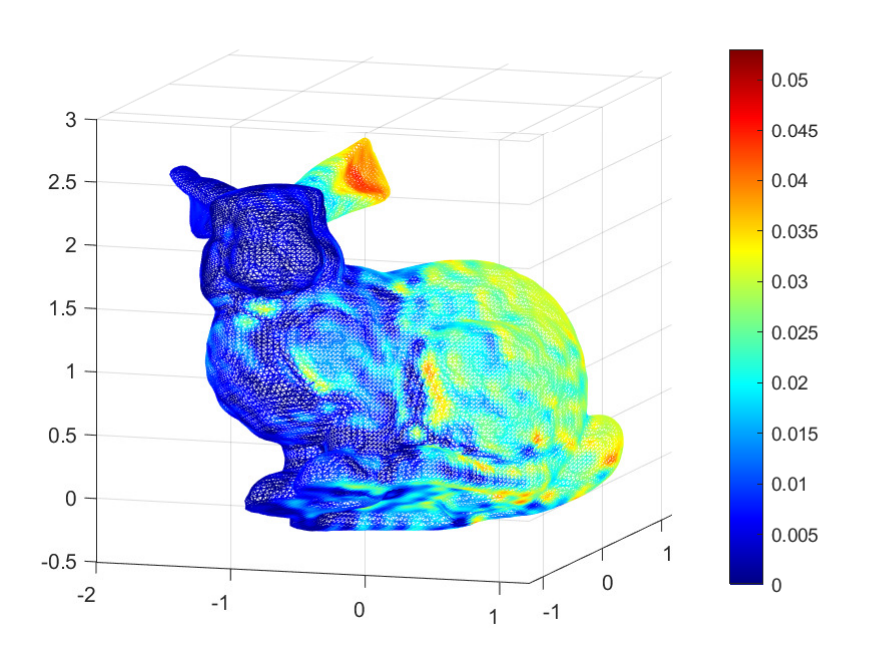}} \\
\subfigure[FEM solution ($N=4\text{k}$)]{\includegraphics[width=.242\textwidth]{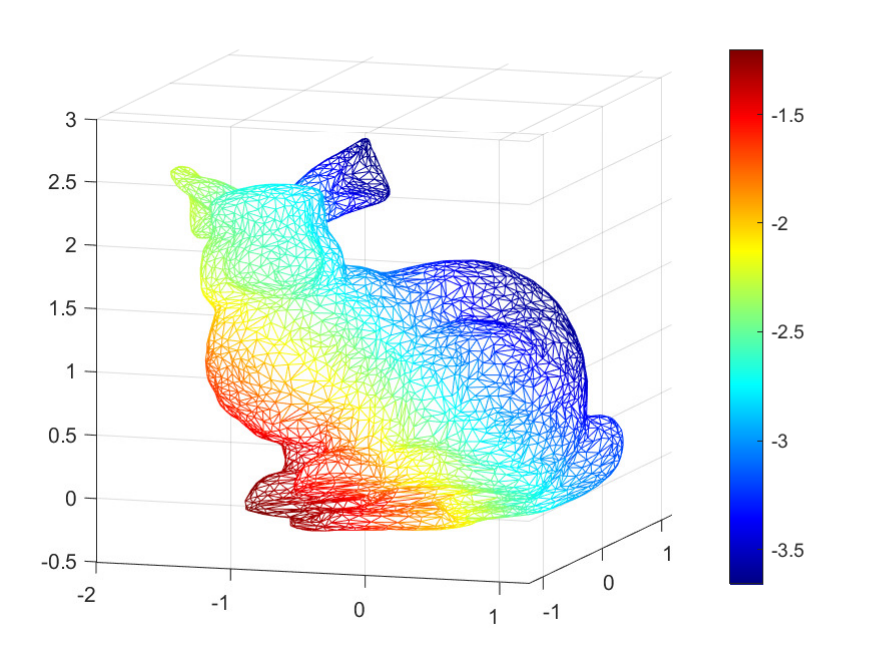}}
\subfigure[|FEM-VBDM| ($N=4\text{k}$) ]{\label{fig:bunnyvbdm}\includegraphics[width=.242\textwidth]{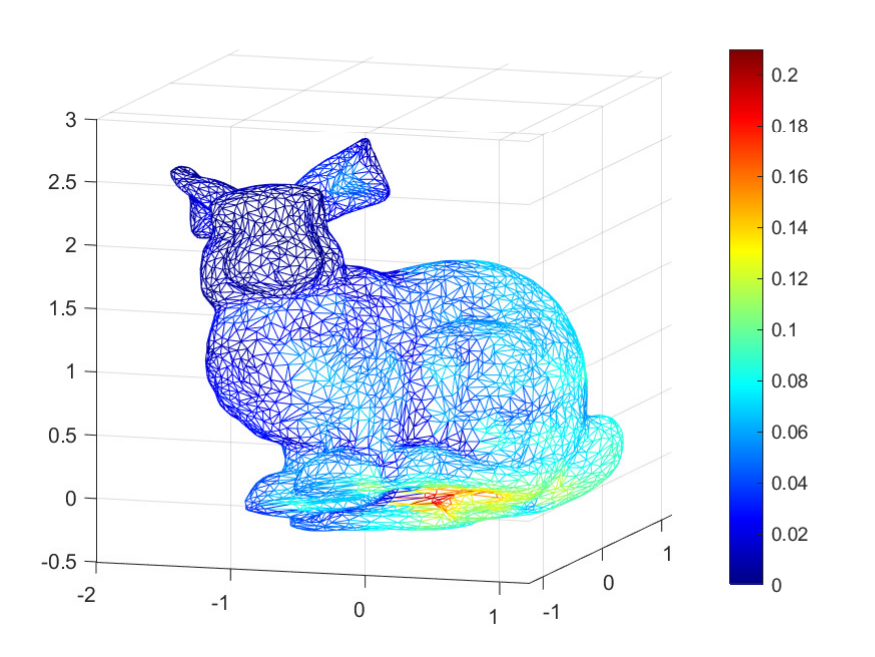}}
\subfigure[|FEM- NN| ($N=4\text{k}$)]{\label{fig:bunnynn}\includegraphics[width=.242\textwidth]{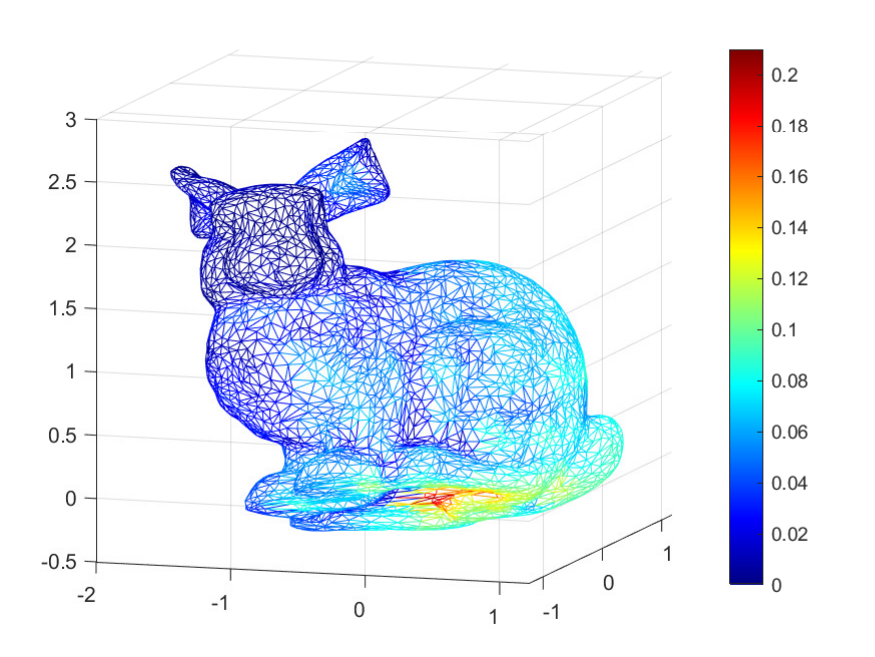}}
\subfigure[|FEM - NN| (test $N=35\text{k}$ ) ]{\label{fig:bunnynn_test}\includegraphics[width=.242\textwidth]{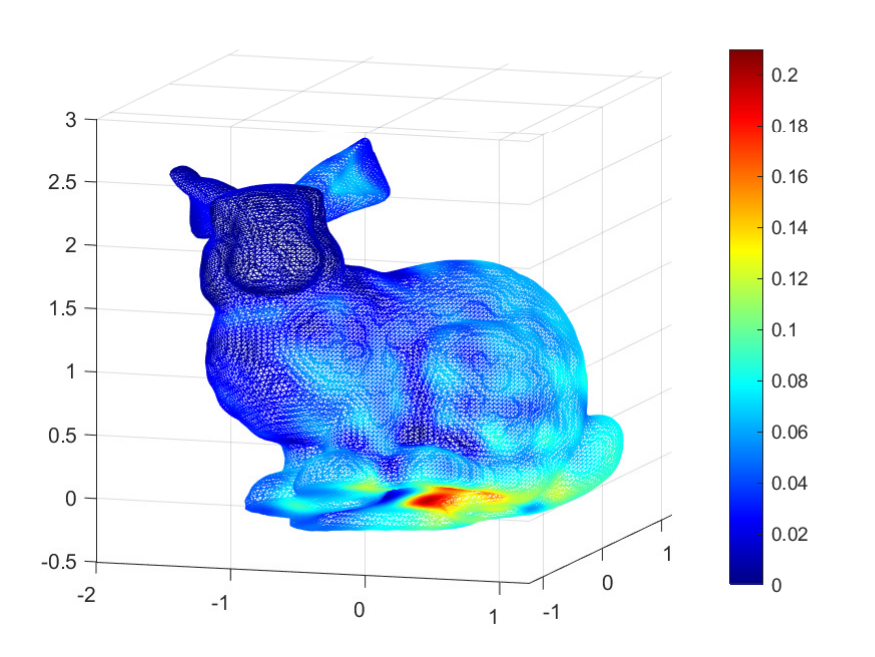}}
\caption{Comparison of the PDE solutions among FEM, VBDM and NN on the Bunny example. {\bf (a)} FEM solution with $N=17299$ training data. {\bf (b)} Absolute difference between FEM and VBDM solutions on training data set of size $N=17299$. {\bf (c)} Absolute difference between FEM and NN solutions on training data set of size $N=17299$. {\bf (d)} Absolute difference between FEM and NN solutions on $34594$ data points from the Marching Cubes algorithm, 17299 of these points were used to obtain errors in (b) and (c). {\bf (e)} FEM solution with $N=4,326$ training data. {\bf (f)} Absolute difference between FEM and VBDM solutions on training data set of size $N=4,326$. (g) Absolute difference between FEM and NN solutions on training data set of size $N=4,326$. (h) Exactly like (d) except that only 4326 of these points are used to obtain errors in (f) and (g).}
\label{fig:bunny}
\end{figure}

\subsubsection{Face example with the Dirichlet boundary condition.} In this face example, we let $a = 0, \kappa=1$ and $f=x_1+x_2+x_3$ in the PDE~\eqref{ellipticPDE}, where $\vx=(x_1,x_2,x_3)\in M \subset \mathbb{R}^3$, and we impose the Dirichlet boundary condition ($u=0$ on $\partial M$). The hyperparameter setting can be referred to Table~\ref{tab:faceparams}. In this example, we found that the constructed $\mathbf{L}_\epsilon$ is extremely ill-conditioned (the condition number is $2.0\times 10^5$ when $N=17157$), the least-squares optimization becomes difficult. To facilitate the convergence of NN training, we increase the training iteration and use a more time-efficient activation function (such as ReLU and Tanh) instead of Polynomial-Sine function. Table~\ref{tab:face} shows the error comparison of the solutions between VBDM and NN methods with the growth of the data point size. We can see that the NN solver gives the solutions with error similar to VBDM on the training data points. Besides, the testing error of NN is close to its training error. Figure~\ref{fig:face} displays the visualization of the FEM, VBDM and NN solutions. 

\begin{table}[H]
  \centering
    \begin{tabular}{cccccc}
    \toprule
          & N     & 2185  & 4340  & 8261  & 17157 \\
    \midrule
    VBDM  & training error & 0.1548 & 0.1340 & 0.0967 & 0.0708 \\
    \midrule
    \multirow{2}[2]{*}{NN} & training error & 0.1548 & 0.1340 & 0.0967 & 0.0911 \\
          & testing error & 0.1761 & 0.1429 & 0.0984 &  -- \\
    \bottomrule
    \end{tabular}%
    \caption{Error comparison of the PDE solutions between VBDM and NN. The testing error is the $\ell_\infty$ error between the FEM and the NN solutions on the 17,157 vertices.}
  \label{tab:face}%
\end{table}%

% Table generated by Excel2LaTeX from sheet 'Sheet2'

\begin{figure}[htbp]
\centering     %%% not \center
\subfigure[FEM solution ($N=8\text{k}$)]{\includegraphics[width=.24\textwidth]{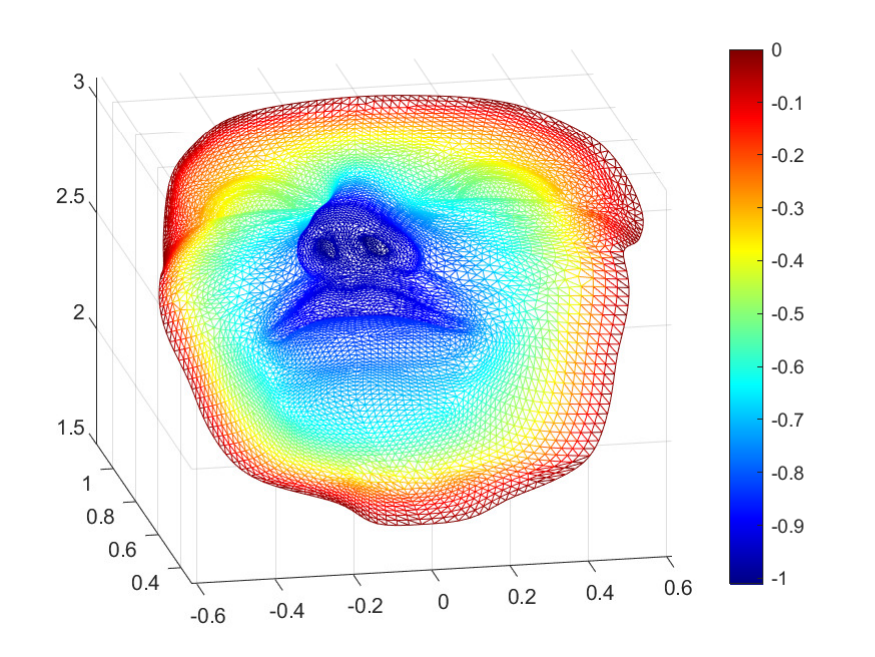}}
\subfigure[|FEM - VBDM| ($N=8\text{k}$)]{\includegraphics[width=.24\textwidth]{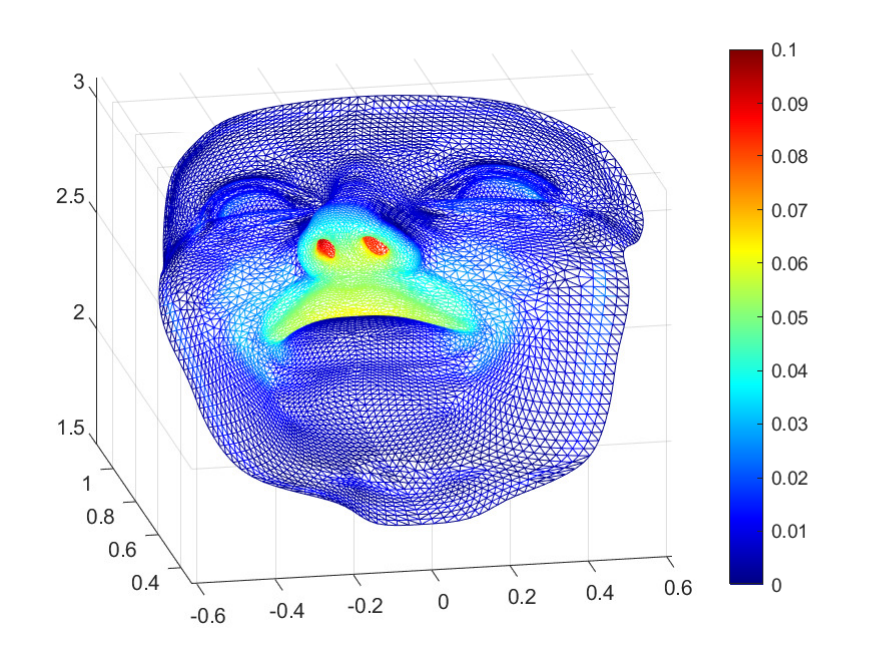}}
\subfigure[|FEM - NN| ($N=8\text{k}$)]{\includegraphics[width=.24\textwidth]{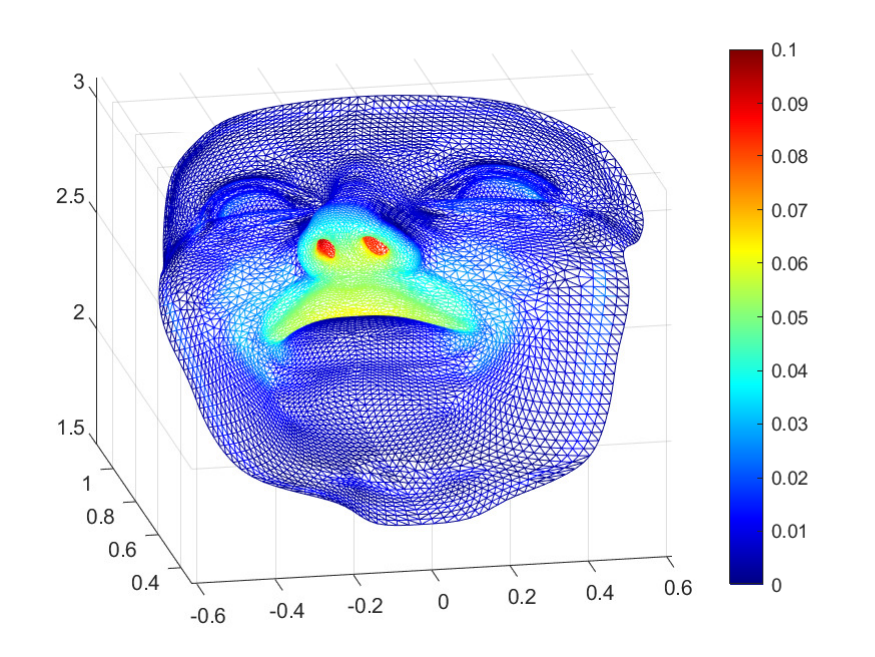}}
\subfigure[|FEM - NN| (test $N=17\text{k}$)]{\includegraphics[width=.24\textwidth]{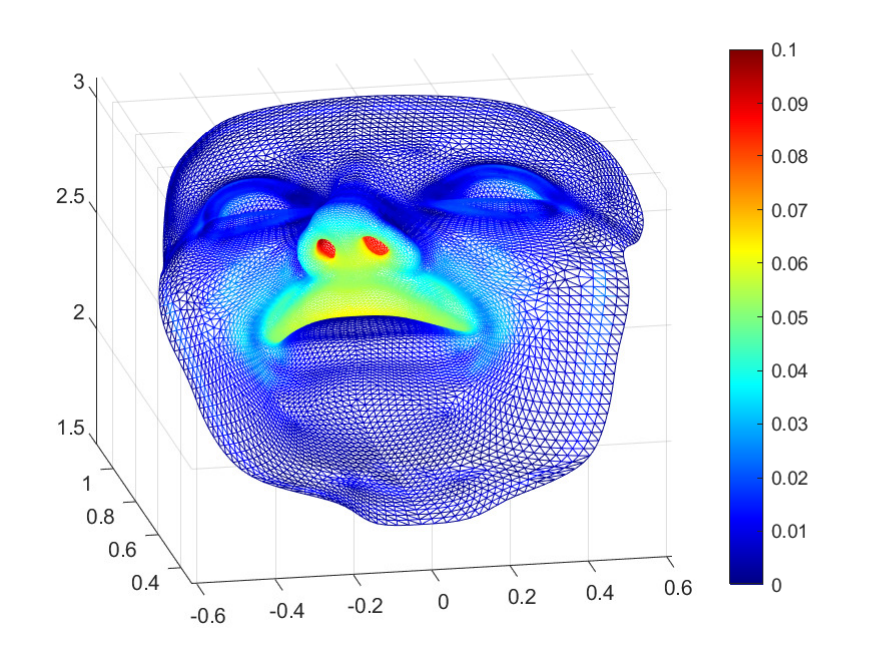}} \\
\subfigure[FEM solution ($N=2\text{k}$)]{\includegraphics[width=.24\textwidth]{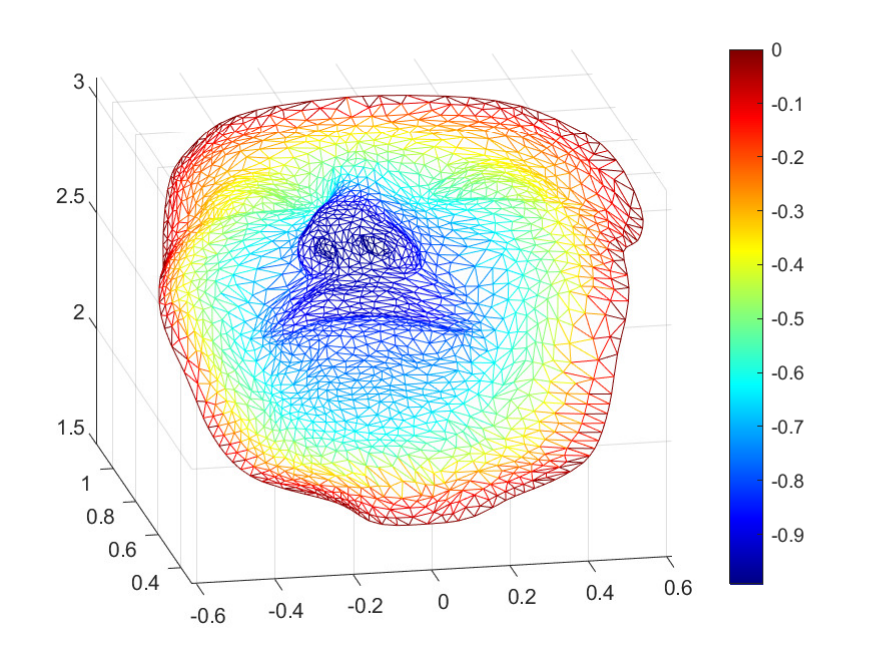}}
\subfigure[|FEM - VBDM| ($N=2\text{k}$)]{\includegraphics[width=.24\textwidth]{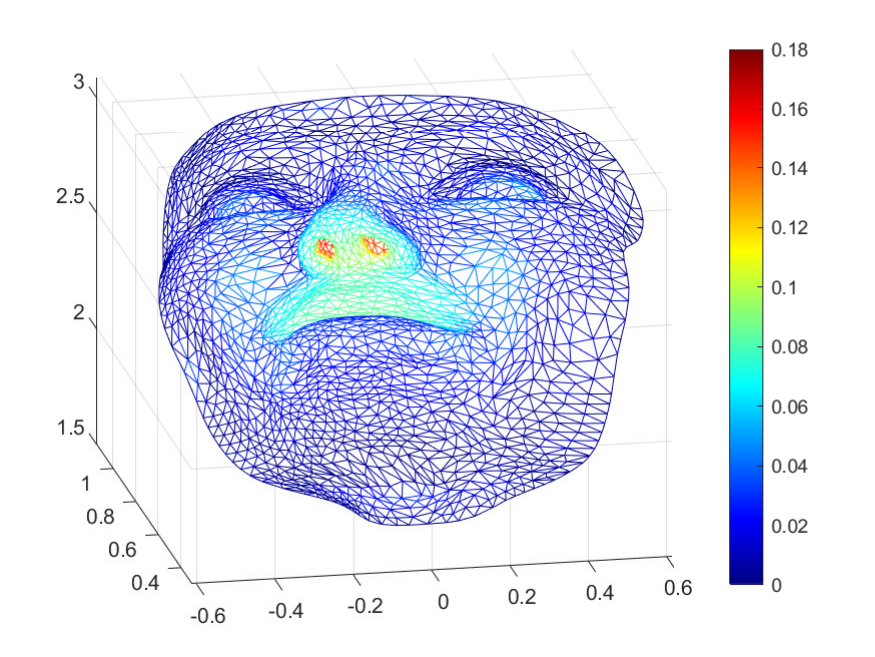}}
\subfigure[|FEM - NN| ($N=2\text{k}$)]{\includegraphics[width=.24\textwidth]{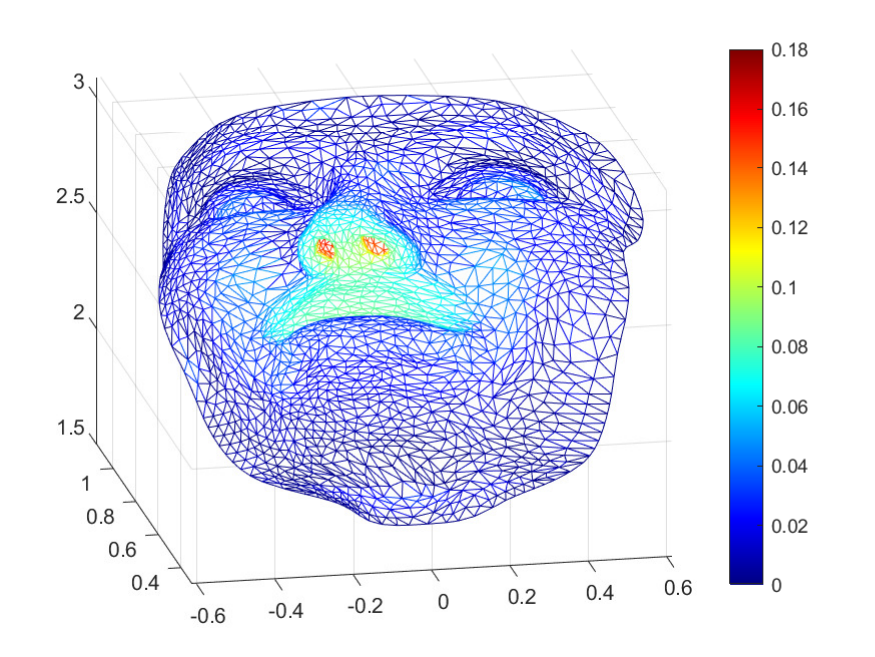}}
\subfigure[|FEM - NN| (test $N=17\text{k}$)]{\label{fig:facenn}\includegraphics[width=.24\textwidth]{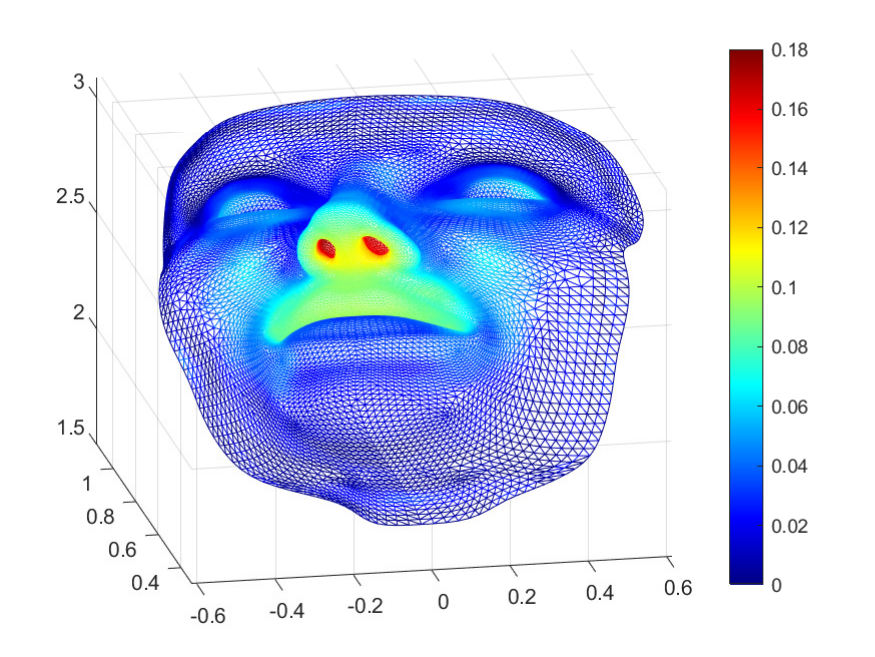}}
\caption{Comparison of the PDE solutions among FEM, VBDM and NN on the Bunny example. {\bf (a)} FEM solution with $N=8261$ training data. {\bf (b)} Absolute difference between FEM and VBDM solutions on training data set of size $N=8261$. {\bf (c)} Absolute difference between FEM and NN solutions on training data set of size $N=8261$. {\bf (d)} Absolute difference between FEM and NN solutions on $17157$ data points from the Marching Cubes algorithm, 8261 of these points were used to obtain errors in (b) and (c). {\bf (e)} FEM solution with $N=2185$ training data. {\bf (f)} Absolute difference between FEM and VBDM solutions on training data set of size $N=2185$. (g) Absolute difference between FEM and NN solutions on training data set of size $N=2185$. (h) Exactly like (d) except that only 2185 of these points are used to obtain errors in (f) and (g).
}
\label{fig:face}
\end{figure}

\comment{
\begin{figure}[htbp]
\centering     %%% not \center
\subfigure[FEM solution ($N=2\text{k}$)]{\label{fig:facefem}\includegraphics[width=.24\textwidth]{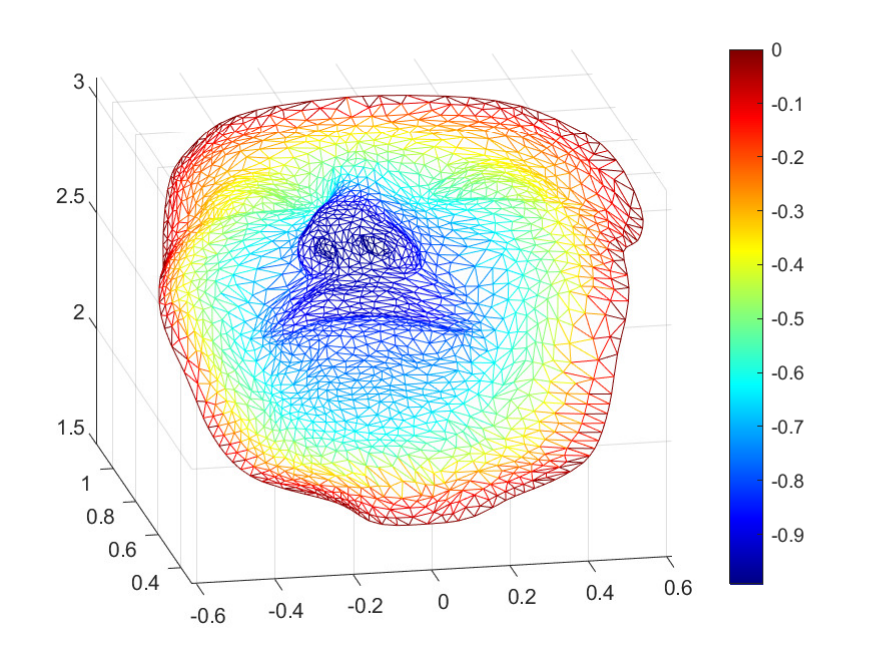}}
\subfigure[|FEM - VBDM| ($N=2\text{k}$)]{\includegraphics[width=.24\textwidth]{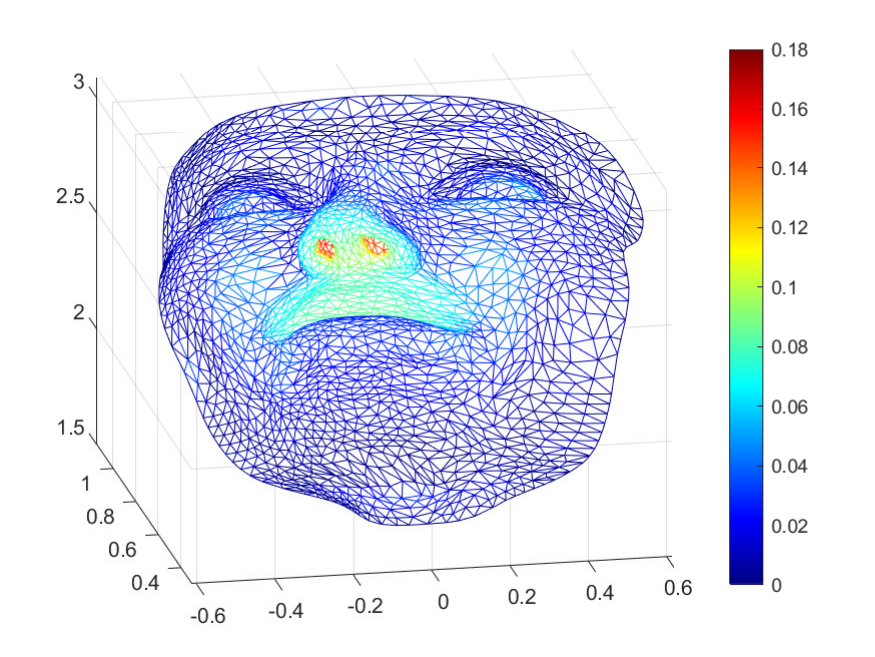}}
\subfigure[|FEM - NN| ($N=2\text{k}$)]{\label{fig:facenn}\includegraphics[width=.24\textwidth]{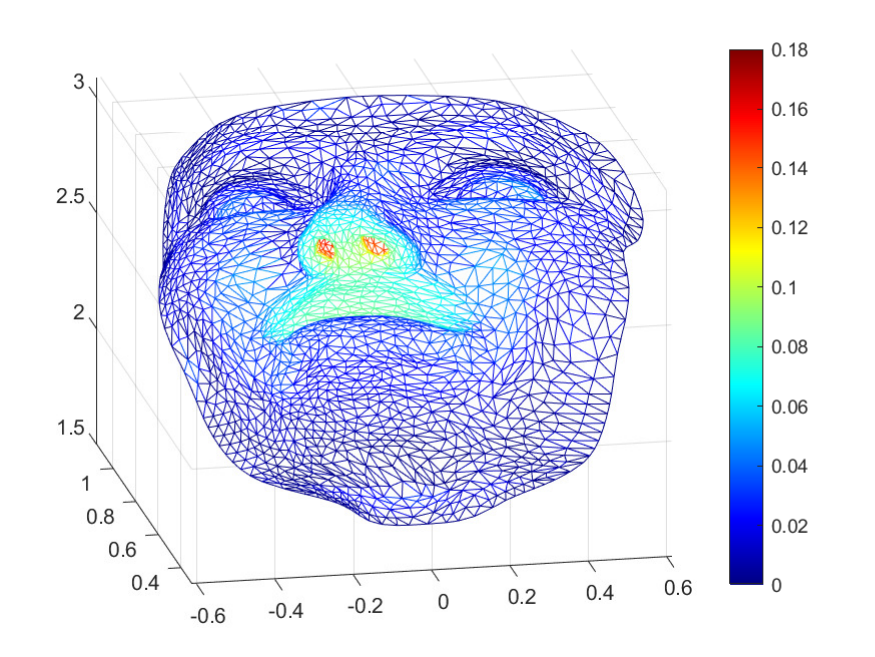}}
\subfigure[|FEM - NN| ( test = $N=17K\text{k}$]{\label{fig:facenn}\includegraphics[width=.24\textwidth]{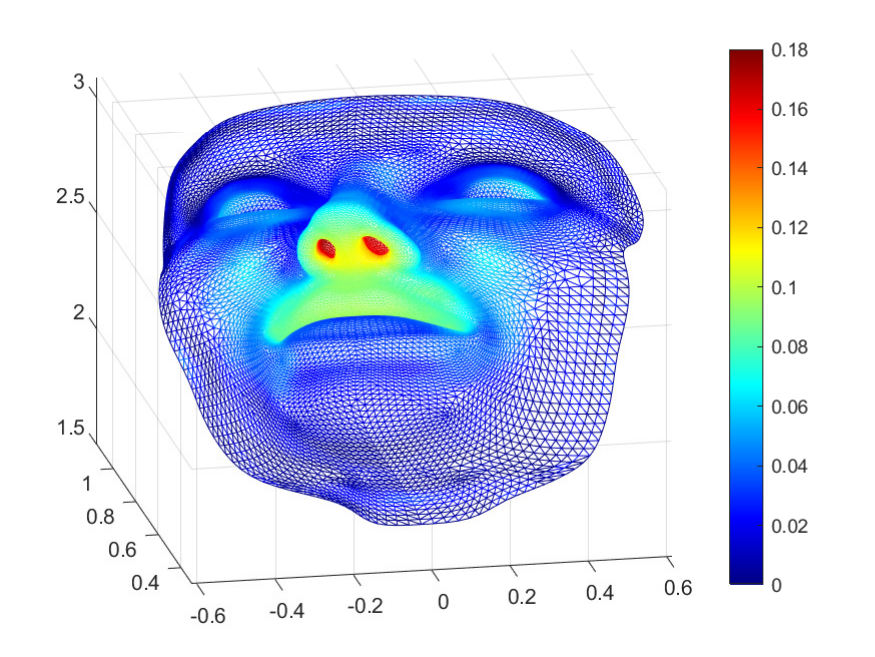}}
\caption{Comparison of the PDE solutions among FEM, VBDM and NN on the Face example. (a) FEM solution on 2185 vertices. (b) Absolute difference between FEM and VBDM solutions on 2185 vertices. (c) Absolute difference FEM and NN solutions (training error on 2185 vertices). (d) Absolute difference FEM and NN solutions on 17157 vertices (trained on 2185 vertices and tested on 17157 vertices), which represents the testing error. }
\label{fig:face}
\end{figure}
}
% Table generated by Excel2LaTeX from sheet 'Sheet2'

\section{Conclusion}
\label{sec:con}

This paper proposed a mesh-free computational framework and machine learning theory for solving PDEs on unknown manifolds given as a form of point clouds based on diffusion maps (DM) and deep learning. Parameterizing manifolds is challenging if the unknown manifold is embedded in a high-dimensional ambient Euclidean space, especially when the manifold is identified with randomly sampled data and has boundaries. First, a mesh-free DM algorithm was introduced to approximate differential operators on point clouds enabling the design of PDE solvers on manifolds with and without boundaries. Second, deep NNs were applied to parametrize PDE solutions. Finally, we solved a least-squares minimization problem for PDEs, where the empirical loss function is evaluated using the DM discretized differential operators on point clouds and the minimizer is identified via SGD. The minimizer provides a PDE solution in the form of an NN function on the whole unknown manifold with reasonably good accuracy. The mesh-free nature and randomization of the proposed solver enable efficient solutions to PDEs on manifolds arbitrary co-dimensional. 
 \rev{Computationally, we found that training NN solutions are computationally more feasible compared to performing (a stable) pseudo-inversion in solving linear problems with a singular matrix as the size of the matrix becomes large. On the other hand, this advantage diminishes when the linear problem involves inversion of non-singular matrix and the goal is to achieve high accuracy; where classical algorithms (e.g., GMRES, iterative methods) are more efficient. When the linear system is too large to compute, it was recently pointed out in \cite{YiqiMichael} that there are advantages to apply NNs to solve linear systems in a similar way as in this paper even though the linear system is non-singular, while traditional methods are not affordable.  Beyond the training procedure, the NN solution is advantageous for interpolating the solution to a new datum.} From the numerical results, we found that the proposed approach produces consistently accurate generalizations that are more robust than the one constructed using the classical Nystr\"om interpolation technique, \rev{especially when the latter is practically not convenient due to the following fact. The accuracy of the Nystr\"om extension depends crucially on the choice of the dimension of the eigenspace, $J$, that allow for the projected solution in \eqref{nystrom} to have a small residual in addition to estimating the leading $J$ eigenvector $\mathbf{v}_i$ of Laplace-Beltrami operator.}
 
 New convergence and consistency theories based on approximation and optimization analysis were developed to support the proposed framework. From the perspective of algorithm development, it is interesting to extend the proposed framework to {vector-valued PDEs with other differential operators in future work.} In terms of theoretical analysis, it is important to develop a generalization analysis of the proposed solver in the future and extend all of the analysis in this paper to manifolds with boundaries.    

\section*{Code availability}
\revv{The neural network training codes are available at the GitHub repository: \url{https://github.com/LeungSamWai/NN4ManifoldPDE}. The source codes are released under MIT license.}

\section*{Acknowledgment}

The research of J. H. was partially supported under the NSF grant DMS-1854299 and the ONR grant N00014-22-1-2193. S. J. was supported by the NSFC Grant No. 12101408. S. L. is supported by the Ross-Lynn fellowship 2021-2022 from Purdue University. S. L. and H. Y. were partially supported by the US National Science Foundation under awards DMS-2244988, DMS-2206333, and the Office of Naval Research Award N00014-23-1-2007.

\appendix

\section{GPDM algorithm for manifolds with boundaries}\label{appendix_a}

In this appendix, we give a brief overview of ghost point diffusion maps (GPDM) to construct the matrix $\mathbf{L}_\epsilon$ for manifolds with Dirichlet boundary conditions.
As mentioned in
\cite{cl:06,harlim:18,gh:18,jiang2023ghost},
the asymptotic expansion (\ref{integralop1}) %and (\ref{integralapprox}) 
in standard DM approaches is not valid near the boundary of the manifold. One way to overcome this boundary issue is the GPDM approach introduced in \cite%
{jiang2023ghost}. This approach extends the classical ghost point
method \cite{leveque2007finite} to solve elliptic and parabolic PDEs on unknown manifolds with boundary conditions \cite{jiang2023ghost,yan2021kernel}. The GPDM approach can be summarized as follows (see \cite{jiang2023ghost,yan2021kernel} for details).

%In our framework, our goal is to supplement the data points with a set of ghost points near the boundary such that the asymptotic expansion for estimating the diffusion operator is valid even for points near the boundary. As in the usual ghost point method for the finite-difference scheme, to obtain a unique solution, we need to impose an additional set of algebraic equations to determine the additional unknowns, namely the function values at the ghost points. However, unlike in the flat domain, it is also unclear how to specify the ghost points when the manifold is unknown.

%\begin{algorithm}
%\label{algm5_1} GPDM approach:

\textbf{Estimation of normal vectors at boundary points }(see details in
Section 2.2 and Appendix A of \cite{yan2021kernel})\textbf{:} Assume the
true normal vector $\boldsymbol{\nu }$ is unknown and it will be numerically
estimated. For well-sampled data, where data points are
well-ordered along intrinsic coordinates, one can identify the estimated $\tilde{%
\boldsymbol{\nu }}$ as the secant line approximation to $\boldsymbol{\nu }$%
. The error is $|\boldsymbol{\nu }-\boldsymbol{\tilde{\nu}}|=\mathcal{O}(h)$, where
the parameter $h$ denotes the distance between consecutive ghost points (see
Fig.~1(b) in \cite{yan2021kernel} for a geometric illustration). For
randomly sampled data, one can use the kernel method to estimate $%
\boldsymbol{\tilde{\nu}}$ and the error is $|\boldsymbol{\nu }-\boldsymbol{%
\tilde{\nu}}|=\mathcal{O}(\sqrt{\epsilon })$ (see Fig.~1(c) in \cite{yan2021kernel}
for a geometric illustration and Appendix~A in \cite{yan2021kernel} for a
detailed discussion).

\textbf{Specification of ghost points:}
The basic idea of ghost points, as introduced in \cite{jiang2023ghost}, is
to specify the ghost points as data points that lie on the exterior normal
collar, $\Delta M$, along the boundary.\ Then all interior points whose
distances are within $\epsilon ^{r}$ from the boundary $\partial M$ are at
least $\epsilon ^{r}$ away from the boundary of the extended manifold $M\cup
\Delta M$. Theoretically, it was shown that, under appropriate conditions,
the extended set $M\cup \Delta M$ can be isometrically embedded with an
embedding function that is consistent with the embedding $M\hookrightarrow
\mathbb{R}^{m}$ when restricted on $M$ (see Lemma~3.5 in \cite%
{jiang2023ghost}).

Technically, for randomly sampled data, the parameter $h$ can be estimated
by the mean distance from the boundary $\mathbf{\bar{x}}_{b}$ to its $P$
(around 10 in simulations) nearest neighbors. Then, given the distance
parameter $h$ and the estimated normal vector $\boldsymbol{\tilde{\nu}}$,
the approximate ghost points are given by,
\begin{equation}
\mathbf{\tilde{x}}_{b,k}=\mathbf{\bar{x}}_{b}+kh\boldsymbol{\tilde{\nu}}%
,\quad \text{for }k=1,\ldots ,{K}\text{ and }b=1,\ldots ,{N}_{b}.
\label{trugho}
\end{equation}%
In addition, one layer of interior ghost points are supplemented as $\mathbf{%
\tilde{x}}_{b,0}=\mathbf{\bar{x}}_{b}-h\boldsymbol{\tilde{\nu}}$. For
well-sampled data, the interior estimated ghost point coincides with one of the interior points on the manifold when the secant line method is used. However, for randomly sampled data, the estimated interior ghost points will not necessarily coincide with an interior point (see Fig.~1 in \cite{yan2021kernel}\ for comparison).

\textbf{Estimation of function values on the ghost points:} 

The main goal here is to estimate the function values $\{u(\mathbf{\tilde{x}}%
_{b,k})\}_{b,k=1}^{{N}_{b},{K}}$ on the exterior ghost points by
extrapolation.
% , where the ghost points $\mathbf{x}_{b,k}$ lie exactly on the
% collar manifold $\Delta M$ (corresponding to the estimates in \eqref{trugho}%
% ). 
We assume that we are given the components of the column vector,
\begin{equation}
\mathbf{u}_{M}:=(u(\mathbf{x}_{1}),\ldots ,u(\mathbf{\tilde{x}}%
_{b,0}),\ldots ,u(\mathbf{x}_{N}))\in \mathbb{R}^{N}.  \label{umh}
\end{equation}%
Here, we stress that the function values $\{u(\mathbf{\tilde{x}}%
_{b,0})\}_{b=1}^{N_b}$ are given exactly like the $u(\mathbf{x}_{i})$ for any $%
\mathbf{x}_{i}\in M$, even when the ghost points $\mathbf{\tilde{x}}_{b,0}$
do not lie on the manifold $M$. Then we will use the components of the
column vector,
\begin{equation}
\mathbf{U}_{G}:=(U_{1,1},\ldots ,U_{{{N_b}},{K}})\in \mathbb{R}^{{{N_b}}{K}},
\label{ugh}
\end{equation}%
to estimate the components of function values on ghost points,
%\mathbf{{U}}%
%_{G}:=
$(u(\mathbf{\tilde{x}}_{1,1}),\ldots ,u(\mathbf{\tilde{x}}_{{{N_b}},{K}}))\in \mathbb{R}^{{{N_b}%
}{K}}$. Numerically, we will obtain the components of $\mathbf{U}_{G}$ by
solving the following linear algebraic equations for each $b=1,\ldots ,{{N_b}}$,
\begin{equation}
\begin{aligned} U_{b,1} - 2u({\bf{x}}_{b})+ u(\tilde{{\bf{x}}}_{b,0}) &=
0,\\ U_{b,2}-2U_{b ,1}+u({\bf{x}}_{b})&=0, \\ U_{b ,k}-2U_{b,k-1}+U_{b
,k-2}&=0, \quad\quad k = 3,\ldots {K}. \\ \end{aligned}  \label{Eqn:uvv_g2}
\end{equation}%
These algebraic equations are discrete analogs of matching the first-order
derivatives along the estimated normal direction, $\tilde{\boldsymbol{\nu }}$%
.

\textbf{Construction of the GPDM estimator:} We now define the GPDM
estimator for the differential operator $\mathcal{L}$ in \eqref{ellipticPDE}%
. The discrete estimator will be constructed based on the available training
data $\{\mathbf{x}_{i}\in M\}_{i=1}^{N}$ and the estimated ghost points, $\{%
\mathbf{\tilde{x}}_{b,k}\}_{b,k=1,0}^{{N}_{b},{K}}$. In particular, since
the interior ghost points, $\{\mathbf{\tilde{x}}_{b,0}\}_{b=1}^{{N}_{b}}$
{may or may not coincide with any interior points on the manifold,} we assume that $X^{h}:=\{\mathbf{x}_{1},\ldots ,%
\mathbf{\tilde{x}}_{b,0},\ldots ,\mathbf{x}_{N}\}$ has $N$ components that
include the estimated ghost points in the
following discussion. With this notation, we define a
non-square matrix,
\begin{equation}
\mathbf{L}^{h}:=(\mathbf{L}^{(1)},\mathbf{L}^{(2)})\in \mathbb{R}^{N\times
(N+N_{b}K)},  \label{eqnLtild}
\end{equation}%
constructed right after \eqref{integralapprox}, %as in \eqref{unnormalizeddgl}, 
by evaluating the kernel on
components of $X^{h}$ for each row and the components of $X^{h}\cup \{%
\mathbf{\tilde{x}}_{b,k}\}_{b,k=1}^{{N}_{b},{K}}$ for each column. With
these definitions and those in (\ref{umh})-(\ref{Eqn:uvv_g2}), we note that
\[
\mathbf{L}^{h}\mathbf{U}=\mathbf{L}^{(1)}\mathbf{u}_{M}+\mathbf{L}^{(2)}%
\mathbf{U}_{G}=(\mathbf{L}^{(1)}+\mathbf{L}^{(2)}\mathbf{G})\mathbf{u}_{M}=%
\mathbf{\tilde{L}}\mathbf{u}_{M},
\]%
where ${\mathbf{U}}:=({\mathbf{u}}_{M},{\mathbf{U}}_{G})\in \mathbb{R}%
^{N+N_{b}K}$ and $\mathbf{G}\in \mathbb{R}^{{{N_b}}{K}\times N}$ is defined as a
solution operator to \eqref{Eqn:uvv_g2}, which is given in a compact form as $%
\mathbf{U}_{G}=\mathbf{G}\mathbf{u}_{M}$. Then, the GPDM estimator $\mathbf{%
\tilde{L}}$\ is defined as an $N\times N$ matrix,
\begin{equation}
\mathbf{\tilde{L}}:=\mathbf{L}^{(1)}+\mathbf{L}^{(2)}\mathbf{G}.
\label{GPDMmatrix}
\end{equation}%
Note that we have the consistency of GPDM estimator for the
differential operator defined on functions that take values on the extended
$M\cup \Delta M$ (see Lemma 2.3 in \cite{yan2021kernel}).

\textbf{Combination with the discretization of the boundary conditions:}
We take $\mathbf{L}_{\epsilon }\in \mathbb{R}^{(N-N_{b})\times N}$ to be a
sub-matrix of $\mathbf{\tilde{L}}$ $\in \mathbb{R}^{N\times N}$ in %
\eqref{GPDMmatrix}, where the $N-N_{b}$ rows of $\mathbf{L}_{\epsilon }$\
correspond to the interior points from $X^{h}$. There are $N-N_{b}$
equations from the linear system $(-\mathbf{a}+\mathbf{L}_{\epsilon })%
\mathbf{u}_{M}=\mathbf{f}$. To close the discretized problem, we use the $%
N_{b}$ equations from the Dirichlet boundary condition at the boundary
points, $u(\mathbf{\bar{x}}_{b})=g(\mathbf{\bar{x}}_{b})$ for $\mathbf{\bar{x%
}}_{b}\in \left\{ \mathbf{\bar{x}}_{1},\ldots ,\mathbf{\bar{x}}%
_{N_{b}}\right\} \subset X^{h}\cap \partial M$. With the construction of $%
\mathbf{L}_{\epsilon }$ and the Dirichlet boundary conditions, we
then solve the problem in \eqref{eqn:DMNN2} using DNN approach.

\section{Global convergence analysis of neural network optimization}\label{appendix_c}

In this section, we will summarize notations and main ideas of the proof of Theorem \ref{thm:lcr} first. Then we prove several lemmas in preparation for the proof of Theorem \ref{thm:lcr}. The proof of Theorem \ref{thm:lcr} will be presented after these lemmas.

\subsection{Preliminaries}
The Rademacher complexity is a basic tool for generalization analysis. In our analysis, we will use several important lemmas and theorems related to it. To be self-contained, they are listed as follows.

\begin{defi}[The Rademacher complexity of a function class $\fF$] \label{def:rad}
    Given a sample set $S=\{z_1,\dots,z_N\}$ on a domain $\fZ$, and a class $\fF$ of real-valued functions defined on $\fZ$, the empirical Rademacher complexity of $\fF$ on $S$ is defined as
    \[
    \Rad_S(\fF)=\frac{1}{N}\Exp_{\vtau}\left[ \sup_{f\in \fF} \sum_{i=1}^N \tau_i f(z_i)  \right],
    \]
    where $\tau_1$, $\dots$, $\tau_N$ are independent random variables drawn from the Rademacher distribution, i.e., $\Prob(\tau_i=+1)=\Prob(\tau_i=-1)=\frac{1}{2}$ for $i=1,\dots,N$.
\end{defi}

First, we recall a well-known contraction lemma for the Rademacher complexity. 

\begin{lem}[Contraction lemma \cite{Shalev-Shwartz2014}]\label{lem..RademacherComplexityContraction}
    Suppose that $\psi_i:\sR\to\sR$ is a $C_\mathrm{L}$-Lipschitz function for each $i\in[N]:=\{1,\ldots,N\}$. 
    For any $\vy\in\sR^N$, let $\vpsi(\vy)=(\psi_1(y_1),\cdots,\psi_N(y_N))^\T$. For an arbitrary set of functions $\fF$ on an arbitrary domain $\fZ$ and an arbitrary choice of samples $S=\{\vz_1,\dots,\vz_N\}\subset\mathcal{Z}$, we have
    \begin{align*}
        \Rad_S(\vpsi\circ \fF)\leq C_\mathrm{L}\Rad_S(\fF).
    \end{align*}
\end{lem}

Second, the Rademacher complexity of linear predictors can be characterized by the lemma below.

\begin{lem}[Rademacher complexity for linear predictors \cite{Shalev-Shwartz2014}]\label{lem..linear}
    Let $\Theta=\{\vw_1,\cdots,\vw_\width\}\in\sR^n$. Let $\fG=\{g(\vw)= \vw^\T\vx:\norm{\vx}_1\leq 1\}$ be the linear function class with parameter $\vx$ whose $\ell^1$ norm is bounded by $1$. Then
    \begin{equation*}
        \Rad_\Theta(\fG)\leq \max_{1\leq k\leq m} \norm{\vw_k}_\infty\sqrt{\frac{2\log(2n)}{\width}}.
    \end{equation*}
\end{lem}

Finally, let us state a general theorem concerning the Rademacher complexity and generalization gap of an arbitrary set of functions $\fF$ on an arbitrary domain $\fZ$, which is essentially given in \cite{Shalev-Shwartz2014}.
\begin{theo}[Rademacher complexity and generalization gap \cite{Shalev-Shwartz2014}]\label{thm..RademacherComplexityGeneralizationGap}
    Suppose that $f$'s in $\fF$ are non-negative and uniformly bounded, i.e., for any $f\in\fF$ and any $\vz\in\fZ$, $0\leq f(\vz)\leq B$. Then for any $\delta\in(0,1)$, with probability at least $1-\delta$ over the choice of $N$ i.i.d. random samples $S=\{\vz_1,\dots,\vz_N\}\subset\mathcal{Z}$, we have
    \begin{align*}
        \sup_{f\in\fF}\Abs{\frac{1}{N}\sum_{i=1}^N f(\vz_i)-\Exp_{\vz}f(\vz)}
        &\leq 2\Exp_{S}\Rad_{S}(\fF)+B\sqrt{\frac{\log(2/\delta)}{2N}},\\
        \sup_{f\in\fF}\Abs{\frac{1}{N}\sum_{i=1}^N f(\vz_i)-\Exp_{\vz}f(\vz)}
        &\leq 2\Rad_{S}(\fF)+3B\sqrt{\frac{\log(4/\delta)}{2N}}.
    \end{align*}
    %where the dummy variable $S'$ denotes a set of $N$ i.i.d. random samples of $\mathcal{Z}$. 
\end{theo}

{Here, we should point out that the distribution of $\vz$ is arbitrary. In our specific application, $\mathbb{E}_{\vz} = \mathbb{E}_\pi$.}

\subsection{Notations and main ideas of the proof of Theorem \ref{thm:lcr}}
Now we are going to present the notations and main ideas of the proof of Theorem \ref{thm:lcr}. Recall that we use a two-layer neural network $\phi(\vx;\vtheta)$ with $\vtheta=\mathrm{vec}\{a_k,\vw_k\}_{k=1}^\width$. In the gradient descent iteration, we use $t$ to denote the iteration or the artificial time variable in the gradient flow. Hence, we define the following notations for the evolution of parameters at time $t$:
\begin{equation*}
     a^t_k 
     := a_k(t), \quad 
     \vw_k^t 
     := \vw_k(t),\quad
     \vtheta^t 
     := \vtheta(t) := \mathrm{vec}{\{a_k^t, \vw_k^t\}}_{k=1}^\width.
\end{equation*}
 Similarly, we can introduce $t$ to other functions or variables depending on $\vtheta(t)$. When the dependency of $t$ is clear, we will drop the index $t$. In the initialization of gradient descent, we set 
\begin{align}\begin{split}\label{eqn:init}
     &a^0_k 
     := a_k(0)\sim \mathcal{N}(0, \gamma^2), \quad 
     \vw_k^0 
     := \vw_k(0)\sim \mathcal{N}(\vzero, \mI_n),\\
     &\vtheta^0 
     := \vtheta(0) := \mathrm{vec}{\{a_k^0, \vw_k^0\}}_{k=1}^\width.
\end{split}\end{align}
Then the empirical risk can be written as
%\begin{eqnarray*}
%    \RS(\vtheta) &=&  \frac{1}{2N} (\mA\phi(X;\vtheta)-{f}(X))^{\T}(\mA\phi(X;\vtheta)-{f}(X)) \\
%    & = &   \frac{1}{2N} (\phi(X;\vtheta)-\mA^{-1}{f}(X))^{\T}\mA^{\T}\mA (\phi(X;\vtheta)-\mA^{-1}{f}(X))  \\
%    & = & \frac{1}{2N}\ve^{\T} \mA^{\T}\mA \ve,
%\end{eqnarray*}
$\RS(\vtheta) = \frac{1}{2N}\ve^{\T} \mA^{\T}\mA \ve$, where we denote $e_i = \phi(\vx_i;\vtheta) - (\mA^{-1}f(X))_i$ for $i\in[N]$ and $\ve = (e_1, e_2, \ldots, e_N)^{\T}$. Hence, the gradient descent dynamics is $\dot{\vtheta} = -\nabla_{\vtheta}\RS(\vtheta)$, or equivalently in terms of $a_k$ and $\vw_k$ as follows:
\BEA
\begin{aligned}\label{componentwisegradflow}
    \dot{a}_k = -\nabla_{a_k}\RS(\vtheta) 
    &= -\frac{1}{N}\sum_{i=1}^N (e^{\T}\mA^{\T}\mA)_i \sigma(\vw_k^\T\vx_i),\\
    %-\frac{1}{N} \bm{e}^{\T}\mA^{\T}\mA  \sigma(\vw_k S),\nonumber\\
    \dot{\vw}_k = -\nabla_{\vw_k}\RS(\vtheta) 
    &= -\frac{1}{N}\sum_{i=1}^N (e^{\T}\mA^{\T}\mA)_i a_k\sigma'(\vw_k^\T\vx_i)\vx_i.
\end{aligned}
\EEA

Adopting the neuron tangent kernel point of view \cite{Arthur18}, in the case of a two-layer neural network with an infinite width, the corresponding kernels $k^{(a)}$ for parameters in the last linear transform and $k^{(w)}$ for parameters in the first layer are functions from $M \times M$ to $\sR$ defined by
\begin{align}
    k^{(a)}(\vx,\vx') 
    &:= \Exp_{\vw\sim \mathcal{N}(\vzero,\mI_n)}g^{(a)}(\vw;\vx,\vx'),\nonumber\\
    k^{(w)}(\vx,\vx')
    &:= \Exp_{(a,\vw)\sim \mathcal{N}(\vzero,\mI_{n+1})}g^{(w)}(a,\vw;\vx,\vx'),\nonumber
\end{align}
where
\begin{align}
    g^{(a)}(\vw;\vx,\vx')
    &:= \left[\sigma(\vw^\T\vx)\right]\nonumber \cdot\left[\sigma(\vw^\T\vx')\right],\nonumber\\
    g^{(w)}(a,\vw;\vx,\vx')
    &:= a^2\big[\sigma'(\vw^\T\vx)\vx\big]\cdot
    \big[\sigma'(\vw^\T\vx')\vx'\big].\nonumber
\end{align}
These kernels evaluated at $N\times N$ pairs of samples lead to $N\times N$ Gram matrices $\mK^{(a)}$ and $\mK^{(\vw)}$ with $K^{(a)}_{ij} = k^{(a)}(\vx_i,\vx_j)$ and $K^{(w)}_{ij} = k^{(w)}(\vx_i,\vx_j)$, respectively. Our analysis requires the matrix $\mK^{(a)}$ to be positive definite, which has been verified for regression problems under mild conditions on random training data $X=\{\vx_i\}_{i=1}^N$ and can be generalized to our case. Hence, we assume this together with the non-singularity of $\mA$.
%\begin{assu}\label{assump..LambdaMin}
%    Assume that: 1) The smallest eigenvalue of $\mK^{(a)}$, denoted as $\lambda_S$, is positive. 2) The smallest eigenvalue of $\mA \mA^{\T}$, denoted as $\lambda_{\mA}$, is positive.
%\end{assu}

For a two-layer neural network with $\width$ neurons, define the $N\times N$ Gram matrix $\mG^{(a)}(\vtheta)$ and $\mG^{(w)}(\vtheta)$ using the following expressions for the $(i,j)$-th entry
\BEA\label{Grammian_G}
\begin{aligned}
    \mG_{ij}^{(a)}(\vtheta) 
    &:= \frac{1}{\width}\sum_{k=1}^\width g^{(a)}(\vw_k;\vx_i,\vx_j),\\
    \mG_{ij}^{(w)}(\vtheta)
    &:= \frac{1}{\width}\sum_{k=1}^\width g^{(w)}(a_k,\vw_k;\vx_i,\vx_j).
\end{aligned}
\EEA
Clearly, $\mG^{(a)}(\vtheta)$ and $\mG^{(w)}(\vtheta)$ are both positive semi-definite for any $\vtheta$. Let $\mG(\vtheta)=\mG^{(a)}(\vtheta)+\mG^{(w)}(\vtheta)$, {taking time derivative of \eqref{eqn:emloss4} and using the equalities in \eqref{componentwisegradflow} and \eqref{Grammian_G}}, then we have the following evolution equation to understand the dynamics of the gradient descent method applied to \eqref{eqn:emloss4}:
%\begin{equation*} 
%     \frac{\D }{\D t}\phi(X;\vtheta)
%    =-\frac{1}{N}\mG(\vtheta)\mA^{\T%}\mA \bm{e},
%\end{equation*}
%\hz{I change $G(x_i,x_j)$ to $G_{ij}$.}
%{\color{magenta}Something is missing in the equation above, I can't understand it.}
%and
\begin{align}\begin{split}\label{eqn:boundS}
        \frac{\D}{\D t}\RS(\vtheta) = - \norm{\nabla_{\vtheta}\RS(\vtheta)}^2_2 &=-\frac{\width}{N^2}\ve^{\T}\mA^{\T}\mA  \mG(\vtheta) \mA^{\T}\mA  \ve \\&
        \leq - \frac{\width}{N^2}\ve^{\T} \mA^{\T}\mA  \mG^{(a)}(\vtheta) \mA^{\T}\mA \ve.
\end{split}\end{align}

Our goal is to show that $\RS(\vtheta)$ converges to zero. These goals are true if the smallest eigenvalue $    \lambda_{\min}\left(\mG^{(a)}(\vtheta)\right)$ of $\mG^{(a)}(\vtheta)$ has a positive lower bound uniformly in $t$, since in this case we can solve \eqref{eqn:boundS} and bound $\RS(\vtheta)$ with a function in $t$ converging to zero when $t\rightarrow\infty$ as shown in Lemma \ref{lem:exp_RS}. In fact, a uniform lower bound of $    \lambda_{\min}\left(\mG^{(a)}(\vtheta)\right)$ can be $\frac{1}{2}\lambda_S$, which can be proved in the following three steps:

\begin{itemize}
    \item (\textbf{Initial phase}) By Assumption \ref{assump..LambdaMin} of $\mK^{(a)}$, we can show that $    \lambda_{\min}\left(\mG^{(a)}(\vtheta(0))\right)\approx \lambda_S$ in Lemma \ref{lem:lambda_min} using the observation that $\mK^{(a)}_{ij}$ is the mean of $g(\vw;\vx_i,\vx_j)$ over the normal random variable $\vw$, while $\mG^{(a)}_{ij}(\vtheta(0))$ is the mean of $g(\vw;\vx_i,\vx_j)$ with $\width$ independent realizations. 
    \item (\textbf{Evolution phase}) The GD dynamics results in $\vtheta(t)\approx \vtheta(0)$ when over-parametrized as shown in Lemma \ref{prop:a_w}, meaning that
    $\lambda_{\min}\left(\mG^{(a)}(\vtheta(0))\right)\approx \lambda_{\min}\left(\mG^{(a)}(\vtheta(t))\right)$.
    \item (\textbf{Final phase}) To show the uniform bound $\lambda_{\min}\left(\mG^{(a)}(\vtheta(t))\right)\geq \frac{1}{2}\lambda_S$ for all $t\geq 0$, we introduce a stopping time $t^*$ via
\begin{equation}\label{eqn:ts}
    t^* = \inf\{t \mid \vtheta(t)\notin \mathcal{M}(\vtheta^0)\},
\end{equation}
where $\mathcal{M}(\vtheta^0) := \left\{\vtheta \mid \norm{\mG^{(a)}(\vtheta) - \mG^{(a)}(\vtheta^0)}_{\mathrm{F}}\leq \frac{1}{4}\lambda_S\right\}$, and show that $t^*$ is equal to infinity in the final proof of Theorem \ref{thm:lcr}.
\end{itemize}

\subsection{Important lemmas for Theorem \ref{thm:lcr}}

Now we are going to prove several lemmas for Theorem \ref{thm:lcr}. In the analysis below, we use $ \bar{a}^t_{k} := \bar{a}_{k}(t):=\gamma^{-1}a_k(t)$ with $0<\gamma<1$, e.g., $\gamma=\frac{1}{\sqrt{\width}}$ or $\gamma=\frac{1}{\width}$, and $\bar{\vtheta}(t):= \mathrm{vec}{\{\bar{a}_k^t, \vw_k^t\}}_{k=1}^\width$.

\begin{lem}\label{lem1}
    For any $\delta\in(0,1)$ with probability at least $1-\delta$ over the random initialization in \eqref{eqn:init}, we have
    \begin{equation}\label{eqn:lem1}
        \begin{aligned}
            \max\limits_{k\in[\width]}\left\{\abs{\bar{a}_k^0},\; \norm{\vw^0_k}_\infty\right\}
             & \leq\sqrt{2
            \log\frac{2\width(n+1)}{\delta}},      \\
            \max\limits_{k\in[\width]}\left\{\abs{a_k^0}\right\}
             & \leq \gamma\sqrt{2
                \log\frac{2\width(n+1)}{\delta}}.
        \end{aligned}
    \end{equation}
\end{lem}
\begin{proof}
    If $\rX \sim \mathcal{N}(0, 1)$, then $\Prob(\abs{\rX} > \eps) \leq 2\E^{-\frac{1}{2}\eps^2}$ for all $\eps > 0$. Since $\bar{a}^0_k\sim \mathcal{N}(0,1)$, ${(\vw_k^0)}_{\alpha}\sim \mathcal{N}(0,1)$ for $k\in[\width], \alpha \in[n]$, and they are all independent, by setting
    \begin{equation*}
        \eps = \sqrt{2\log\frac{2\width(n+1)}{\delta}},
    \end{equation*}
    one can obtain
    \begin{equation*}
        \begin{aligned}
            \Prob\left(\max\limits_{k\in[\width]}\left\{\abs{\bar{a}_k^0},\norm{\vw^0_k}_\infty\right \}>\eps\right)
             & = \Prob\left(\left(\bigcup\limits_{k\in[\width]}\left\{\abs{\bar{a}_k^0}>\eps\right\}\right)\bigcup\left(\bigcup\limits_{k\in[\width],\alpha\in[n]}\left\{\abs{{(\vw_k^0)}_{\alpha}}>\eps\right\}\right)\right) \\
             & \leq \sum_{k=1}^\width \Prob\left(\abs{\bar{a}_k^0}>\eps\right) + \sum_{k=1}^\width\sum_{\alpha=1}^n \Prob\left(\abs{{(\vw^0_k)}_{\alpha}}>\eps\right)                             \\
             & \leq 2\width \e^{-\frac{1}{2}\eps^2} + 2\width n \e^{-\frac{1}{2}\eps^2}                                                                                                       \\
             & = 2\width(n+1)\e^{-\frac{1}{2}\eps^2}                                                                                                                                    \\
             & = \delta,
        \end{aligned}
    \end{equation*}
    which implies the conclusions of this lemma.
\end{proof}

%\noindent Then we introduce the following definition
% \begin{definition}
%     If $\rX$ is a random variable, we define the sub-gaussian norm as
%     \begin{equation*}
%         \norm{\rX}_{\phi_2} := \inf\{s>0 \mid \Exp\exp(\rX^2/s^2)\leq 2\},
%     \end{equation*}
%     and we let $C_t := \norm{\norm{\vw}_2^6}_{\phi_2} = \norm{(\chi^2(n))^3}_{\phi_2}$ where $\vw\sim \mathcal{N}(0,\mI_n)$ and $\chi^2(n)$ is $\chi^2$ distribution with $n$ degrees of freedom.
% \end{definition}
\begin{lem}\label{lem2}
    For any $\delta\in(0,1)$ with probability at least $1-\delta$ over the random initialization in \eqref{eqn:init}, we have
    \begin{equation*}
        \RS(\vtheta^0)\leq\frac{1}{2}\left(1 + 3\gamma n^r \sqrt{\width}\|\mA\|_2\left(2\log\frac{4\width(n+1)}{\delta}\right)^{(r+1)/2}\left(r\sqrt{2\log(2n)}+\sqrt{\log(8/\delta)/{2}}\right)\right)^2,
    \end{equation*}
\end{lem}
\begin{proof}
    From Lemma~\ref{lem1} we know that with probability at least $1-\delta/2$,
    \begin{equation*}
        \abs{\bar{a}_k^0} \leq \sqrt{2\log\frac{4\width(n+1)}{\delta}} \quad \text{and} \quad \norm{\vw_k^0}_1 \leq n\sqrt{2\log\frac{4\width(n+1)}{\delta}}.
    \end{equation*}
        Let
    \begin{equation*}
        \fH = \{h(\bar{a},\vw;\vx) \mid h(\bar{a},\vw;\vx) = \bar{a}\sigma(\vw^\T\vx),\vx\in\Omega \}.
    \end{equation*}
    Each element in the above set is a function of $\bar{a}$ and $\vw$ while $\vx\in [0,1]^n$ is a parameter. 
    Since $\norm{\vx}_\infty\leq 1$, we have
    \begin{equation*}
        \begin{aligned}
            \abs{h(\bar{a}^0_k,\vw^0_k;\vx)}
              \leq \abs{\bar{a}_k^0} \norm{\vw_k^0}_1^r  \leq n^r \left(2{\log \frac{4\width(n+1)}{\delta}}\right)^{(r+1)/2}.
        \end{aligned}
    \end{equation*}
    Then with probability at least $1-\delta/2$, by the Rademacher-based uniform convergence theorem, we have
    \begin{equation*}
        \begin{aligned}
            \frac{1}{\gamma \width}\sup_{\vx\in\Omega}\abs{\phi(\vx;\vtheta^0)}
             & = \sup_{\vx\in\Omega}\Abs{\frac{1}{\width}\sum_{k=1}^\width h(\bar{a}^0_k,\vw^0_k;\vx)-\Exp_{(\bar{a},\vw)\sim \mathcal{N}(0,\mI_{n+1})}h(\bar{a},\vw;\vx)}\\
             & \leq 2\Rad_{\bar{\vtheta}^0}(\fH) + 3n^r \left( 2{\log \frac{4\width(n+1)}{\delta}}\right)^{(r+1)/2} \sqrt{\frac{\log(8/\delta)}{{2}\width}},
        \end{aligned}
    \end{equation*}
    where
    \begin{align*}
        \Rad_{\bar{\vtheta}^0}(\fH) 
        := \frac{1}{\width}\Exp_{\vtau}\left[\sup_{\vx\in\Omega}\sum_{k=1}^\width\tau_k h(\bar{a}_k^0,\vw_k^0;\vx)\right]
        = \frac{1}{\width}\Exp_{\vtau}\left[\sup_{\vx\in\Omega}\sum_{k=1}^\width\tau_k \bar{a}_k^0 \sigma(\vw^{0\T}_k\vx)
        \right],
    \end{align*}
    where $\vtau$ is a random vector in $\sN^\width$ with i.i.d. entries $\{\tau_k\}_{k=1}^\width$ following the Rademacher distribution. With probability at least $1-\delta/2$, $\psi_k(y_k)=\bar{a}_k\sigma(y_k)$ for $k\in[\width]$ is a Lipschitz continuous function with a Lipschitz constant 
    \[
    r n^{r-1}\left(2\log\frac{4\width(n+1)}{\delta}\right)^{r/2}
    \] 
    when $y_k\in[-n\sqrt{2\log(4\width(n+1)/\delta)},n\sqrt{2\log(4\width(n+1)/\delta)}]$. We continuously extend $\psi_k(y_k)$ to the domain $\mathbb{R}$ with the same Lipschitz constant.
    
    Applying Lemma \ref{lem..RademacherComplexityContraction} with $\psi_k(y_k)$, we have
    \begin{eqnarray}
        \frac{1}{\width}\Exp_{\vtau}\left[\sup_{\vx\in\Omega}\sum_{k=1}^\width\tau_k \bar{a}_k^0 \sigma(\vw^{0\T}_k\vx)
        \right]
       & \leq & \frac{1}{\width} r n^{r-1}\left(2\log\frac{4\width(n+1)}{\delta}\right)^{r/2}\Exp_{\vtau}\left[\sup_{\vx\in\Omega}\sum_{k=1}^\width\tau_k \vw^{0\T}_k\vx\right]\nonumber\\
       &\leq & \frac{r n^{r}\sqrt{2\log(2n)}}{\sqrt{\width}} \left(2\log\frac{4\width(n+1)}{\delta}\right)^{(r+1)/2},
        \label{eq..InitialI_1Contraction}
    \end{eqnarray}
    where the second inequality is by the Rademacher bound for linear predictors in Lemma \ref{lem..linear}. 
   
    So one can get
    \begin{align*}
        \sup_{\vx\in\Omega}\abs{\phi(\vx;\vtheta^0)} &\leq \gamma \sqrt{m}n^r\left(2\log\frac{4\width(n+1)}{\delta}\right)^{(r+1)/2}\left( 2r \sqrt{2\log(2n)}+3\sqrt{\log(8/\delta){/2}}\right)\\
        &\leq 3\gamma\sqrt{\width}n^r\left(2\log\frac{4\width(n+1)}{\delta}\right)^{(r+1)/2}\left(r \sqrt{2\log(2n)}+\sqrt{\log(8/\delta){/2}}\right).
    \end{align*}
    Then
    \begin{equation*}
        \begin{aligned}
            \RS(\vtheta^0)
             & \leq \frac{1}{2N} \left( \|\mA\|_2\|\phi(S;\bmtheta^0)\|_2+\sqrt{N} \right)^2
             %\frac{1}{2N}\sum_{i=1}^{N}\left(1 + \abs{\phi(\vx_i;\vtheta^0)}\right)^{2}                   
             \\
             & \leq \frac{1}{2}\left(1 + 3\gamma\sqrt{\width}n^r\|\mA\|_2\left(2\log\frac{4\width(n+1)}{\delta}\right)^{(r+1)/2}\left(r \sqrt{2\log(2n)}+\sqrt{\log(8/\delta){/2}}\right)\right)^2,
        \end{aligned}
    \end{equation*}
    where the first inequality comes from the fact that $|f|\leq 1$ by our assumption of the target function.
\end{proof}

% \begin{theo}[Sub-gaussian Hoeffding inequality]
%     Suppose $\rX_1,\ldots,\rX_\width$ are sub-gaussian random variables with $\Exp\rX_1=\mu$, then $\forall s\geq 0$ we have
%     \begin{equation*}
%         \Prob\left(\abs{\frac{1}{\width}\sum_{k=1}^\width\rX_k-\mu}\geq s\right)\leq 2\exp\left(-\frac{mC_0 s^2}{\norm{\rX_1}^2_{\phi_2}}\right),
%     \end{equation*}
%     where $C_0$ is an absolute constant.
% \end{theo}
% \begin{proof}
% \end{proof}

The following lemma shows the positive definiteness of $\mG^{(a)}$ at initialization.

\begin{lem}\label{lem:lambda_min}
    For any $\delta\in(0,1)$, if $\width\geq\frac{16N^4C_n}{\lambda_S^2\delta}$, then with probability at least $1-\delta$ over the random initialization in \eqref{eqn:init}, we have
    \begin{equation*}
        \lambda_{\min}\left(\mG^{(a)}(\vtheta^0)\right)\geq\frac{3}{4}\lambda_S,
    \end{equation*}
    where $C_n := \Exp\norm{\vw}_1^{4r}<+\infty$ with $\vw\sim \mathcal{N}(\vzero,\mI_n)$.
\end{lem}
\begin{proof}
    We define $\Omega_{ij} := \{\vtheta^0 \mid \lvert \mG^{(a)}_{ij}(\vtheta^0) - \mK^{(a)}_{ij}\rvert \leq \frac{\lambda_S}{4N}\}$.
    Note that
    \begin{equation*}
        \abs{g^{(a)}(\vw_k^0;\vx_i,\vx_j)}  \leq  \norm{\vw_k^0}^{2r}_1.
    \end{equation*}
    So
    \begin{equation*}
        \Var\left(g^{(a)}(\vw_k^0;\vx_i,\vx_j)\right) \leq \Exp\left(g^{(a)}(\vw_k^0;\vx_i,\vx_j)\right)^2 \leq \Exp\norm{\vw_k^0}^{4r}_1 = C_n,
    \end{equation*}
    and
    \begin{equation*}
        \Var\left(\mG_{ij}^{(a)}(\vtheta^0)\right) = \frac{1}{\width^2}\sum_{k=1}^{\width}\Var\left(g^{(a)}(\vw_k^0;\vx_i,\vx_j)\right) \leq \frac{ C_n}{\width}.
    \end{equation*}
    Then the probability of the event $\Omega_{ij}$ has the lower bound:
    \begin{equation*}
        \Prob(\Omega_{ij}) \geq 1 - \frac{\Var\left(\mG_{ij}^{(a)}(\vtheta^0)\right)}{[\lambda_S/(4N)]^2} \geq 1 - \frac{16N^2C_n}{\lambda_S^2\width}.
    \end{equation*}
    Thus, with probability at least $\left(1 - \frac{16N^2C_n}{\lambda_S^2\width}\right)^{N^2} \geq 1 - \frac{16N^4C_n}{\lambda_S^2\width}$, we have all events $\Omega_{ij}$ for $i,j\in[N]$ to occur. This implies that with probability at least $1 - \frac{16N^4C_n}{\lambda_S^2\width}$, we have 
    \begin{equation*}
        \norm{\mG^{(a)}(\vtheta^0) - \mK^{(a)}}_{\mathrm{F}} \leq \frac{\lambda_S}{4}.
    \end{equation*}
    Note that $\mG^{(a)}(\vtheta^0)$ and $\mK^{(a)}$ are positive semi-definite normal matrices. Let $\bm{v}$ be the singular vector of $\mG^{(a)}(\vtheta^0)$ corresponding to the smallest singular value, then 
    \begin{equation*}
        \lambda_{\min}\left(\mG^{(a)}(\vtheta^0)\right) = \bm{v}^{\T} \mG^{(a)}(\vtheta^0)\bm{v}= \bm{v}^{\T} \mK^{(a)}\bm{v}  +\bm{v}^{\T} (\mG^{(a)} (\vtheta^0)- \mK^{(a)})\bm{v}  \geq \lambda_S - \norm{\mG^{(a)} (\vtheta^0)- \mK^{(a)}}_{2} \geq \lambda_S - \norm{\mG^{(a)} (\vtheta^0)- \mK^{(a)}}_{\mathrm{F}}.
    \end{equation*}
    
    So 
    \begin{equation*}
        \lambda_{\min}\left(\mG^{(a)}(\vtheta^0)\right)   \geq \lambda_S - \norm{\mG^{(a)} (\vtheta^0)- \mK^{(a)}}_{\mathrm{F}} \geq \frac{3}{4}\lambda_S.
    \end{equation*}

    For any $\delta\in(0,1)$, if $\width\geq\frac{16N^4 C_n}{\lambda_S^2\delta}$, then with probability at least $1-\frac{16N^4C_n}{\lambda_S^2\width}\geq 1-\delta$ over the initialization $\vtheta^0$, we have $\lambda_{\min}\left(\mG^{(a)}(\vtheta^0)\right) \geq \frac{3}{4}\lambda_S$.
\end{proof}

The following lemma estimates the empirical loss dynamics before the stopping time $t^*$ in \eqref{eqn:ts}.

\begin{lem}\label{lem:exp_RS}
    For any $\delta\in(0,1)$, if $\width\geq\frac{16N^4 C_n}{\lambda_S^2\delta}$, then with probability at least $1-\delta$ over the random initialization in \eqref{eqn:init}, we have for any $t\in[0, t^*)$
    \begin{equation*}
        \RS(\vtheta(t)) \leq \exp\left(-\frac{\width\lambda_S \lambda_{\mA}   t}{N}\right)\RS(\vtheta^0).
    \end{equation*}
\end{lem}
\begin{proof}
    From Lemma~\ref{lem:lambda_min}, for any $\delta\in(0,1)$ with probability at least $1-\delta$ over initialization $\vtheta^0$ and for any $t\in[0,t^*)$ with $t^*$ defined in \eqref{eqn:ts}, we have $\vtheta(t)\in\mathcal{M}(\vtheta^0)$ defined right after \eqref{eqn:ts}
    % in \eqref{eqn:Mtheta} 
    and
    \begin{equation*}
        \begin{aligned}
            \lambda_{\min}\left(\mG^{(a)}(\vtheta)\right)
              \geq \lambda_{\min}\left(\mG^{(a)}(\vtheta^0)\right) - \norm{\mG^{(a)}(\vtheta) - \mG^{(a)}(\vtheta^0)} _{\mathrm{F}} 
             \geq \frac{3}{4}\lambda_S - \frac{1}{4}\lambda_S                        = \frac{1}{2}\lambda_S.
        \end{aligned}
    \end{equation*}
    Note that $\mG_{ij} = \frac{1}{\width}\nabla_{\vtheta}\phi(\vx_i;\vtheta)\cdot\nabla_{\vtheta}\phi(\vx_j;\vtheta)$ and
    $\nabla_{\vtheta}\RS = \frac{1}{N} \nabla_{\vtheta}\phi(S;\vtheta) \mA^{\T}\mA \bm{e}$
   % $\nabla_{\vtheta}\RS = \frac{1}{N}\sum_{i=1}^{N}e_i\nabla_{\vtheta}\phi(\vx_i;\vtheta)$
    , so
    \begin{equation*}
        \norm{\nabla_{\vtheta}\RS (\vtheta(t))}^2_2 = \frac{\width}{N^2}\ve^{\T} \mA^{\T}\mA \mG(\vtheta(t)) \mA^{\T}\mA \ve \geq \frac{\width}{N^2}\ve^{\T} \mA^{\T}\mA \mG^{(a)}(\vtheta(t)) \mA^{\T}\mA \ve,
        \end{equation*}
        where the last equation is true by the fact that $G^{(w)}(\vtheta(t))$ is a Gram matrix and hence positive semi-definite. Together with 
        \begin{eqnarray*}
         \frac{\width}{N^2}\ve^{\T} \mA^{\T}\mA \mG^{(a)}(\vtheta(t)) \mA^{\T}\mA \ve &\geq&  \frac{\width}{N^2}\lambda_{\min}\left(\mG^{(a)}(\vtheta(t))\right)\ve^{\T} \mA^{\T}\mA \mA^{\T}\mA \ve. \\
         &\geq & \frac{2\width}{N}\lambda_{\min}\left(\mG^{(a)}(\vtheta(t))\right) \lambda_{\min}\left(\mA \mA^{\T}\right)   \RS(\vtheta(t)) \\
         &\geq & \frac{\width}{N}\lambda_S \lambda_{\mA} \RS(\vtheta(t)),
    \end{eqnarray*}
    then finally we get
    \begin{equation*}
        \frac{\D}{\D t}\RS(\vtheta(t)) = - \norm{\nabla_{\vtheta}\RS(\vtheta(t))}^2_2 \leq - \frac{\width}{N}\lambda_S \lambda_{\mA}  \RS(\vtheta(t)).
    \end{equation*}
    Integrating the above equation yields the conclusion in this lemma.
\end{proof}

The following lemma shows that the parameters in the two-layer neural network are uniformly bounded in time during the training before time $t^*$ as defined in \eqref{eqn:ts}.%-\eqref{eqn:Mtheta}.

\begin{lem}\label{prop:a_w}
    For any $\delta\in(0,1)$, if \[\width\geq\max\left\{ \frac{32 N^4C_n}{\lambda_S^2\delta}, \frac{5\max\{n,r\} N\sqrt{2\lambda_{\mA,N}\RS(\vtheta^0)}}{\lambda_S \lambda_{\mA}} \left(2n\sqrt{2\log\frac{4\width(n+1)}{\delta}}\right)^{r-1} \right\},\] then with probability at least $1-\delta$ over the random initialization in \eqref{eqn:init}, for any $t\in[0, t^\ast)$ and any $k\in [\width]$,
    \begin{equation*}
        \begin{aligned}
             & \abs{a_k(t) - a_k(0)} \leq q,                    & \norm{\vw_k(t) - \vw_k(0)}_\infty \leq q,\\
             & \abs{a_k(0)} \leq \gamma\eta, \quad 
             & \norm{\vw_k(0)}_\infty\leq \eta,
        \end{aligned}
    \end{equation*}
    where
    \begin{equation*}
        q := \left(2 n\sqrt{2\log\frac{4\width(n+1)}{\delta}} \right)^r\frac{2\max\{n,r\} N\sqrt{2\lambda_{\mA,N}\RS(\vtheta^0)}}{n\width\lambda_S \lambda_{\mA}}  
    \end{equation*}
    and
    \[
    \eta:=\sqrt{2\log\frac{4\width(n+1)}{\delta}}.
    \]
\end{lem}
\begin{proof}
    Let $\xi(t) = \max\limits_{k\in[\width],s\in[0,t]}\{\abs{a_k(s)},\norm{\vw_k(s)}_\infty\}$. Note that
    \begin{eqnarray*}
        \abs{\nabla_{a_k}\RS{(\vtheta)}}^2
         & =& \left[\frac{1}{N}\sum_{i=1}^N (e^{\T}\mA^{\T}\mA)_i \sigma(\vw_k^\T\vx_i)\right]^2 \\
         &\leq & \norm{\vw_k}_1^{2r} \left[\frac{1}{N}\sum_{i=1}^N (|e^{\T}\mA^{\T}\mA|)_i \right]^2\\
         &\leq &2\norm{\vw_k}_1^{2r} \lambda_{\mA,N} \RS(\vtheta) \\
         &\leq &2 n^{2r} (\xi(t))^{2r} \lambda_{\mA,N} \RS(\vtheta),
    \end{eqnarray*}
    { where $\lambda_{\mA,N}$ denotes the largest eigenvalue of $\mA$,}
    and
    \begin{eqnarray*}
        \norm{\nabla_{\vw_k}\RS{(\vtheta)}}^2_\infty
       & =   &  \Big\lVert\frac{1}{N}\sum_{i=1}^N (e^{\T}\mA^{\T}\mA)_i a_k\sigma'(\vw_k^\T\vx_i)\vx_i\Big\rVert_\infty^2 \\
        & \leq   &  \abs{a_k}^2 r^2\norm{\vw_k}_1^{2(r-1)} \Big\lVert\frac{1}{N}\sum_{i=1}^N (e^{\T}\mA^{\T}\mA)_i \vx_i\Big\rVert_\infty^2 \\
        & \leq   &  \abs{a_k}^2 r^2\norm{\vw_k}_1^{2(r-1)} \left|\frac{1}{N}\sum_{i=1}^N (|e^{\T}\mA^{\T}\mA|)_i \right|^2 \\
       &\leq& 2 \abs{a_k}^2  r^2\norm{\vw_k}_1^{2(r-1)} \lambda_{\mA,N}\RS(\vtheta) \\
       &\leq & 2r^2n^{2(r-1)}(\xi(t))^{2r} \lambda_{\mA,N} \RS(\vtheta).
    \end{eqnarray*}
    From Lemma~\ref{lem:exp_RS}, if $\width\geq \frac{32 N^4C_n}{\lambda_S^2\delta}$, then with probability at least $1 - \delta/2$ over random initialization, {one can represent \eqref{componentwisegradflow}} in an integral form and obtain,
    \begin{equation*}
        \begin{aligned}
            \abs{a_k(t) - a_k(0)}
             & \leq \int_0^t\abs{\nabla_{a_k}\RS(\vtheta(s))}\diff{s}                                                                 \\
             & \leq \sqrt{2 \lambda_{\mA,N}}n^r (\xi(t))^r \int_{0}^{t} \sqrt{\RS(\vtheta(s))}\diff{s}                                        \\
             & \leq \sqrt{2 \lambda_{\mA,N}}n^r (\xi(t))^r\int_{0}^{t}\sqrt{\RS(\vtheta^0)}\exp\left(-\frac{\width\lambda_S \lambda_{\mA} s}{2N}\right)\diff{s} \\
             & \leq \frac{2\sqrt{2 \lambda_{\mA,N} }n^r N\sqrt{\RS(\vtheta^0)}}{\width\lambda_S \lambda_{\mA}} (\xi(t))^r                                         \\
             & \leq p  (\xi(t))^r,
        \end{aligned}
    \end{equation*}
    where $p := \frac{2\sqrt{2\lambda_{\mA,N}}\max\{n,r\}n^{r-1}  N\sqrt{\RS(\vtheta^0)}}{\width\lambda_S \lambda_{\mA}}$.  %$p := \frac{\blue{5}{\color{red}2}\sqrt{2}n  N\sqrt{\RS(\vtheta^0)}}{\width\lambda_S}$ Similarly,
    \begin{equation*}
        \begin{aligned}
            \norm{\vw_k(t) - \vw_k(0)}_\infty
             & \leq \int_{0}^{t} \norm{\nabla_{\vw_k}\RS(\vtheta(s))}_\infty\diff{s}                                                                      \\
             & \leq \sqrt{2\lambda_{\mA,N}} r n^{r-1}(\xi(t))^{r} \int_{0}^t \sqrt{\RS(\vtheta(s))} \diff{s}                                           \\
             & \leq \sqrt{2\lambda_{\mA,N}} r n^{r-1}(\xi(t))^{r} \int_{0}^{t} \sqrt{\RS(\vtheta^0)}\exp\left(-\frac{\width\lambda_S \lambda_{\mA} s}{2N}\right)\diff{s} \\
             & \leq \frac{2 \sqrt{2\lambda_{\mA,N}} N r n^{r-1}\sqrt{\RS(\vtheta^0)}}{\width\lambda_S \lambda_{\mA} } (\xi(t))^{r}                                            \\
             & \leq p (\xi(t))^{r}.
        \end{aligned}
    \end{equation*}
    We should point out that the above inequalities hold for all $t\in(0,t^*)$ since the upper bounds are based on Lemma~\ref{lem:exp_RS}, thus
    \begin{equation}\label{eqn:xib}
        \xi(t) \leq \xi(0) + p(\xi(t))^{r},
    \end{equation}
    for all $t\in (0,t^*)$.
    From Lemma~\ref{lem1} with probability at least $1 - \delta/2$,
    \begin{align}\label{eqn:xi0}
        \xi(0)=\max_{k\in[\width]}\{\abs{a_k(0)},\norm{\vw_k(0)}_\infty\}
        &\leq \max\left\{\gamma\sqrt{2\log\frac{4\width(n+1)}{\delta}},\sqrt{2\log\frac{4\width(n+1)}{\delta}}\right\}
        \leq \sqrt{2\log\frac{4\width(n+1)}{\delta}}=\eta.
    \end{align}
    Since
    \begin{equation*}
        \width
        \geq  \frac{5\max\{n,r\} N\sqrt{2\lambda_{\mA,N}\RS(\vtheta^0)}}{\lambda_S \lambda_{\mA}} \left(2 n\sqrt{2\log\frac{4\width(n+1)}{\delta}}\right)^{r-1}= \frac{5}{2}\width p (2\eta)^{r-1},
    \end{equation*}
    then $p(2\eta)^{r-1}\leq \frac{2}{5}$. 
    Let
    \begin{equation*}
        t_0 := \inf\{t \mid \xi(t) > 2\eta\}.
    \end{equation*}
    Then $t_0$ is the first time for the magnitude of a NN parameter exceeding $2\eta$. Recall that $t^*$, introduced in \eqref{eqn:ts}, denotes the first time for $\norm{\mG^{(a)}(\vtheta) - \mG^{(a)}(\vtheta^0)}_{\mathrm{F}}> \frac{1}{4}\lambda_S$. We will prove $t_0\geq t^*$, i.e., we show that, as long as the kernel $\mG^{(a)}(\vtheta(t))$ introduced by the gradient descent method is well controlled around its initialization $\mG^{(a)}(\vtheta(0))$, all network parameters have a well-controlled magnitude in the sense that there is no parameter with a magnitude larger than $2\eta$. We will prove $t_0\geq t^*$ by contradiction. Suppose that $t_0 < t^*$. For $t\in[0, t_0)$, by \eqref{eqn:xib}, \eqref{eqn:xi0}, and $\xi(t)\leq 2\eta$, we have 
    \begin{equation*}
        \xi(t) \leq \eta + p(2\eta)^{r-1}\xi(t)
        \leq \eta + \frac{2}{5}\xi(t),
    \end{equation*}
    then
    \begin{equation*}
        \xi(t) \leq \frac{5}{3}\eta.
    \end{equation*}
    After letting $t\to t_0$, the inequality just above contradicts with the definition of $t_0$. So $t_0 \geq  t^*$ and then $\xi(t) \leq 2\eta$ for all $t \in[0,t^*)$. Thus
    \begin{equation*}
        \begin{aligned}
            \abs{a_k(t) - a_k(0)}        & \leq (2 \eta)^r p            \\
            \norm{\vw_k(t) - \vw_k(0)}_\infty & \leq (2\eta)^r p.
        \end{aligned}
    \end{equation*}
    Finally, notice that
    \begin{equation}\label{eqn:q}
        \begin{aligned}
            (2\eta)^r p
              = \left(2 n\sqrt{2\log\frac{4\width(n+1)}{\delta}} \right)^r\frac{2\max\{n,r\} N\sqrt{2\lambda_{\mA,N}\RS(\vtheta^0)}}{n\width\lambda_S \lambda_{\mA}}     
              = q,
        \end{aligned}
    \end{equation}
    which ends the proof.
\end{proof}

\subsection{Proof of Theorem \ref{thm:lcr}}
Now we are ready to prove Theorem \ref{thm:lcr}.

\begin{proof}[Proof of Theorem \ref{thm:lcr}.]
    From Lemma~\ref{lem:exp_RS}, it is sufficient to prove that the stopping time $t^*$ in Lemma~\ref{lem:exp_RS} is equal to $+\infty$. We will prove this by contradiction.
    
    Suppose $t^* < +\infty$. Note that
    \begin{equation}\label{eqn:lcr1}
        \abs{\mG_{ij}^{(a)}(\vtheta(t^*)) - \mG_{ij}^{(a)}(\vtheta(0))} \leq \frac{1}{\width}\sum_{k=1}^\width \abs{g^{(a)}(\vw_k(t^*);\vx_i,\vx_j) - g^{(a)}(\vw_k(0);\vx_i,\vx_j)}.
    \end{equation}
    By the mean value theorem,
    \begin{eqnarray*}
        \abs{g^{(a)}(\vw_k(t^*);\vx_i,\vx_j) - g^{(a)}(\vw_k(0);\vx_i,\vx_j)} \leq  \norm{\nabla_{\vw} g^{(a)}\left(c\vw_k(t^*) + (1-c)\vw_k(0);\vx_i,\vx_j\right)}_\infty\norm{\vw_k(t^*) - \vw_k(0)}_1
%         &\leq \norm{\nabla g\left(c\vw_k(t^*) + (1-c)\vw_k(0);\vx_i,\vx_j\right)}_1\norm{\vw_k(t^*) - \vw_k(0)}_1
    \end{eqnarray*}
    for some $c\in (0, 1)$. Further computation yields
    \begin{equation*}
        \begin{aligned}
            \nabla_{\vw} g^{(a)}(\vw;\vx_i,\vx_j) =
           \Big[\sigma'(\vw^\T\vx_i)\vx_i\Big]   \times \Big[\sigma(\vw^\T\vx_j)\Big]  + \Big[\sigma'(\vw^\T\vx_j)\vx_j\Big]\times \Big[\sigma(\vw^\T\vx_i)\Big]
        \end{aligned}
    \end{equation*}
    for all $\bm{w}$. Hence, it holds for all $\bm{w}$ that $\norm{\nabla_{\vw} g^{(a)}(\vw;\vx_i,\vx_j)}_\infty
              \leq 2r\norm{\vw}_1^{2r-1}$. Therefore, the bound in \eqref{eqn:lcr1} becomes
    \begin{equation}\label{eqn:lcr2}
        \abs{\mG_{ij}^{(a)}(\vtheta(t^*)) - \mG_{ij}^{(a)}(\vtheta(0))} \leq \frac{2r}{\width}\sum_{k=1}^\width \norm{c\vw_k(t^*) + (1-c)\vw_k(0)}_1^{2r-1} \norm{\vw_k(t^*) - \vw_k(0)}_1.
    \end{equation}
    By Lemma~\ref{prop:a_w},
    \begin{equation*}
        \norm{c\vw_k(t^*) + (1-c)\vw_k(0)}_1 \leq \norm{\vw_k(0)}_1 + \norm{\vw_k(t^*) - \vw_k(0)}_1 \leq n(\eta + q) \leq 2n\eta,
    \end{equation*}
    where $\eta$ and $q$ are defined in Lemma~\ref{prop:a_w}. So, \eqref{eqn:lcr2} and the above inequalities indicate
    \begin{equation*}
        \abs{\mG_{ij}^{(a)}(\vtheta(t^*)) - \mG_{ij}^{(a)}(\vtheta(0))} \leq r(2n)^{2r}\eta^{2r-1}q,
    \end{equation*}
    and
        \begin{eqnarray*}
            \norm{\mG^{(a)}(\vtheta(t^*)) - \mG^{(a)}(\vtheta(0))} _{\mathrm{F}}
             & \leq &nr(2n)^{2r}\eta^{2r-1}q    \\
             & = &nr(2n)^{2r}(2\log\frac{4\width(n+1)}{\delta})^{(3r-1)/2} \left(2 {n
             }\right)^r\frac{2\max\{n,r\} N\sqrt{2\lambda_{\mA,N}\RS(\vtheta^0)}}{n\width\lambda_S \lambda_{\mA}}    \\
             &=&2^{9r/2+1} n^{3r-1} N^2 r\left(\log\frac{4\width(n+1)}{\delta}\right)^{(3r-1)/2} \frac{\max\{n,r\} \sqrt{\lambda_{\mA,N}\RS(\vtheta^0)}}{\width\lambda_S \lambda_{\mA}} \\
             & \leq &\frac{1}{4}\lambda_S,
        \end{eqnarray*}
    if we choose
    \begin{equation*}
        \width\geq 2^{9r/2+3} n^{3r-1} N^2 r\left(\log\frac{4\width(n+1)}{\delta}\right)^{(3r-1)/2} \frac{\max\{n,r\} \sqrt{\lambda_{\mA,N}\RS(\vtheta^0)}}{\lambda_S^2 \lambda_{\mA}}.
    \end{equation*}
    The fact that $\norm{\mG^{(a)}(\vtheta(t^*)) - \mG^{(a)}(\vtheta(0))} _{\mathrm{F}}\leq \frac{1}{4}\lambda_S$ above contradicts with the definition of $t^*$ in \eqref{eqn:ts}.

    Let us summarize the conclusion in the above discussion. Let $C_n := \Exp\norm{\vw}_1^{4r}<+\infty$ with $\vw\sim \mathcal{N}(\vzero,\mI_n)$. The largest eigenvalue and the condition number of $\mA \mA^{\T}$ are denoted as $\lambda_{\mA,N}$ and $\kappa_A$, respectively. For any $\delta\in(0,1)$ , define 
\[
m_1= \frac{32 N^4C_n}{\lambda_S^2\delta},
\]
\[
m_2=  \frac{5\max\{n,r\} N\sqrt{2\lambda_{\mA,N}\RS(\vtheta^0)}}{\lambda_S \lambda_{\mA}} \left(2n\sqrt{2\log\frac{4\width(n+1)}{\delta}}\right)^{r-1},
\]
and
\begin{equation}\label{eqn:mmm}
m_3=2^{9r/2+3} n^{3r-1} N^2 r\left(\log\frac{4\width(n+1)}{\delta}\right)^{(3r-1)/2} \frac{\max\{n,r\} \sqrt{\lambda_{\mA,N}\RS(\vtheta^0)}}{\lambda_S^2 \lambda_{\mA}}.
\end{equation}
Then when $m\geq\max\{m_1,m_2,m_3\}$, with probability at least $1-\delta$ over the random initialization $\vtheta^0$, we have, for all $t\geq 0$,
    \begin{equation*}
        \RS(\vtheta(t))\leq\exp\left(-\frac{\width\lambda_S \lambda_{\mA}t}{N}\right)\RS(\vtheta^0).
    \end{equation*}
Note that $\mathcal{O}(\kappa_A\text{poly}(N,r,n,\frac{1}{\delta},\frac{1}{\lambda_S}))\geq\max\{m_1,m_2,m_3\}$. Hence, we have completed the proof.
\end{proof}

\section{More details on the numerical experiments}\label{appendix_d}

In this Appendix, we report the detailed hyper-parameter setting for the six examples presented in the main text. %We also report the numerical values of the errors (up to four decimals) corresponding to the result in Figures~\ref{fig:2derror}-\ref{fig:2derrorbc} in the main text.
For convenience, we also report the notations in Table~\ref{tab:notation}. 

% \textbf{Devices and environments.} The experiments of DM are conducted on a workstation with 32$\times$ Intel(R) Xeon(R) CPU E5-2667 v4 @ 3.20GHz and 1 TB RAM and Matlab R2019a. We conduct the NN-solver on a workstation with 16$\times$ Intel(R) Xeon(R) Gold 5122 CPU @ 3.60GHz and 93G RAM using Pytorch 1.0 and 1$\times$ Tesla V100. 

% \textbf{Training details. }We use a 3-hidden-layer FNN with 
% the same width $m$ per hidden layer and 
% the smooth Polynomial-Sine activation function in \cite{liang2021reproducing}. Polynomial-Sine is defined as 
% $
%     \alpha_1 \sin(\beta_1 x)+ \alpha_2 x + \alpha_3 x^2,
% $
% where $\beta_1$, $\alpha_i, i=1,2,3$ are trainable parameters initialized by normal distribution $\mathcal{N}(1,0.01), \mathcal{N}(1,0.01), \mathcal{N}(0,0.01), \mathcal{N}(0,0.01)$ respectively. In Adam, we use an initial learning rate of 0.01 for $T$ iterations. The learning rate follows cosine decay with the increasing training iterations, i.e., the learning rate decays by multiplying a factor $0.5( cos(\frac{\pi t}{T}) + 1)$, where $t$ is the current iteration.

\begin{table}[H]
  \centering
    \begin{adjustbox}{width=0.8\textwidth,center}
    \begin{tabular}{ll}
    \toprule
    \textbf{Notation} & \textbf{Explanation} \\
    \midrule
    $k$     & $k$-nearest neighbors in DM \\
    $k_2$     & variable bandwidth parameter in VBDM \\
    $\epsilon$ & the bandwidth parameter for local integral in DM \\
    $m$     & the width of hidden layer in FNN \\
    $T$     & the number of iterations for training a FNN\\
    $N$     & the number of training points \\
    $N_b$    & the number of training points on the boundary \\
    $\gamma$ & the regularization coefficient for $\frac{1}{2}\|\bm{\phi}_{\bm{\theta}}\|_2^2$ introduced in Section~\ref{NNsolver}\\
    $\lambda$ & the penalty coefficient to enforce the boundary condition in the loss function~\eqref{eqn:DMNN2}\\
    \bottomrule
    \end{tabular}%
  \end{adjustbox}
  \caption{The summary of hyperparameter notations in the algorithms.}
  \label{tab:notation}%
\end{table}%

\begin{table}[H]
  \centering
    \begin{tabular}{clccccc}
    \toprule
          & $N$ & 625	& 1,225	& 2,500	& 5,041	& 10,000 \\
    \midrule
    \multirow{2}[2]{*}{VBDM} & $\epsilon$ & 0.0146  & 0.0078  &  0.0068  &  0.0036 & 0.0026 \\
          & $k$ & 100   & 100   & 200   & 200 & 300 \\
          & $k_2$ & 30   & 30   & 60   & 600 & 90 \\
    \midrule
    \multirow{2}[2]{*}{NN} & $T$ & 2k   & 3k   & 4k   & 6k & 8k\\
          & $m$ & 50   & 71   & 100   & 141 & 200 \\
          & $\gamma$ & 0.001   & 0.001   & 0.001   & 0.001 & 0.001 \\ 
    \bottomrule
    \end{tabular}%
    \caption{The hyperparameter setting for {\bf Example 1}: 2D Sphere.}
  \label{tab:2dsphereconfigs}%
\end{table}%

\begin{table}[H]
  \centering
  \begin{adjustbox}{width=0.75\textwidth,center}
    \begin{tabular}{clcccccccc}
    \toprule
          & $N$ & 625   & 1225  & 2500  & 5041  & 10000 & 19881 & 40000 & 80089 \\
    \midrule
    \multirow{2}[2]{*}{DM} & $\epsilon$ & 0.1166 & 0.0508 & 0.0237 & 0.0118 & 0.0059 & 0.0029 & 0.0014 & 0.0008 \\
          & $k$ & 128   & 128   & 128   & 256   & 256   & 256   & 512   & 768 \\
    \midrule
    \multicolumn{1}{c}{\multirow{3}[2]{*}{NN}} & $T$ & 2000 & 3000 & 4000 & 4000 & 4000 & 4000 & 4000 & 8000 \\
          & $m$ & 50    & 71    & 100   & 141   & 200   & 282   & 400   & 583 \\
          & $\gamma$ & 0.001 & 0.001 & 0.001 & 0.001 & 0.001 & 0.002 & 0.005 & 0.01 \\
    \bottomrule
    \end{tabular}%
	\end{adjustbox}
  \caption{The hyperparameter setting for {\bf Example 2}: 2D torus embedded in $\mathbb{R}^3$.}
  \label{tab:2dconfigs}%
\end{table}%

\comment{
\begin{table}[htbp]
  \centering
  \begin{adjustbox}{width=\textwidth,center}
    \begin{tabular}{cp{9.165em}cccccccc}
    \toprule
          & \multicolumn{1}{l}{$N$} & 625   & 1225  & 2500  & 5041  & 10000 & 19881 & 40000 & 80089 \\
    \midrule
    \multirow{2}[2]{*}{DM} & \multicolumn{1}{l}{forward error} & 3.7837  & 1.8985  & 1.0231  & 0.4871  & 0.2519  & 0.1251  & 0.0630  & 0.0326  \\
          & \multicolumn{1}{l}{inverse error} & 0.7606  & 0.3175  & 0.1511  & 0.0680  & 0.0336  & 0.0159  & N/A   & N/A \\
    \midrule
    \multirow{2}[2]{*}{NN} & \multicolumn{1}{l}{training error} & 0.7563    & 0.3183   & 0.1508   & 0.0680   & 0.0336   & 0.0158   & 0.0082   & 0.0036   \\
          & testing error & 0.7764  & 0.3213   & 0.1516   & 0.0682   & 0.0336   & 0.0159   & 0.0082   & 0.0036  \\
    \bottomrule
    \end{tabular}%
	\end{adjustbox}
  \caption{The errors for DM and NN 
  corresponding to {\bf Example~1}: 2D torus embedded in $\mathbb{R}^3$. N/A indicates that the result is not computable. }
  \label{tab:2derror}%
\end{table}%
}

\begin{table}[H]
  \centering
  \begin{adjustbox}{width=0.5\textwidth,center}
    \begin{tabular}{clccccc}
    \toprule
          & $N$ & 1024 &	2025	&4096	&16384	&32400 \\
    \midrule
    \multirow{2}[2]{*}{DM} & $\epsilon$ & 0.0221 &	0.0096 &	0.0048 &	0.0013 &	0.00064  \\
          & $k$ & 128&	128	&128	&256	&256 \\
    \midrule
    \multirow{3}[2]{*}{NN} & $T$ & 2000   & 3000   & 4000   & 10000   & 12000 \\
          & $m$ & 100&	100&	150&	250&	250 \\
          & $\lambda$ & 5.0   & 5.0   & 5.0   & 5.0   & 5.0  \\
    \bottomrule
    \end{tabular}%
	\end{adjustbox}
  \caption{The hyperparameter setting for {\bf Example 3}: 2D semi-torus with the Dirichlet condition.}
  \label{tab:2dconfigsbc}%
\end{table}%

\begin{table}[H]
  \centering
  \begin{adjustbox}{width=0.48\textwidth,center}
   \begin{tabular}{clccccc}
    \toprule
          & $N$ & 512   & 1331  & 4096  & 12167 & 24389 \\
    \midrule
    \multirow{2}[2]{*}{DM} & $\epsilon$ & 0.43  & 0.23  & 0.12  & 0.073 & 0.051 \\
          & $k$ & 128   & 128   & 256   & 256   & 256 \\
    \midrule
    \multirow{3}[2]{*}{NN} & $T$ & 1000   & 2000   & 2000   & 3000   & 3000 \\
          & $m$ & 100   & 150   & 250   & 400   & 500 \\
          & $\gamma$ & 0.001 & 0.001 & 0.005  & 0.005 & 0.005 \\
    \bottomrule
    \end{tabular}%
	\end{adjustbox}
  \caption{The hyperparameter setting for {\bf Example 4}: 3D manifold embedded in $\mathbb{R}^{12}$. }
  \label{tab:3dconfigs}%
\end{table}%

\comment{
\begin{table}[htbp]
  \centering
  \begin{adjustbox}{width=0.7\textwidth,center}
    \begin{tabular}{clccccc}
    \toprule
          & $N$ & 512   & 1331  & 4096  & 12167 & 24389 \\
    \midrule
    \multirow{2}[2]{*}{DM} & forward error & 0.4241 & 0.1498 & 0.0403 & 0.0109 & 0.0039 \\
          & inverse error & 0.2614 & 0.1113 & 0.0347 & 0.0092 & 0.0033 \\
    \midrule
    \multirow{2}[2]{*}{NN} & training error & 0.2665  & 0.1148  & 0.0297 & 0.0066  & 0.0023  \\
          & testing error & 0.2715 & 0.1346  & 0.0302  & 0.0069  & 0.0024  \\
    \bottomrule
    \end{tabular}%
	\end{adjustbox}
  \caption{The errors for DM and NN corresponding to {\bf Example~2}: 3D manifold embedded in $\mathbb{R}^{12}$. }
  \label{tab:3derror}%
\end{table}%
}

\comment{
\begin{table}
  \centering
  \begin{adjustbox}{width=0.7\textwidth,center}
    \begin{tabular}{clccccc}
    \toprule
          & $N$ & 1024  & 2025  & 4096  & 16384 & 32400 \\
    \midrule
    \multirow{2}[2]{*}{DM} & forward error & 6.7976 & 6.8733 & 7.0843 & 8.1603 & 8.9551 \\
          & inverse error & 0.0627 & 0.0286 & 0.0149 & 0.0047 & 0.0026 \\
    \midrule
    \multirow{2}[2]{*}{NN} & training error & 0.0613 &	0.0285 &	0.0150 &	0.0048 	& 0.0029  \\
          & testing error & 0.0634 &	0.0293 &	0.0154 &	0.0048 &	0.0029  \\
    \bottomrule
    \end{tabular}%
	\end{adjustbox}
  \caption{The errors for DM and NN corresponding to {\bf Example~3}: 2D semi-torus with Dirichlet condition.}
  \label{tab:2derrorbc}%
\end{table}%
}
% Table generated by Excel2LaTeX from sheet 'Sheet1'
\begin{table}[H]
  \centering
    \begin{tabular}{clcccc}
    \toprule
          & $N$ & 4326  & 8650  & 17299 & 34594 \\
    \midrule
    \multirow{2}[2]{*}{VBDM} & $\epsilon$ & 0.0048  & 0.0024  & 0.0012  & 0.0006  \\
          & $k$ & 100   & 100   & 100   & 100 \\
          & $k_2$ & 30   & 30   & 30   & 30 \\
    \midrule
    \multirow{3}[2]{*}{NN} & $T$ & 30k   & 40k   & 50k   & 100k \\
          & $m$ & 100   & 150   & 150   & 150 \\
    \bottomrule
    \end{tabular}%
    \caption{The hyperparameter setting for {\bf Example 5}: Bunny.}
  \label{tab:bunnyparams}%
\end{table}%

\begin{table}[H]
  \centering
    \begin{tabular}{clcccc}
    \toprule
          & $N$ & 2185  & 4340  & 8261  & 17157 \\
    \midrule
    \multirow{2}[2]{*}{VBDM} & $\epsilon$ & 0.000793 & 0.000370 & 0.000185 & 0.000081 \\
          & $k$ & 60    & 60    & 60    & 60 \\
           & $k_2$ & 20    & 20    & 20    & 20 \\
    \midrule
    \multirow{4}[2]{*}{NN} & $T$ & 100k  & 100k  & 100k  & 200k \\
          & $m$ & 100   & 100   & 200   & 500 \\
          & Activation & ReLU  & ReLU  & ReLU  & Tanh \\
          & Layer & 4     & 4     & 6     & 6 \\
          & $\lambda$ & 1     & 1     & 10     & 10000 \\
    \bottomrule
    \end{tabular}%
     \caption{The hyperparameter setting for {\bf Example 6}: Face with the Dirichlet condition.}
  \label{tab:faceparams}%
\end{table}%

\bibliographystyle{plain}  
\bibliography{ref,refhz}

\end{document}